\documentclass[a4paper,reqno,12pt]{amsart}
\usepackage[T1]{fontenc}
\usepackage[utf8]{inputenc}
\usepackage[english]{babel}
\usepackage{amsmath,amssymb,graphicx,psfrag}
\usepackage{fullpage}
\usepackage{enumitem}
\usepackage{siunitx}
\usepackage{pgfplots}
\pgfplotsset{
every axis/.append style={
axis x line=box, axis y line=box,
tick label style={font=\tiny},			
font={\fontsize{8pt}{12pt}\selectfont}, 	
title style={font=\tiny,yshift=-1.0ex},		
xlabel style={font=\tiny,yshift=+1.0ex}		
}
}
\usepackage{tikz}
\usetikzlibrary{shapes,snakes}		
\usepackage{stmaryrd}
\usepackage{subcaption}
\usepackage{boxedminipage}
\usepackage{mdframed}
\usepackage{bm}				
\usepackage{pgffor}			
\usepackage{booktabs}			
\usepackage[margin=0pt]{caption}	
\usepackage{booktabs}			
\usepackage{hyperref}
\usepackage{url}			

\DeclareMathOperator*{\Refine}{\mathsf{Refine}}
\DeclareMathOperator*{\Doerfler}{D\text{\"o}rfler}

\newcommand{\ff}{\boldsymbol{f}}					
\renewcommand{\gg}{\boldsymbol{g}}					
\newcommand{\y}{\mathbf{y}}						
\newcommand{\N}{\mathbb{N}}						
\newcommand{\R}{\mathbb{R}}						
\newcommand{\E}{\mathbb{E}}						
\newcommand{\EE}{\mathcal{E}}
\newcommand{\BB}{\mathcal{B}}						
\newcommand{\NN}{\mathcal{N}} 						
\newcommand{\G}{\mathcal{G}} 						
\newcommand{\TT}{\mathcal{T}} 						
\newcommand{\MM}{\mathcal{M}}						
\newcommand{\RR}{\mathcal{R}}						
\renewcommand{\SS}{\mathcal{S}}						
\newcommand{\set}[3][\big]{#1\{#2\,:\,#3#1\}}		
\newcommand{\hull}[1]{{\rm span}\left(#1\right)}	
\newcommand{\reff}[2]{\stackrel{\eqref{#1}}{#2}}	
\newcommand{\Cov}{{\rm Cov}}    					
\newcommand{\supp}{{\rm supp}}    					
\newcommand{\Cloc}{C_{\rm loc}} 					
\newcommand{\Cthm}{C_{\rm thm}} 					
\newcommand{\Cstab}{C_{\rm stb}} 					
\newcommand{\CY}{C_Y} 					            
\newcommand{\azeromin}{a_0^{\rm min}}				
\newcommand{\azeromax}{a_0^{\rm max}}				
\newcommand{\dpiy}{d\pi(\y)}						
\newcommand{\gotI}{\mathfrak{I}} 					
\newcommand{\gotP}{\mathfrak{P}} 					
\newcommand{\gotQ}{\mathfrak{Q}} 					
\newcommand{\gotM}{\mathfrak{M}} 					
\newcommand{\VXP}{V_{X\gotP}}						
\newcommand{\VYP}{V_{Y\gotP}}						
\newcommand{\VXQ}{V_{X\gotQ}}						
\newcommand{\VXPhat}{\widehat{V}_{X\gotP}}			
\newcommand{\uXP}{u_{X\gotP}}						
\newcommand{\vXP}{v_{X\gotP}}						
\newcommand{\zXP}{z_{X\gotP}}						
\newcommand{\uXPhat}{\widehat{u}_{X\gotP}}			
\newcommand{\Xhat}{{\widehat{X}}}					
\newcommand{\eXPhat}{\widehat{e}_{X\gotP}}			
\newcommand{\VXhatP}{V_{\widehat{X}\gotP}}			
\newcommand{\vXhatP}{v_{\widehat{X}\gotP}}			
\newcommand{\thetaX}{\theta_X} 						
\newcommand{\thetaP}{\theta_{\mathfrak{P}}} 		
\newcommand{\eXQnu}{e_{X\gotQ}^{(\nu)}}				
\newcommand{\muXP}{\mu_{X\gotP}}					
\newcommand{\zetaXP}{\zeta_{X\gotP}}				
\newcommand{\uref}{u_{\text{ref}}}					
\newcommand{\Ql}{\gotQ_\ell}						
\newcommand{\ul}{u_\ell}							
\newcommand{\zl}{z_\ell}							
\newcommand{\uhatl}{\widetilde u_\ell}				
\newcommand{\zhatl}{\widetilde z_\ell}				
\newcommand{\eps}[1]{\varepsilon^{(#1)}}  			
\def\vertexY{\xi}
\def\basisY{\BB}
\def\KY{K}

\newcommand{\enorm}[3][]{#1\vert\hspace*{-.4mm}#1\Vert \, #2 \, #1\Vert\hspace*{-.4mm}#1\vert_{#3}}
\newcommand{\bnorm}[2][]{\enorm[#1]{#2}{}}
\newcommand{\norm}[3][]{#1\lVert #2#1\rVert_{#3}}		
\DeclareRobustCommand{\rchi}{{\mathpalette\irchi\relax}}
\newcommand{\irchi}[2]{\raisebox{\depth}{$#1\chi$}} 

\def\zero#1{%
	\foreach \index in {1, ..., #1} {%
		0 %
    	\ifnum \index < #1
    		\ %
		\fi
  	}
}
\def\eps#1#2{%
  	(%
  	\foreach \index in {1,...,#2} {%
		\ifnum \index = #1 %
			1 %
		\else
			0 %
		\fi
		\ifnum \index = #2 %
			)%
  		\else
			\ %
		\fi
  	}
}

\newcommand{\fontsizetwo}{0.80em}
\newcommand{\fontsizethree}{0.7em}

\newcommand{\smallfontthree}[1]{ {\fontsize{\fontsizethree}{\fontsizetwo}\selectfont{#1}} }

\newtheorem{lemma}{Lemma}[section]
\newtheorem{theorem}[lemma]{Theorem}

\newtheorem{remark}[lemma]{Remark}
\newtheorem{algorithm}[lemma]{Algorithm}

\definecolor{darkGreen}{RGB}{0,100,0}
\definecolor{myViolet}{RGB}{153,50,204}
\definecolor{myOrange}{HTML}{FF8000}
\definecolor{myBrown}{rgb}{0.6 0.4 0.2}
\definecolor{myBrownDark}{RGB}{168 114 60}
\definecolor{myYellow}{HTML}{CCCC00}
\definecolor{myCyan}{HTML}{99FFFF}
\definecolor{myGray}{HTML}{707070}

\usepackage{fancyhdr}
\advance\footskip0.4cm
\advance\textheight-0.4cm
\calclayout
\makeatletter
\def\@seccntformat#1{%
  \protect\textup{\protect\@secnumfont
    \ifnum\pdfstrcmp{subsection}{#1}=0 \bfseries\fi
    \csname the#1\endcsname
    \protect\@secnumpunct
  }%
}
\makeatother

\title{Goal-oriented error estimation and adaptivity\\
for elliptic PDEs with parametric or\\
uncertain inputs}

\author{Alex Bespalov}
\address{School of Mathematics, University of Birmingham, Edgbaston, Birmingham B15 2TT, UK}
\email{a.bespalov@bham.ac.uk}

\author{Dirk Praetorius}
\address{Institute for Analysis and Scientific Computing, TU Wien, Wiedner Hauptstra\ss{}e~8--10, 1040 Vienna, Austria}
\email{dirk.praetorius@asc.tuwien.ac.at}

\author{Leonardo Rocchi}
\address{School of Mathematics, University of Birmingham, Edgbaston, Birmingham B15 2TT, UK}
\email{lxr507@bham.ac.uk}

\author{Michele Ruggeri}
\address{Faculty of Mathematics, University of Vienna, Oskar-Morgenstern-Platz~1, 1090 Vienna, Austria}
\email{michele.ruggeri@univie.ac.at}

\thanks{{\em Acknowledgements.} This work was initiated when AB visited the Institute for Analysis
and Scientific Computing at TU Wien in 2017. AB would like to thank the Isaac Newton Institute for
Mathematical Sciences for support and hospitality during the programme ``Uncertainty quantification
for complex systems: theory and methodologies'', where part of the work on this paper was undertaken.
This work was supported by the EPSRC grant EP/K032208/1.
The work of AB and LR was supported by the EPSRC under grant EP/P013791/1.
The work of DP and MR was supported by the Austrian Science Fund (FWF) under grants W1245 and F65.
}
\date{\today}
\begin{document}
%
\begin{abstract}
We use the ideas of goal-oriented error estimation and adaptivity to design and implement
an efficient adaptive algorithm for approximating linear quantities of interest derived
from solutions to elliptic partial differential equations (PDEs) with parametric or uncertain
inputs.
In the algorithm, the stochastic Galerkin finite element method (sGFEM) is used to approximate
the solutions to primal and dual problems that depend on a countably infinite number of
uncertain parameters.
Adaptive refinement is guided by an innovative strategy that combines the error reduction
indicators computed for spatial and parametric components of the primal and dual solutions.
The key theoretical ingredient is a novel two-level \textsl{a~posteriori} estimate
of the energy error in sGFEM approximations.
We prove that this error estimate is reliable and efficient.
The effectiveness of the goal-oriented error estimation strategy and the performance of the
goal-oriented adaptive algorithm are tested numerically for three representative model
problems with parametric coefficients and for three quantities of interest (including
the approximation of pointwise values).
\end{abstract}
\maketitle
%
\section{Introduction} 
\noindent
Partial differential equations (PDEs) with parametric uncertainty are ubiquitous in mathematical models
of physical phenomena and in engineering applications.
The efficient numerical solution of such PDE problems presents a number of theoretical and practical challenges.
A particularly challenging class of problems is represented by PDEs 
whose inputs and outputs depend on infinitely many uncertain parameters.
For this class of problems, numerical algorithms are sought that are able to identify
a finite set of most important parameters to be incorporated into the basis of the approximation space,
such that the solution to the underlying PDE, or the quantity of interest derived from the solution,
can be approximated to a prescribed accuracy (engineering tolerance) with minimal computational work.

Adaptive techniques based on rigorous \textsl{a~posteriori} error analysis of computed solutions
provide an effective mechanism for building approximation spaces and accelerating convergence
of computed solutions.
These techniques rely heavily on how the approximation error is estimated and controlled.
One may choose to estimate the error in the global energy norm and use the associated error indicators
to enhance the computed solution and drive the energy error estimate to zero.
However, in practical applications, simulations often target a specific (e.g., localized) feature of the solution,
called the quantity of interest and represented using a linear functional of the solution.
In these cases, the energy norm may give very little useful information about the simulation error.

Alternative error estimation techniques, such as goal-oriented error estimations, e.g.,
by the dual-weighted residual methods, allow to control the errors in the quantity of interest.
While for deterministic PDEs, these error estimation techniques and the associated adaptive algorithms
are very well studied
(see, e.g., \cite{MR1352472, MR1322810, beckerRannacher1996, prudhommeOden1999, beckerRannacher2001, gilesSuli2002, MR1960405}
for the \textsl{a~posteriori} error estimation and \cite{ms09, bet11, hp16, fpv16} for a rigorous convergence
analysis of adaptive algorithms), relatively little work has been done for PDEs with parametric or uncertain inputs.
For example, in the framework of (intrusive) stochastic Galerkin finite element methods (sGFEMs) (see, e.g.,~\cite{gs91, lps2014}),
\textsl{a~posteriori} error estimation of linear functionals of solutions to PDEs with parametric uncertainty is addressed
in~\cite{mathelinLeMaitre2007} and, for nonlinear problems, in~\cite{butlerDawsonWildey2011}.
In particular, in~\cite{mathelinLeMaitre2007}, a rigorous estimator for the error in the quantity of interest is derived
and several adaptive refinement strategies are discussed.
However, the authors comment that the proposed estimator lacks information about the structure of the estimated error;
in particular, it does not allow to assess individual contributions from spatial and parametric discretizations
and to perform anisotropic refinement of the parametric approximation space
(see~\cite[page~113]{mathelinLeMaitre2007}).

As for nonintrusive methods, the goal-oriented error estimation techniques and the associated
adaptive algorithms are proposed
in~\cite{almeidaOden2010} for the stochastic collocation method (see, e.g., \cite{babuvskaNobileTempone2007})
and in~\cite{emn16} for the multilevel Monte Carlo sampling (see, e.g.,~\cite{Giles15}).
\textsl{A~posteriori} error estimates for quantities of interest derived from generic surrogate approximations
(either intrusive or nonintrusive) are introduced in~\cite{bryantPrudhommeWildey2015}.
These estimates provide separate error indicators for spatial and parametric discretizations.
The indicators are then used to identify dominant sources of error and guide the adaptive algorithm
for approximating the quantity of interest.
Various adaptive refinement strategies are discussed in~\cite{bryantPrudhommeWildey2015}
and tested for model PDE problems with inputs that depend on a \emph{finite} number
of uncertain parameters.

Our main aim in this paper is to design an adaptive sGFEM algorithm
for accurate approximation of moments of a quantity of interest~$Q(u)$,
which is a linear functional of the solution $u$ to the following model problem whose coefficient depends linearly
on \emph{infinitely many} parameters:
\begin{equation} \label{eq:strongform}
\begin{aligned}
   -\nabla_{x} \cdot (a(x,\y) \nabla_{x} u(x,\y)) & = f(x),  && \quad x \in D,\ \y \in \Gamma,  \\
   u(x,\y) & = 0,                                                   && \quad  x \in \partial D,\ \y \in \Gamma.
\end{aligned}
\end{equation}
Here, $D \subset {\R}^2$ is a bounded Lipschitz domain
with polygonal boundary $\partial D$,
$\Gamma := \prod_{m=1}^\infty \Gamma_m$ is the parameter domain
with $\Gamma_m$ being bounded intervals in $\R$ for all $m \in \N$,
$f \in H^{-1}(D)$, and the parametric coefficient $a = a(\cdot, \y)$ is represented as
\begin{equation}\label{eq1:a}
 a(x,\y) = a_0(x) + \sum_{m=1}^\infty y_m a_m(x),
 \quad x \in D,\ \y = (y_1, y_2, \ldots) \in \Gamma,
\end{equation}
for a family of functions $a_m(x) \in L^\infty(D)$, $m \in \N_0$, and with the series converging uniformly in $L^\infty(D)$
(an example of such a representation is the Karhunen--Lo\`eve expansion of a random field with 
given covariance function and mean $a_0(x)$; see, e.g., \cite{gs91,lps2014}).

In this work, we are particularly interested in estimating and controlling
the expected error in the quantity of interest, $\mathbb{E}[Q(u - u_N)]$,
where $u_N$ is an approximation of $u$ via sGFEM.
This enables us to use the ideas of goal-oriented adaptivity, where one is interested in controlling the error
in the goal functional $G(u) := \mathbb{E}[Q(u)]$
(rather than, e.g., in the energy norm).

\subsection{Goal-oriented error estimation in the abstract setting} \label{subsec:abstract:goaloriented:setting}

In order to motivate the design of our adaptive algorithm, let us first recall the idea of goal-oriented error estimation.
Let $V$ be a Hilbert space and denote by $V'$ its dual space.
Let $B: V \times V \to \R$ be a continuous, elliptic, and symmetric bilinear form 
with the associated energy norm $\bnorm{\cdot}$, i.e., $\bnorm{v}^2 := B(v,v)$ for all $v \in V$.
Given two continuous linear functionals $F, G \in V'$,
our aim is to approximate $G(u)$, where $u \in V$ is the unique solution
of the \emph{primal problem}:
\begin{equation*} 
      B(u,v) = F(v)
\quad \text{for all $v \in V.$}
\end{equation*}
To this end, the standard approach (see, e.g., \cite{MR1352472, beckerRannacher2001, gilesSuli2002, MR1960405})
considers $z \in V$ as the unique solution to the \emph{dual problem}:
\begin{equation*} 
      B(v,z) = G(v) \quad \text{for all $v \in V$.}
\end{equation*}
Let $V_\star$ be a finite dimensional subspace of $V$. Let 
$u_\star\in V_\star$ (resp., $z_\star\in V_\star$) be the unique Galerkin approximation
of the solution to the primal (resp., dual) problem, i.e., 
\[
     B(u_\star,v_\star) = F(v_\star) \quad
     (\text{resp., } B(v_\star,z_\star) = G(v_\star)) \qquad \text{for all $v_\star \in V_\star.$}
\]
Then, it follows that
\begin{equation} \label{eq:errorInequality}
   \lvert G(u)-G(u_\star) \rvert
   = \lvert B(u-u_\star,z) \rvert
   = \lvert B(u-u_\star,z-z_\star)\rvert
       \leq \bnorm{u-u_\star}\,\bnorm{z-z_\star},
\end{equation}
where the second equality holds due to Galerkin orthogonality. 

Assume that $\mu_\star$ and $\zeta_\star$ are reliable estimates
for the energy errors $\bnorm{u-u_\star}$ and $\bnorm{z-z_\star}$, respectively, i.e.,  
\begin{equation} \label{energy:norms:estimators}
      \bnorm{u-u_\star} \lesssim \mu_\star
      \quad\text{and}\quad
      \bnorm{z-z_\star} \lesssim \zeta_\star
\end{equation}
(hereafter, $a \lesssim b$ means the existence
of a generic positive constant $C$ such that $a \le C b$, and $a \simeq b$ abbreviates $a \lesssim b \lesssim a$).
Hence, inequality \eqref{eq:errorInequality} implies that the product $\mu_\star \, \zeta_\star$ is a reliable error
estimate for the approximation error in the goal functional:
\begin{equation} \label{goal:error:reliability}
      \lvert G(u)-G(u_\star) \rvert  
      \lesssim \mu_\star \, \zeta_\star.
\end{equation}

\subsection{Main contributions and outline of the paper} 

In view of estimate \eqref{goal:error:reliability},
a goal-oriented adaptive algorithm must drive the \emph{product} of computable
energy error estimates $\mu_\star$ and $\zeta_\star$ to zero
with the best possible rate.
Aiming to design such an algorithm for the parametric model problem~\eqref{eq:strongform},
our first step is to find appropriate \textsl{a~posteriori} error estimates $\mu_\star$ and $\zeta_\star$.

There have been several very recent works that addressed \textsl{a~posteriori} error estimation
of stochastic Galerkin approximations for parametric problems, including
explicit residual-based \textsl{a~posteriori} error estimators in~\cite{egsz14, egsz15},
local equilibration error estimators in~\cite{em16},
and hierarchical estimators in~\cite{bps14, bs16}.
In this paper, we propose a novel \textsl{a~posteriori} error estimation technique
that can be used to control the energy errors in the primal and dual Galerkin approximations (see~\eqref{energy:norms:estimators}).
Similarly to the aforementioned works, we exploit the tensor product structure of the approximation space to separate the error contributions due to spatial approximations from
the ones that are due to parametric approximations.
Then, building on the hierarchical framework developed in~\cite{bps14, bs16}
and using the ideas from~\cite{MR1621301,ms99} (see also~\cite{dly89,MR1393909} for earlier works in this direction), we define a new two-level \textsl{a~posteriori} estimate of the energy error
and prove that it is reliable and efficient.
One of the key advantages of this new estimator is that
it avoids the solution of linear systems when estimating the errors coming from spatial approximations  
(while keeping the hierarchical structure of the estimator) and thus speeds up the computation.

The goal-oriented adaptive algorithm developed in this paper
draws information from the error estimates $\mu_\star,\,\zeta_\star$ and performs a balanced adaptive refinement
of spatial and parametric components of the finite-dimensional space  $V_\star$
to reduce the error in approximating the goal functional $G(u)$.
Specifically, the marking is performed by employing and extending the strategy proposed in~\cite{fpv16},
and the refinements are driven by the estimates of reduction in the product of energy errors
$\bnorm{u-u_\star} \, \bnorm{z-z_\star}$
(that provides an upper bound for the error in the goal functional, see~\eqref{eq:errorInequality}).

Finally, we use three representative examples of parametric PDEs posed over square, L-shaped, and slit domains
as well as different quantities of interest to demonstrate the performance and effectiveness
of our goal-oriented adaptive strategy.

The rest of the paper is organised as follows.
In Section~\ref{sec:weakdiscrete} we set the parametric model problem \eqref{eq:strongform} in weak form and introduce the sGFEM discretization for this problem.
Section~\ref{sec:errorestimators} concerns \textsl{a~posteriori} error estimation
in the energy norm for the model problem~\eqref{eq:strongform}.
First, we review a hierarchical \textsl{a~posteriori} error estimation strategy in~\S\ref{subsec:errorestimators:hierarchical}.
Then, in~\S\ref{sec:two:level:errorestimator},
a new two-level \textsl{a~posteriori} error estimate is introduced and proved to be reliable and efficient.
The goal-oriented adaptive sGFEM algorithm employing two-level error estimates
is presented in Section~\ref{sec:goaloriented:adaptivity},
and the results of numerical experiments are reported in Section~\ref{sec:numerical:examples}.
Finally, in Section~\ref{sec:remarks}, we summarize the results of the paper and discuss some possible extensions.

\section{Parametric model problem and its discretization} \label{sec:weakdiscrete}

\subsection{Weak formulation}

Consider the parametric model problem \eqref{eq:strongform} with the coefficient $a = a(x,\y)$ represented as in~\eqref{eq1:a}.
Without loss of generality (see \cite[Lemma~2.20]{sg11}), we assume
that $\Gamma_m := [-1,1]$ for all $m \in \N$.
In order to ensure convergence of the series in~\eqref{eq1:a} and positivity of $a(x,\y)$ for each $x \in D$ and $\y \in \Gamma$,
we assume that there exists $\azeromin, \azeromax > 0$ such that 
\begin{equation} \label{eq2:a}
0 < a_0^{\rm min} \le a_0(x) \le a_0^{\rm max} < \infty
\quad \text{a.e.\ in } D
\end{equation}
and
\begin{align}\label{eq3:a}
 \tau := \frac{1}{a_0^{\rm min}} \, \sum_{m=1}^\infty \norm{a_m}{L^\infty(D)} < 1.
\end{align}
Let us introduce a measure $\pi = \pi(\y)$ on $(\Gamma, \mathcal{B}(\Gamma))$,
where $\mathcal{B}(\Gamma)$ is the Borel $\sigma$-algebra on $\Gamma$.
We assume that $\pi$  is the product of symmetric probability measures $\pi_m$ on
$(\Gamma_m, \mathcal{B}(\Gamma_m))$, with $\mathcal{B}(\Gamma_m)$ being the Borel 
$\sigma$-algebra on $\Gamma_m$, i.e.,
$\pi(\y) = \prod_{m=1}^\infty \pi_m(y_m)$, $\y \in \Gamma$.

Now, we can consider the Bochner space $V := L^2_\pi(\Gamma; H^1_0(D))$, where
$L^2_\pi(\Gamma)$ is the standard Lebesgue space on $\Gamma$ (with respect to the measure $\pi$) and
$H^1_0(D)$ denotes the Sobolev space 
of functions in $H^1(D)$ vanishing at the boundary $\partial D$ in the sense of traces.
For each $u,v \in V$, we define the following symmetric bilinear forms
\begin{align}
     \notag
     B_0(u,v) &:= \int_\Gamma \int_D a_0(x) \nabla u(x,\y)\cdot\nabla v(x,\y) \, dx \, d\pi(\y),\\
     \label{def:B}
     B(u,v) &:= B_0(u,v) + \sum_{m=1}^\infty \int_\Gamma \int_D y_m \, a_m(x) \nabla u(x,\y)\cdot\nabla v(x,\y) \, dx \, d\pi(\y).
\end{align}
Note that assumptions~\eqref{eq2:a} and \eqref{eq3:a} ensure that $B_0(\cdot,\cdot)$ and $B(\cdot,\cdot)$ are continuous and elliptic on $V$.
Therefore, they induce norms that we denote by $\enorm{\cdot}{0}$ and $\bnorm{\cdot}$, respectively.
Moreover, with
$0 < \lambda := \frac{a_0^{\rm min}}{a_0^{\rm max}\,(1+\tau)} < 1 <
 \Lambda := \frac{a_0^{\rm max}}{a_0^{\rm min}\,(1-\tau)} < \infty$, there holds
\begin{equation} \label{eq:lambda}
 \lambda \, \bnorm{v}^2 \le \enorm{v}{0}^2 \le \Lambda \, \bnorm{v}^2
 \quad \text{for all } v \in V;
\end{equation}
see, e.g., \cite[Proposition 2.22]{sg11}.

We can now introduce the weak formulation of \eqref{eq:strongform} that reads as follows:
Given $f \in H^{-1}(D)$, find $u \in V$ such that
\begin{equation} \label{eq:weakform}
	B(u,v) = F(v) := \int_\Gamma \int_D f(x) v(x,\y) \, dx \, d\pi(\y)
	\quad \text{for all $v \in V$}.
\end{equation}
The Lax--Milgram lemma proves the existence and uniqueness of the solution $u \in V$ to~\eqref{eq:weakform}.
 
\subsection{Discrete formulations}

For any finite-dimensional subspace of $V$, problem~\eqref{eq:weakform}
can be discretized by using the Galerkin projection onto this subspace.
The construction of the finite-dimensional subspaces of $V$ relies on the fact that
the Bochner space $V = L^2_\pi(\Gamma; H^1_0(D))$ is isometrically isomorphic to the tensor product Hilbert space 
$H^1_0(D) \otimes L_\pi^2(\Gamma)$ (see, e.g., \cite[Theorem B.17, Remark C.24]{sg11}).
Therefore, we can define the finite-dimensional subspace of $V$
as the tensor product of independently constructed finite-dimensional
subspaces of $H^1_0(D)$ and $L_\pi^2(\Gamma)$.

Let $\TT$ be a conforming triangulation of $D$ into compact simplices, and
let $\NN_{\TT}$ denote the corresponding set of interior vertices.
For the finite-dimensional subspace of $H^1_0(D)$, we use the space of first-order (P1) finite element functions:
\[
   X = \SS^1_0(\TT) := \set{v \in H^1_0(D)}{v|_T \text{ is affine for all $T \in \TT$}}.
\]

Let us now introduce the finite-dimensional subspaces of $L^2_\pi(\Gamma)$.
For each $m \in \N$, let $(P_n^m)_{n\in\N_0}$ denote the sequence of univariate polynomials
which are orthonormal with respect to the inner product $\langle \cdot,\cdot \rangle_{\pi_m}$ in $L^2_{\pi_m}(\Gamma_m)$,
such that $P_n^m$ is a polynomial of degree $n \in \N_0$. 
It is well-known that these polynomials form an orthonormal basis of $L^2_{\pi _m}(\Gamma_m)$
and {can be constructed using} the three-term recurrence formula
(see, e.g., \cite[Theorem 1.29]{g04} and recall that $\pi_m$ is symmetric)
\begin{equation} \label{three:term:rec}
      \beta_n^m \, P_{n+1}^m(y_m)
      = y_m \, P_n^m(y_m) - \beta_{n-1}^m \, P_{n-1}^m (y_m)
      \quad \text{for all } 
      n \in \N_0
\end{equation}
with initialisation $P_0^m \equiv 1$, $P_{-1}^m \equiv 0$ and coefficients
\[
   \beta_{-1}^m := 1,\quad
    \beta_n^m := \bigg(\int_{\Gamma_m} \big(y_m \, P_n^m(y_m) - \beta_{n-1}^m \, P_{n-1}^m (y_m)\big)^2 \, d\pi_m(y_m)\bigg)^{1/2}
   \quad
   \text{ for } n \in \N_0.
\]

In order to construct an orthonormal basis in $L_\pi^2(\Gamma)$,
consider the following countable set of finitely supported multi-indices
\begin{equation*} 
      \gotI \!:=\! \set{ \nu=(\nu_1,\nu_2,\dots) \in \N_0^\N}{\!\#\supp(\nu) < \infty}
\text{ with } \supp(\nu) \!:=\! \set{m\in\N}{\nu_m\neq0}.
\end{equation*}
Throughout the paper, the set $\gotI$, as well as any of its subsets, will be called the \emph{index set}. 
For each index $\nu \in \gotI$, we define the tensor product polynomial
\[
   P_\nu(\y) := \prod_{m\in\supp(\nu)} P_{\nu_m}^m(y_m).
\]
The set $\set{P_\nu}{\nu\in\gotI}$ is an orthonormal basis of $L^2_\pi(\Gamma)$; see~\cite[Theorem~2.12]{sg11}.
Since the Bochner space $V$ is isometrically isomorphic to $L^2_{\pi}(\Gamma) \otimes H^1_0(D)$,
each function $v \in V$ can be represented in the form
\begin{equation} \label{eq:representation}
      v(x,\y) = \sum_{\nu \in \gotI} v_\nu(x) P_\nu(\y)
      \quad\text{with unique coefficients $v_\nu \in H_0^1(D)$}.
\end{equation}
There holds the following important, yet elementary, observation.

\begin{lemma}\label{lemma:orthogonal}
For all $v,w \in V$, the following equality holds:
\begin{equation} \label{eq1:lemma:orthogonal}
   B_0(v,w) = \sum_{\nu \in \gotI} \int_D a_0(x) \, \nabla v_\nu(x) \cdot \nabla w_\nu(x) \, dx.
\end{equation}
In particular,
\begin{equation} \label{eq2:lemma:orthogonal}
   \enorm{v}{0}^2 = \sum_{\nu \in \gotI} \norm{a_0^{1/2}\nabla v_\nu}{L^2(D)}^2.
\end{equation}
\end{lemma}

\begin{proof}
Using representation~\eqref{eq:representation} for $v, w \in V$, we have that
\begin{align*}
   B_0(v,w) &= \sum_{\mu,\nu \in \gotI} \int_\Gamma \int_D a_0(x) \, \nabla v_\nu(x) \cdot \nabla w_\mu(x) P_\nu(\y) P_\mu(\y) \, dx \, d\pi(\y)
 \\
 &= \sum_{\mu,\nu \in \gotI} \bigg( \int_D a_0(x) \, \nabla v_\nu(x) \cdot \nabla w_\mu(x) \, dx \bigg)
 \bigg(\int_\Gamma P_\nu(\y) P_\mu(\y) \, d\pi(\y)\bigg).
\end{align*}
This proves~\eqref{eq1:lemma:orthogonal},
since $(P_\nu)_{\nu\in\gotI}$ is an orthonormal basis of $L^2_\pi(\Gamma)$.
Furthermore, selecting $w=v$ in~\eqref{eq1:lemma:orthogonal} we obtain~\eqref{eq2:lemma:orthogonal}.
\end{proof}

For any finite index set $\gotP \subset \gotI$,
the finite-dimensional subspace of $L^2_\pi(\Gamma)$
is given by $\hull{\set{P_\nu}{\nu\in\gotP}}$.
Thus, we can now define the finite-dimensional subspace $\VXP$ of~$V$~as
\begin{equation} \label{discrete:space:VXP}
  \VXP := X \otimes \hull{\set{P_\nu}{\nu\in\gotP}}.
\end{equation}
The discrete formulation of \eqref{eq:weakform} then reads as follows: Find $\uXP \in \VXP$ such that
\begin{equation} \label{eq:discreteform}
   B(\uXP,v_{X\gotP}) = F(v_{X\gotP})
   \quad \text{for all } v_{X\gotP} \in \VXP.
\end{equation}
Since $\VXP$ is a tensor product space, the Galerkin solution $\uXP {\in} \VXP$ can be represented~as
\[
  \uXP(x,\y) = \sum_{\nu \in \gotP} u_{\nu}(x) P_\nu(\y)
  \quad \text{with unique coefficients } u_{\nu} \in X.
\]
We implicitly assume that $\gotP$ always contains the zero-index ${\mathbf{0}} := (0,0,\dots)$.
We say that a parameter $y_m$ is \emph{active} in the index set $\gotP$
(and hence, in the Galerkin solution~$\uXP$) if $m \in \supp(\gotP) := \bigcup_{\nu \in \gotP} \supp(\nu)$.

The approximation provided by $\uXP$ can be improved by enriching the subspace $\VXP$.
There are many ways how one can combine the enrichments in spatial (finite element) approximations
with the ones in polynomial approximations on the parameter domain
(see, e.g., \cite{bps14, egsz14, bs16}).
Here we follow the approach developed in~\cite{bs16}.

Let us consider a conforming triangulation $\widehat\TT$ of $D$ obtained by a uniform refinement of $\TT$. 
We choose the enriched finite element space as
\begin{equation} \label{eq:enrichment:space}
 \widehat X := \SS^1_0(\widehat\TT) = X \oplus Y,
 \quad \text{where} \quad 
 Y := \set{v \in \widehat{X}}{v(\xi_\TT) = 0\ \ \text{for all $\xi_\TT \in \NN_\TT$}}.
\end{equation}
Here, the subspace $Y \subset H_0^1(D)$ is called the \emph{detail (finite element) space}. 
Note that the sum in~\eqref{eq:enrichment:space} is indeed a direct sum, i.e., $X \cap Y = \{0\}$.

In order to enrich the polynomial space on $\Gamma$,
we consider a finite index set $\gotQ\subset\gotI$ such that $\gotP\cap\gotQ = \emptyset$
and define the enriched index set $\widehat\gotP := \gotP \cup \gotQ$. 
The subset $\gotQ$ is called the \emph{detail index set}. 

The enriched finite-dimensional subspace of $V$ is then defined as follows:
\begin{equation*} 
      \VXPhat := \VXP\, \oplus\, \VYP \,\oplus\, \VXQ,
\end{equation*}
where
\begin{equation*} 
      \VYP := Y \otimes \hull{\set{P_\nu}{\nu\in\gotP}}
      \quad \text{and} \quad 
      \VXQ := X \otimes \hull{\set{P_\nu}{\nu\in\gotQ}}.
\end{equation*}
Note that $\VXP \oplus \VYP$ is a direct sum, whereas the direct sums $\VXP \oplus \VXQ$
and $\VYP \oplus \VXQ$ are also orthogonal, since $\gotP\cap\gotQ = \emptyset$.

Consider now the discrete formulation posed on $\VXPhat$:
Find $\uXPhat \in \VXPhat$ such that
\begin{equation} \label{eq:discreteform:enriched}
   B(\uXPhat, \widehat v_{X\gotP}) = F(\widehat v_{X\gotP})
   \quad \text{for all } \widehat v_{X\gotP} \in \VXPhat.
\end{equation}
Since $\VXP \subset \VXPhat$, the Galerkin orthogonality 
\begin{equation*}
 B(u - \uXPhat , \widehat v_{X\gotP}) = 0
 \quad \text{for all } \widehat v_{X\gotP} \in \VXPhat
\end{equation*}
and the symmetry of the bilinear form $B(\cdot,\cdot)$ imply that
\begin{equation} \label{eq:hh2:Pythagoras}
 \bnorm{u - \uXPhat}^2 + \bnorm{\uXPhat -\uXP}^2
 = \bnorm{u -\uXP}^2.
\end{equation}
In particular, this yields that $\bnorm{u - \uXPhat} \le \bnorm{u -\uXP}$.
As in~\cite{bs16}, we assume that a stronger property (usually referred to as \emph{saturation assumption}) holds:
There exists a uniform constant $0 < \beta < 1$ such that
\begin{equation}\label{eq:saturation}
 \bnorm{u - \uXPhat} \le \beta \, \bnorm{u -\uXP}.
\end{equation}

\section{\textsl{A~{\scriptsize POSTERIORI}} error estimation in the energy norm} \label{sec:errorestimators}
\noindent
Let $\uXP \in \VXP$ and $\uXPhat \in \VXPhat$ be two Galerkin approximations defined by~\eqref{eq:discreteform}
and~\eqref{eq:discreteform:enriched}, respectively.
It is well known that the two-level error $\uXP - \uXPhat$ provides
a reliable and efficient estimate for the error $u -\uXP$ in the energy norm.
Indeed, on the one hand, \eqref{eq:hh2:Pythagoras} implies the efficiency, i.e.,
\begin{equation} \label{eq:hh2:efficient}
 \bnorm{\uXPhat -\uXP} \le \bnorm{u -\uXP},
\end{equation}
and, on the other hand, it follows from \eqref{eq:hh2:Pythagoras} and elementary calculations
that the saturation assumption~\eqref{eq:saturation} is equivalent to the reliability, i.e.,
\begin{equation}\label{eq:hh2:reliable}
 \bnorm{u -\uXP} \le C \, \bnorm{\uXPhat -\uXP}
 \quad \text{with } C := (1-\beta^2)^{-1/2}.
\end{equation}
However, this error estimate bears the computational cost associated with finding the enhanced solution $\uXPhat \in \VXPhat$.
In addition to that, the evaluation of the norm $\bnorm{\uXP - \uXPhat}$ is expensive, due to its dependence on the
coefficients $a_m$ corresponding to all active parameters $y_m$ in the index set $\gotP$; see~\eqref{def:B}.
Therefore, our aim is to derive lower and upper bounds 
for $\bnorm{\uXPhat -\uXP}$ in terms of another error estimate
that avoids the computation of $\uXPhat$ and is inexpensive to evaluate.
Two approaches to this task are discussed in the next two subsections. 

\subsection{Hierarchical error estimators} \label{subsec:errorestimators:hierarchical}

One way to avoid the computation of $\uXPhat$ is to use a standard hierarchical approach to \textsl{a~posteriori} error estimation
that goes back to~\cite{BankW_85_SAE} (see also~\cite{bankHierarchical1996, ao00}).
In the context of parametric operator equations (and, in particular, for the parametric model problem~\eqref{eq:strongform}),
this approach was pursued in~\cite{bps14, bs16}.
Let us briefly recall the construction of the error estimators proposed in~\cite{bs16}.

Let $\eXPhat \in \VXPhat$ be the unique solution to the problem
\begin{equation} \label{eq:hat-e_XP}
   B_0(\eXPhat , \widehat v_{X\gotP}) = F(\widehat v_{X\gotP}) - B(\uXP, \widehat v_{X\gotP})
   \quad \text{for all } \widehat v_{X\gotP} \in \VXPhat.
\end{equation}
It follows from \eqref{eq:hat-e_XP} that $B_0(\eXPhat , \widehat v_{X\gotP}) = B(\uXPhat - \uXP , \widehat v_{X\gotP})$ for any $\widehat v_{X\gotP} \in \VXPhat$.
Hence, selecting $\widehat v_{X\gotP} = \eXPhat$ and $\widehat v_{X\gotP} = \uXPhat - \uXP$, 
the variational formulation~\eqref{eq:hat-e_XP} and the equivalence between $\bnorm{\cdot}$ and $\enorm{\cdot}{0}$ (see~\eqref{eq:lambda}) prove that
\begin{equation} \label{est:hat-e_XP}
 \lambda \, \enorm{\eXPhat}{0}^2 \le \bnorm{\uXPhat - \uXP}^2
 \le \Lambda \, \enorm{\eXPhat}{0}^2.
\end{equation}
Now, let $e_{\widehat X\gotP} \in \VXP \oplus \VYP =: V_{\widehat X\gotP}$ be the unique solution to
\begin{equation} \label{eq:e_hatXP}
 B_0(e_{\widehat X\gotP} , v_{\widehat X\gotP}) = F(v_{\widehat X\gotP}) - B(\uXP, v_{\widehat X\gotP})
 \quad \text{for all } v_{\widehat X\gotP} \in V_{\widehat X\gotP},
\end{equation}
and, for each $\nu \in \gotQ$, let $e_{X\gotQ}^{(\nu)} \in X \otimes \hull{P_\nu}$ be the unique solution to
\begin{equation} \label{eq:e_Xnu} 
 B_0(e_{X\gotQ}^{(\nu)} , v_{X}P_\nu) = F(v_{X}P_\nu) - B(\uXP, v_{X}P_\nu) \quad \text{for all $v_X \in X$}.
\end{equation}
Note that all subspaces $X \otimes \hull{P_\nu}$ ($\nu \in \gotQ$) are pairwise orthogonal with respect to $B_0(\cdot,\cdot)$.
Moreover, since $\gotP \cap \gotQ = \emptyset$,
the subspace $V_{\widehat X\gotP}$ is $B_0(\cdot,\cdot)$-orthogonal to $X \otimes \hull{P_\nu}$ for each $\nu \in \gotQ$.
Therefore, the following decomposition holds
\begin{equation} \label{eq0:final}
 \eXPhat = e_{\widehat X\gotP} + \sum_{\nu \in \gotQ} e_{X\gotQ}^{(\nu)}
 \quad \text{with} \quad
 \enorm{\eXPhat}{0}^2 = \enorm{e_{\widehat X\gotP}}{0}^2 + \sum_{\nu \in \gotQ} \enorm{e_{X\gotQ}^{(\nu)}}{0}^2.
\end{equation}
Replacing $e_{\widehat X\gotP}$ in \eqref{eq0:final} with the hierarchical error estimator $e_{Y\gotP} \in \VYP$ satisfying
\begin{equation} \label{eq:e_YP}
 B_0(e_{Y\gotP} , v_{Y\gotP}) = F(v_{Y\gotP}) - B(\uXP, v_{Y\gotP})
 \quad \text{for all } v_{Y\gotP} \in \VYP,
\end{equation}
\cite{bs16} introduces the following \textsl{a~posteriori} error estimate
\begin{equation} \label{est:bs16}
 \eta_{X\gotP}^2 := \enorm{e_{Y\gotP}}{0}^2 + \sum_{\nu \in \gotQ} \enorm{e_{X\gotQ}^{(\nu)}}{0}^2.
\end{equation}
We refer to $e_{Y\gotP}$ and $e_{X\gotQ}^{(\nu)}$ as the spatial and parametric error estimators, respectively.
Note that the parameter-free $B_0(\cdot,\cdot)$-norm is used in \eqref{est:bs16}
for efficient evaluation of the error estimators of both types.

Under the saturation assumption~\eqref{eq:saturation},
it has been shown in~\cite[Theorem~4.1]{bs16} that $\eta_{X\gotP}$
provides an efficient and reliable estimate for the energy norm of the discretization error.
In particular, the following inequalities hold
\begin{equation*} 
      \sqrt{\lambda} \, \eta_{X\gotP} \leq \bnorm{u - \uXP} \leq
      \frac{\sqrt \Lambda}{\sqrt{1 - \beta^2}\,\sqrt{1 - \gamma^2}} \, \eta_{X\gotP},
\end{equation*}
where $\lambda, \Lambda$ are the constants in \eqref{eq:lambda},
$\beta \in [0,1)$ is the saturation constant in \eqref{eq:saturation},
and
$\gamma \in [0,1)$ is the smallest constant in the 
strengthened Cauchy--Schwarz inequality
for the finite element subspaces $X$ and $Y$ (see, e.g., \cite[equation~(5.26)]{ao00}), i.e.,
\[
   \gamma = \sup_{u \in X,\, v \in Y}
                     \frac{\big| (a_0 \nabla u, \nabla v)_{L^2(D)} \big|}
                            {\norm{a_0^{1/2}\nabla u}{L^2(D)} \, \norm{a_0^{1/2}\nabla v}{L^2(D)}}.
\]

Note that in order to compute the spatial error estimator $e_{Y\gotP}$ and the parametric error estimators $e_{X\gotQ}^{(\nu)}$ ($\nu \in \gotQ$),
one needs to solve the linear systems associated
with discrete formulations~\eqref{eq:e_YP} 
and~\eqref{eq:e_Xnu}, respectively. 
In the following section, we propose a two-level error estimation technique
that avoids the solution of the linear system for the spatial error estimator.

\subsection{Two-level \textsl{a~posteriori} error estimation in the energy norm} \label{sec:two:level:errorestimator}

Recall that $\widehat\TT$ denotes a uniform refinement of the triangulation $\TT$.
Let $\NN_{\widehat \TT}$ denote the set of interior vertices of $\widehat\TT$ and suppose that 
$\NN_{\widehat \TT} \backslash \NN_\TT = \{\vertexY_1, \dots, \vertexY_n\}$. 
For each new vertex $\vertexY_j \in \NN_{\widehat \TT} \backslash \NN_\TT$, let $\varphi_j \in \widehat X$ be 
the corresponding hat function, i.e., $\varphi_j(\vertexY_j) = 1$ and $\varphi_j(\vertexY) = 0$ for all 
$\vertexY \in \NN_{\widehat \TT} \backslash \{\vertexY_j\}$. 
Then, the set $\basisY := \{\varphi_1, \dots, \varphi_n\}$ is a basis of
the detail finite element space $Y$ defined in~\eqref{eq:enrichment:space}.
Moreover, there exists a constant $\KY \ge 1$ such that 
\begin{equation} \label{K_constant}
   \# \set{\varphi_j \in \basisY}{\text{interior}\big(\supp(\varphi_j) \cap T\big) \neq \emptyset} \le \KY
   \quad \text{for all } T \in \TT.
\end{equation}

Turning now to the detail index set $\gotQ \subset \gotI$, we follow the construction suggested in~\cite{bs16}.
Let $\varepsilon^{(m)} := ( \varepsilon^{(m)}_1,  \varepsilon^{(m)}_2,\dots)$ ($m \in \N$) denote
the Kronecker delta index such that $\varepsilon^{(m)}_k =\delta_{mk}$, for all $k \in \N$. 
For a fixed $\overline{M} \in \N$, define
\begin{equation} \label{finite:index:set:Q}
\gotQ := \set{ \mu \in \gotI\setminus \gotP}
                              {\mu  = \nu \pm \varepsilon^{(m)} \text{ for some } \nu \in \gotP
                              \text{ and some } m = 1,\dots, M_\gotP + \overline{M} },
\end{equation}
where $M_\gotP \in \N$ is the number of active parameters in the index set $\gotP$, that is 
\begin{equation*}
  M_\gotP := 
  \begin{cases}
           0        &  \text{if $\gotP = \{ (0,0,\dots) \} $}, \\
           \max \set{ \max (\supp (\nu) ) } {\nu \in \gotP \setminus \{ (0,0,\dots) \} } & \text{otherwise}.
  \end{cases}
\end{equation*}
Thus, for a given $\gotP  \subset \gotI$, the index set $\gotQ$ defined by~\eqref{finite:index:set:Q} contains
only those ``neighbors'' of all indices in $\gotP$ that have up to $M_\gotP + \overline{M}$ active parameters,
that is $\overline{M}$ parameters more than currently activated in the index set $\gotP$
(we refer to Lemma~4.3 and Corollary~4.1 in \cite{bs16} for theoretical underpinnings of this construction). 

Having fixed the detail space $Y$ and the detail index set $\gotQ$, we can now define the following estimate
of the energy error $\bnorm{u - \uXP}$,
which avoids the computation of $e_{\widehat X\gotP}$ from~\eqref{eq:e_hatXP}:
\begin{equation} \label{def:mu}
          \muXP^2 := \sum_{\nu \in \gotP} \sum_{j=1}^n 
          \frac{|F(\varphi_jP_\nu)-B(\uXP,\varphi_jP_\nu)|^2}{\norm{a_0^{1/2}\nabla\varphi_j}{L^2(D)}^2}
           + \sum_{\nu \in \gotQ} \enorm{e_{X\gotQ}^{(\nu)}}{0}^2,
\end{equation}
where $\eXQnu$ are defined in \eqref{eq:e_Xnu} for each $\nu \in \gotQ$.

The following theorem is the main result of this section.

\begin{theorem}\label{theorem:twolevel}
Let $u \in V$ be the solution to problem~\eqref{eq:weakform}, and let
$\uXP \in \VXP$ and $\uXPhat \in  \VXPhat$ be two Galerkin approximations satisfying
\eqref{eq:discreteform} and \eqref{eq:discreteform:enriched}, respectively.
Then, there exists a constant $\Cthm \ge 1$, which depends only
on the shape regularity of $\TT$ and $\widehat\TT$, the (local) mesh-refinement rule,
and the mean field $a_0$, such that the error estimate $\muXP$ defined by~\eqref{def:mu} satisfies
\begin{equation} \label{eq1:theorem:twolevel}
\frac{\lambda}{\KY} \, \muXP^2
\le \bnorm{\uXPhat - \uXP}^2
\le \Lambda \Cthm  \, \muXP^2,
\end{equation}
where $\lambda, \Lambda$ are the constants in \eqref{eq:lambda} and
$\KY$ is the constant in~\eqref{K_constant}.

Furthermore, under the saturation assumption~\eqref{eq:saturation}, there holds
\begin{equation} \label{eq2:theorem:twolevel}
 \frac{\lambda}{\KY} \, \muXP^2
 \le \bnorm{u - \uXP}^2 
 \le \frac{\Lambda \Cthm}{1-\beta^2} \, \muXP^2,
\end{equation}
where $\beta \in [0,1)$ is the saturation constant in \eqref{eq:saturation}.
\end{theorem}

\begin{remark} 
On the one hand, Theorem~{\rm \ref{theorem:twolevel}} shows that $\muXP$ provides a reliable and efficient estimate
for the energy norm of the error (see~\eqref{eq2:theorem:twolevel}). 
On the other hand, recall that $\bnorm{\uXPhat - \uXP}$ is the error reduction (in the energy norm)
that would be achieved if the
enhanced solution $\uXPhat \in \widehat{V}_{X\gotP}$ were to be computed (see~\eqref{eq:hh2:Pythagoras}).
Hence,
inequalities~\eqref{eq1:theorem:twolevel} show that $\muXP$ also provides an estimate for this error reduction.
Moreover, note that Theorem~{\rm \ref{theorem:twolevel}} holds for any finite detail index set $\gotQ \subset \gotI \backslash \gotP$
and any conforming refinement $\widehat\TT$ of $\TT$ (and the corresponding detail space $Y$). Finally, we stress that our proof of
Theorem~{\rm \ref{theorem:twolevel}} holds for any spatial dimension,
while we restrict the proof to 2D to ease the presentation.
\end{remark}%

\begin{remark}
For the implementation of $\muXP$, note that the spatial contributions include:
\begin{itemize}
\item
in the numerator, the entries of the algebraic residual of $\uXP$, where the Ga\-ler\-kin data are computed with respect to the enrichment $V_{Y\gotP}$;
\item
in the denominator, the diagonal elements of the spatial stiffness matrix with 
respect to the detail space $Y$.
\end{itemize}
Moreover, the denominator can be easily simplified.
Suppose that $T \in \TT$ and $\varphi_j \in \basisY$ with 
$\supp(\varphi_j) \subseteq {\rm patch}(T)$. With $h_T$ being the diameter of $T$, there holds
\begin{equation*}
 \norm{a_0^{1/2}\nabla\varphi_j}{L^2(D)}^2 \simeq 
 h_T^{-2} \, \norm{\varphi_j}{L^2(D)}^2 \simeq 1,
\end{equation*}
where the equivalence constants depend only on the shape regularity of $\widehat\TT$, the (local) mesh-refinement rule,
and the mean field $a_0$.
\end{remark}

In order to prove Theorem~\ref{theorem:twolevel}, let us collect some auxiliary results.

\begin{lemma} 
Let $v = \sum_{\nu\in\gotP}\sum_{j=1}^n v_{j\nu} P_\nu \in V_{Y\gotP}$ with $v_{j\nu} \in \hull{\varphi_j}$.
Then,
\begin{equation} \label{eq:lemma:localization}
 \KY^{-1} \, \enorm{v}{0}^2
 \le \sum_{\nu\in\gotP}\sum_{j=1}^n \norm{a_0^{1/2}\nabla v_{j\nu}}{L^2(D)}^2
 \le  \Cloc \, \enorm{v}{0}^2,
\end{equation}
where the constant $\Cloc > 0$ depends only 
on the shape regularity of $\widehat\TT$, the (local) mesh-refinement rule,
and the mean field $a_0$.
\end{lemma}

\begin{proof}
The proof consists of three steps.

{\bf Step~1.} 
Let $T \in \TT$ and $w_j \in \hull{\varphi_j}$ for all $j=1,\dots,n$. Observe that
\begin{equation*}
 \norm[\Big]{a_0^{1/2} \nabla \sum_{j=1}^n w_j}{L^2(T)}
 \le \sum_{j=1}^n \norm{a_0^{1/2} \nabla w_j}{L^2(T)}
 \le \sqrt{\KY} \, \bigg( \sum_{j=1}^n \norm{a_0^{1/2} \nabla w_j}{L^2(T)}^2 \bigg)^{1/2}.
\end{equation*}
Hence, summing over all $T \in \TT$, we obtain
\begin{equation*}
 \norm[\Big]{a_0^{1/2}\nabla \sum_{j=1}^n w_j}{L^2(D)}^2
 \le \KY \sum_{T \in \TT} \sum_{j=1}^n \norm{a_0^{1/2} \nabla w_j}{L^2(T)}^2
 = \KY \, \sum_{j=1}^n \norm{a_0^{1/2}\nabla w_j}{L^2(D)}^2.
\end{equation*}%

{\bf Step~2.} 
To prove the converse estimate, 
let $T \in \TT$ and $w_j \in \hull{\varphi_j}$ for all $j=1,\dots,n$.
Note that $\norm{a_0^{1/2}\nabla w_Y}{L^2(T)}=0$ for $w_Y \in Y$ implies that $w_Y|_T = 0$,
since $w_Y|_T$ would be a constant with $w_Y(\xi_T) = 0$ for all $\xi_T \in \NN_\TT \cap T$. 
Thus, using the representation $w_Y = \sum_{j=1}^n w_j$ with unique $w_j \in \hull{\varphi_j}$ for all $j = 1,\dots, n$,
the quantities
\begin{equation*}
 \bigg(\sum_{j=1}^n \norm{a_0^{1/2} \nabla w_j}{L^2(T)}^2\bigg)^{1/2} 
 \quad\text{and}\quad
 \norm[\bigg]{a_0^{1/2} \nabla \sum_{j=1}^n w_j}{L^2(T)}
\end{equation*}
define two norms on $Y|_T := \set{w_Y|_T}{w_Y \in Y}$. 
Due to equivalence of norms on finite dimensional spaces, we use the standard scaling argument to obtain
\begin{equation} \label{eq:local_step2:2}
 \sum_{j=1}^n \norm[\big]{a_0^{1/2} \nabla w_j}{L^2(T)}^2
 \simeq \norm[\bigg]{a_0^{1/2} \nabla \sum_{j=1}^n w_j}{L^2(T)}^2 \quad
 \text{for all } w_j \in \hull{\varphi_j},\ j = 1,\dots, n,
\end{equation}
where the equivalence constants depend on $a_0$ and the shape regularity of $\widehat\TT$, as well as on the type of the
mesh-refinement strategy (that affects the configuration of the local space $Y|_T$).
Summing the upper bounds in \eqref{eq:local_step2:2} over all $T \in \TT$, we prove that
\begin{equation*}
 \sum_{j=1}^n \norm{a_0^{1/2}\nabla w_j}{L^2(D)}^2
 \lesssim \norm[\Big]{a_0^{1/2}\nabla \sum_{j=1}^n w_j}{L^2(D)}^2
 \quad\text{for all $w_j \in \hull{\varphi_j}$, $j=1,\dots,n$}.
\end{equation*}

{\bf Step~3.}
Using Lemma~\ref{lemma:orthogonal}, the estimates proved in Step~1 and Step~2 imply that
\begin{equation*}
 \enorm{v}{0}^2
 \reff{eq2:lemma:orthogonal}= \sum_{\nu\in\gotP} \norm[\Big]{a_0^{1/2}\nabla \sum_{j=1}^n v_{j\nu}}{L^2(D)}^2
 \simeq \sum_{\nu\in\gotP} \sum_{j=1}^n  \norm{a_0^{1/2}\nabla v_{j\nu}}{L^2(D)}^2.
\end{equation*}
This concludes the proof.
\end{proof}

\begin{lemma} \label{lemma:interpolation}
Let $\mathsf{P}_\TT$ be the nodal interpolation operator onto $\SS^1(\TT)$. 
For any function $\vXhatP = \sum_{\nu\in\gotP} \widehat{v}_{\nu} P_\nu \in \VXhatP$ 
with $\widehat v_\nu \in \Xhat$,
define $v_{X\gotP} = \sum_{\nu\in\gotP} (\mathsf{P}_\TT\, \widehat{v}_{\nu})P_\nu \in \VXP$. 
Then, we have the representation
\begin{equation} \label{eq1:lemma:interpolation}
\vXhatP - \vXP = \sum_{\nu\in\gotP} \sum_{j=1}^n v_{j\nu} P_\nu \in \VYP
 \quad\text{with } v_{j\nu} \in \hull{\varphi_j},
\end{equation}
and there holds
\begin{equation} \label{eq2:lemma:interpolation}
\enorm{\vXhatP - \vXP}{0}
 \le \Cstab \, \enorm{\vXhatP}{0},
\end{equation}
where the constant $\Cstab > 0$ depends only 
on the shape regularity of $\widehat\TT$, the (local) mesh-refinement rule,
and the mean field $a_0$.
\end{lemma}

\begin{proof}
The proof consists of two steps.

{\bf Step~1.}
Let $v_\Xhat \in \Xhat$. Then, $v_X := \mathsf{P}_\TT\, v_\Xhat \in X = \SS^1_0(\TT)$. 
Since $\Xhat = X \oplus Y$, there exist unique $w_X \in X$ and $w_Y \in Y$ such that $v_\Xhat - v_X = w_X + w_Y$. 
Observe that $(v_\Xhat - v_X)(\xi_\TT) = 0 = w_Y(\xi_\TT)$ for every vertex $\xi_\TT \in \NN_\TT$.
Hence $w_X(\xi_\TT) = 0$ for all $\xi_\TT \in \NN_\TT$ and hence $w_X=0$, i.e., 
$v_\Xhat - v_X \in Y$. 
Moreover, a scaling argument proves that
\begin{equation*}
 \norm{a_0^{1/2}\nabla(\mathsf{P}_\TT\,v_{\widehat X})}{L^2(T)}
 \lesssim \norm{a_0^{1/2}\nabla v_{\widehat X}}{L^2(T)}
 \quad \text{for all } T \in \TT,
\end{equation*}
where the hidden constant depends only on $a_0$ and the shape regularity of $\widehat\TT$, as well as
on the type of the mesh-refinement strategy (that affects the configuration of the local space $Y|_T$).
Summing this estimate over all $T \in \TT$, we see that
\begin{equation} \label{dp:0511a}
 \norm{a_0^{1/2}\nabla(\mathsf{P}_\TT\,v_{\widehat X})}{L^2(D)}
 \lesssim \norm{a_0^{1/2}\nabla v_{\widehat X}}{L^2(D)}
 \quad \text{for all } v_{\widehat X} \in \widehat X.
\end{equation}

{\bf Step~2.} 
Recall that $v_{\Xhat\gotP} - \vXP = \sum_{\nu\in\gotP} (\widehat v_{\nu} - \mathsf{P}_\TT\, \widehat v_{\nu}) P_\nu$. 
According to Step~1, $\widehat v_{\nu} - \mathsf{P}_\TT\, \widehat v_{\nu} \in Y$ and hence 
$\widehat v_{\nu} - \mathsf{P}_\TT\, \widehat v_{\nu} = \sum_{j=1}^n v_{j\nu}$ with some $v_{j\nu} \in \hull{\varphi_j}$. 
This proves~\eqref{eq1:lemma:interpolation}. 
Moreover, Lemma~\ref{lemma:orthogonal} yields that
\begin{equation*}
 \enorm{\vXP}{0}^2
 \reff{eq2:lemma:orthogonal} = \sum_{\nu\in\gotP} \norm{a_0^{1/2}\nabla(\mathsf{P}_\TT\, \widehat v_{\nu})}{L^2(D)}^2 
 \reff{dp:0511a}\lesssim \sum_{\nu\in\gotP} \norm{a_0^{1/2}\nabla \widehat v_{\nu}}{L^2(D)}^2
 \reff{eq2:lemma:orthogonal}= \enorm{v_{\Xhat\gotP}}{0}^2.
\end{equation*}
The triangle inequality then proves~\eqref{eq2:lemma:interpolation}.
\end{proof}

To state the following lemma, we need some further notation. Let $\G_{X\gotP} : V \to \VXP$ 
be the orthogonal projection onto $\VXP$ with respect to $B_0(\cdot,\cdot)$, i.e., for all $w \in V$,
\begin{equation*}
B_0(\G_{X\gotP}w , \vXP) = B_0(w , \vXP )
\quad \text{for all } \vXP \in \VXP.
\end{equation*}
Furthermore, for $\nu \in \gotP$ and $\varphi_j \in \basisY$ (i.e., for $j \in \{1,\ldots,n\}$),
let $\G_{j\nu} : V \to \hull{\varphi_j P_\nu}$ be the 
orthogonal projection onto $\hull{\varphi_j P_\nu}$ with respect to $B_0(\cdot,\cdot)$, i.e., for all $w \in V$,
\begin{equation*}
 B_0(\G_{j\nu}w , vP_\nu) = B_0(w , v P_\nu)
 \quad \text{for all } v \in \hull{\varphi_j}.
\end{equation*}

\begin{lemma}\label{lemma:twolevel}
For any $\vXhatP \in \VXhatP$, the following estimates hold
\begin{equation} \label{eq:lemma:twolevel}
 \CY^{-1} \, \enorm{\vXhatP}{0}^2
 \le \enorm{\G_{X\gotP} \vXhatP}{0}^2 + \sum_{\nu\in\gotP} \sum_{j=1}^n \enorm{\G_{j\nu} \vXhatP}{0}^2
 \le 2 \KY \, \enorm{\vXhatP}{0}^2,
\end{equation}
where the constant $\CY\ge1$ depends only on the shape regularity of $\widehat\TT$,
the (local) mesh-refinement rule, and the mean field $a_0$.
Moreover, the upper bound holds with constant $\KY$ (instead of $2\KY$), if $\G_{X\gotP} \vXhatP = 0$.
\end{lemma}

\begin{proof}
The proof consists of two steps.

{\bf Step~1.}
Let us prove the lower bound in~\eqref{eq:lemma:twolevel}. 
To this end, let $\vXhatP \in \VXhatP$ and choose $\vXP \in \VXP$ as in Lemma~\ref{lemma:interpolation}. 
Then we have that
\begin{eqnarray*}
 \enorm{\vXhatP}{0}^2
 &\hspace{-9pt}{\reff{eq1:lemma:interpolation}=}\hspace{-11pt}& 
 B_0(\vXhatP , \vXP) + \sum_{\nu\in\gotP}\sum_{j=1}^n B_0(\vXhatP , v_{j\nu} P_\nu)
 \\
 &\hspace{-9pt}{=}\hspace{-11pt}& 
 B_0(\G_{X\gotP}\vXhatP , \vXP) + \sum_{\nu\in\gotP}\sum_{j=1}^n B_0(\G_{j\nu}\vXhatP , v_{j\nu} P_\nu)
 \\
 &\hspace{-9pt}{\le}\hspace{-11pt}& 
 \Big( \enorm{\G_{X\gotP}\vXhatP}{0}^2 + \sum_{\nu\in\gotP}\sum_{j=1}^n \enorm{\G_{j\nu}\vXhatP}{0}^2 \Big)^{1/2}
 \Big( \enorm{\vXP}{0}^2 + \sum_{\nu\in\gotP}\sum_{j=1}^n \enorm{v_{j\nu} P_\nu}{0}^2 \Big)^{1/2}.
\end{eqnarray*}
First, note that 
\begin{equation*}
 \enorm{\vXP}{0}
 \le \enorm{\vXhatP}{0}
 + \enorm{\vXhatP - \vXP}{0}
 \reff{eq2:lemma:interpolation}
 \le (1 + \Cstab) \, \enorm{\vXhatP}{0}.
\end{equation*}
Second, we use the upper bound in~\eqref{eq:lemma:localization} to obtain that
\begin{align*}
 &\sum_{\nu\in\gotP}\sum_{j=1}^n \enorm{v_{j\nu} P_\nu}{0}^2
  \reff{eq2:lemma:orthogonal}= \sum_{\nu\in\gotP}\sum_{j=1}^n \norm{a_0^{1/2}\nabla v_{j\nu}}{L^2(D)}^2
  \\&\qquad
  \reff{eq:lemma:localization}\le
  \Cloc \, \enorm[\Big]{\sum_{\nu\in\gotP}\sum_{j=1}^n v_{j\nu} P_\nu}{0}^2
  \reff{eq1:lemma:interpolation}= \Cloc \, \enorm{\vXhatP - \vXP}{0}^2
  \reff{eq2:lemma:interpolation}\le \Cloc \Cstab^2 \, \enorm{\vXhatP}{0}^2.
\end{align*}
Combining the foregoing three estimates, we conclude that
\begin{equation*}
 \enorm{\vXhatP}{0}^2
 \le \CY \, \Big( \enorm{\G_{X\gotP}\vXhatP}{0}^2 +
 \sum_{\nu\in\gotP}\sum_{j=1}^n \enorm{\G_{j\nu}\vXhatP}{0}^2 \Big),
\end{equation*}
where $\CY = (1+\Cstab)^2 \, + \, \Cloc \Cstab^2 \ge 1$.

{\bf Step~2.}
Let us now prove the upper bound in~\eqref{eq:lemma:twolevel}.
One has
\begin{align*}
 &\enorm{\G_{X\gotP}\vXhatP}{0}^2 + \sum_{\nu\in\gotP}\sum_{j=1}^n \enorm{\G_{j\nu}\vXhatP}{0}^2
 = B_0(\G_{X\gotP}\vXhatP , \vXhatP) 
 + \sum_{\nu\in\gotP}\sum_{j=1}^n B_0(\G_{j\nu}\vXhatP , \vXhatP)
 \\&\quad
 = B_0\Big(\G_{X\gotP}\vXhatP + \sum_{\nu\in\gotP}\sum_{j=1}^n \G_{j\nu}\vXhatP , \vXhatP\Big)
 \le \enorm[\Big]{ \G_{X\gotP}\vXhatP + \sum_{\nu\in\gotP}\sum_{j=1}^n \G_{j\nu}\vXhatP }{0}
 \,\enorm{\vXhatP}{0}.
\end{align*}
First, note that
\begin{align*}
 \enorm[\Big]{ \G_{X\gotP}\vXhatP + \sum_{\nu\in\gotP}\sum_{j=1}^n \G_{j\nu}\vXhatP }{0}
 \le \sqrt{2} \, \Big( \enorm{\G_{X\gotP}\vXhatP}{0}^2 + \enorm[\Big]{\sum_{\nu\in\gotP}\sum_{j=1}^n \G_{j\nu}\vXhatP }{0}^2 \Big)^{1/2}.
\end{align*}
Second, let $\G_{j\nu}\vXhatP = v_{j\nu} P_\nu$. Then, $v_{j\nu} \in \hull{\varphi_j}$ and
\begin{align*}
 \enorm[\Big]{\sum_{\nu\in\gotP}\sum_{j=1}^n \G_{j\nu}\vXhatP }{0}^2
 = \enorm[\Big]{\sum_{\nu\in\gotP}\sum_{j=1}^n v_{j\nu} P_\nu }{0}^2
 &\reff{eq:lemma:localization}\le
 \KY
 \sum_{\nu\in\gotP}\sum_{j=1}^n \norm{a_0^{1/2}\nabla v_{j\nu}}{L^2(D)}^2
 \\&
 \reff{eq2:lemma:orthogonal}= \KY \sum_{\nu\in\gotP}\sum_{j=1}^n \enorm{\G_{j\nu}\vXhatP }{0}^2.
\end{align*}
Combining the foregoing three inequalities, we obtain the estimate
\begin{equation*}
  \Big( \enorm{\G_{X\gotP}\vXhatP}{0}^2 + \sum_{\nu\in\gotP}\sum_{j=1}^n \enorm{\G_{j\nu}\vXhatP}{0}^2 \Big)^{1/2}
 \le \sqrt{2\KY} \, \enorm{\vXhatP}{0},
\end{equation*}
which yields the desired upper bound in~\eqref{eq:lemma:twolevel}.
\end{proof}

\begin{proof}[Proof of Theorem~\ref{theorem:twolevel}]
The proof consists of two steps.

{\bf Step~1.} 
Recall the definition of $e_{\Xhat\gotP} \in \VXhatP$ given in~\eqref{eq:e_hatXP}.
Since $V_{X\gotP} \subset V_{\Xhat\gotP}$, we deduce from~\eqref{eq:discreteform} and~\eqref{eq:e_hatXP} that
\begin{equation*}
 B_0(e_{\Xhat\gotP} , \vXP) = 0
 \quad \text{for all } \vXP \in \VXP.
\end{equation*}
Hence, $\G_{X\gotP} e_{\Xhat\gotP} = 0$ and therefore Lemma~\ref{lemma:twolevel} proves that
\begin{equation*}
 \CY^{-1} \, \enorm{e_{\Xhat\gotP}}{0}^2
 \le \sum_{\nu\in\gotP}\sum_{j=1}^n \enorm{\G_{j\nu}e_{\Xhat\gotP}}{0}^2
  \le \KY \, \enorm{e_{\Xhat\gotP}}{0}^2.
\end{equation*}
Since $C_Y, \KY \ge 1$, we use decomposition~\eqref{eq0:final} to obtain
\begin{equation} \label{dp:0511c}
\CY^{-1} \, \enorm{\widehat e_{X\gotP}}{0}^2
\le \sum_{\nu\in\gotP} \sum_{j=1}^n \enorm{\G_{j\nu}e_{\Xhat\gotP}}{0}^2 + \sum_{\nu \in \gotQ} \enorm{e_{X\gotQ}^{(\nu)}}{0}^2
\le \KY \, \enorm{\widehat e_{X\gotP}}{0}^2. 
\end{equation}

{\bf Step~2.}
The orthogonal projection onto the one-dimensional space $\hull{\varphi_j P_\nu}$ satisfies
\begin{equation*}
 \G_{j\nu} v = \frac{B_0(v,\varphi_j P_\nu)}{\enorm{\varphi_j P_\nu}{0}^2}\,\varphi_j P_\nu
 \quad \text{for all } v \in V.
\end{equation*}
Hence, 
\begin{equation*}
 \enorm{\G_{j\nu}e_{\widehat X\gotP}}{0}
 \reff{eq:e_hatXP} = \frac{|F(\varphi_j P_\nu) - B(\uXP,\varphi_j P_\nu)|}{\enorm{\varphi_j P_\nu}{0}}
 \reff{eq2:lemma:orthogonal}= \frac{|F(\varphi_j P_\nu) - B(\uXP,\varphi_j P_\nu)|}{\norm{a_0^{1/2} \nabla \varphi_j}{L^2(D)}}.
\end{equation*}
Using the definition of $\muXP$ given in~\eqref{def:mu}, estimate~\eqref{dp:0511c} thus implies that 
\begin{equation*}
 \frac{\lambda}{\KY} \, \muXP^2
 \reff{dp:0511c}\le \lambda \, \enorm{\widehat e_{X\gotP}}{0}^2
 \reff{est:hat-e_XP}\le \bnorm{\uXPhat - \uXP}^2 
 \reff{est:hat-e_XP}\le \Lambda \, \enorm{\widehat e_{X\gotP}}{0}^2
 \reff{dp:0511c}\le \Lambda \, \CY \, \muXP^2.
\end{equation*}
This proves~\eqref{eq1:theorem:twolevel} with $\Cthm = \CY$.
Estimate~\eqref{eq2:theorem:twolevel} then immediately follows from~\eqref{eq:hh2:efficient} and~\eqref{eq:hh2:reliable}.
\end{proof}

\section{Goal-oriented adaptivity for parametric problems} \label{sec:goaloriented:adaptivity}

\subsection{Goal-oriented error estimation in the parametric setting} 

First, let us formulate the abstract result on goal-oriented error estimation (see \S\ref{subsec:abstract:goaloriented:setting})
in the context of the sGFEM discretization for the parametric model problem~\eqref{eq:strongform}.
Let $u \in V = L^2_\pi(\Gamma; H^1_0(D)) $ be the unique \emph{primal} solution satisfying~\eqref{eq:weakform}.
Then, given a function $g \in H^{-1}(D)$, let us consider the quantity of interest $Q(u(\cdot,\y)) := \int_D g(x) u(x,\y) \, dx$.
Then, introducing the goal functional $G \in V'$ defined by
\begin{equation} \label{parametric:goal}
      G(v) := \int_\Gamma Q(v(\cdot,\y)) \, \dpiy =
      \int_\Gamma \int_D g(x)v(x,\y) \, dx \, \dpiy\quad \text{for all $v \in V$},
\end{equation}
we are interested in approximating $G(u)$---the mean value of the quantity of interest.

Let $z \in V$ be the unique \emph{dual} solution satisfying
\begin{equation} \label{parametric:dual:problem}
      B(v,z) = G(v) \quad \text{for all $v \in V$}.
\end{equation}
Considering the same finite-dimensional subspace $\VXP \subset V$
as used for the (primal) Galerkin approximation $\uXP \in \VXP$ (see~\eqref{discrete:space:VXP} and~\eqref{eq:discreteform}),
let $\zXP \in \VXP$ be the dual Galerkin solution  satisfying
\begin{equation} \label{parametric:dual:problem:discrete}
      B(\vXP,\zXP) = G(\vXP) \quad \text{for all $\vXP \in \VXP$.}
\end{equation}
Recall that $\muXP$ defined by \eqref{def:mu} provides a reliable and efficient estimate
for the energy error in the Galerkin approximation of the primal solution $u$.
Let us denote by $\zetaXP$ the corresponding estimate for the energy error
in the Galerkin approximation of the dual solution $z$
(recall that the bilinear form $B(\cdot,\cdot)$ is symmetric).
It follows from Theorem~\ref{theorem:twolevel} that
\[
  \bnorm{u-\uXP} \lesssim \muXP 
  \quad\text{and}\quad
  \bnorm{z-\zXP} \lesssim \zetaXP.
\]
From the abstract result in \S\ref{subsec:abstract:goaloriented:setting}
(see~\eqref{eq:errorInequality}--\eqref{goal:error:reliability}), we therefore conclude that
the error in approximating~$G(u)$
can be controlled by the product of the two error estimates~$\muXP$ and~$\zetaXP$, i.e., 
$\lvert G(u)-G(\uXP) \rvert  \lesssim \muXP\,\zetaXP$.

Let us now discuss some important ingredients of the goal-oriented adaptive algorithm.

\begin{figure}[b!]
\begin{minipage}{.19\textwidth}
\centering%
\includegraphics[width=\textwidth]{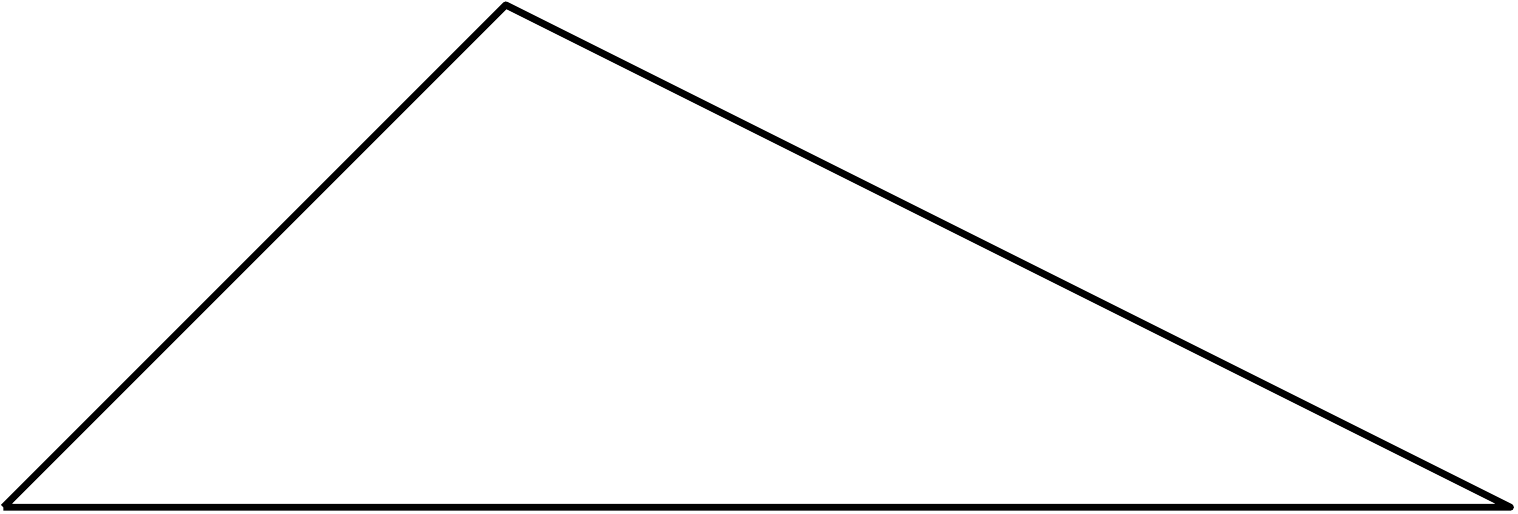}%
\\[.5mm]%
\includegraphics[width=\textwidth]{figures/refinement/bisec0.pdf}%
\end{minipage}
\begin{minipage}{.19\textwidth}
\centering%
\includegraphics[width=\textwidth]{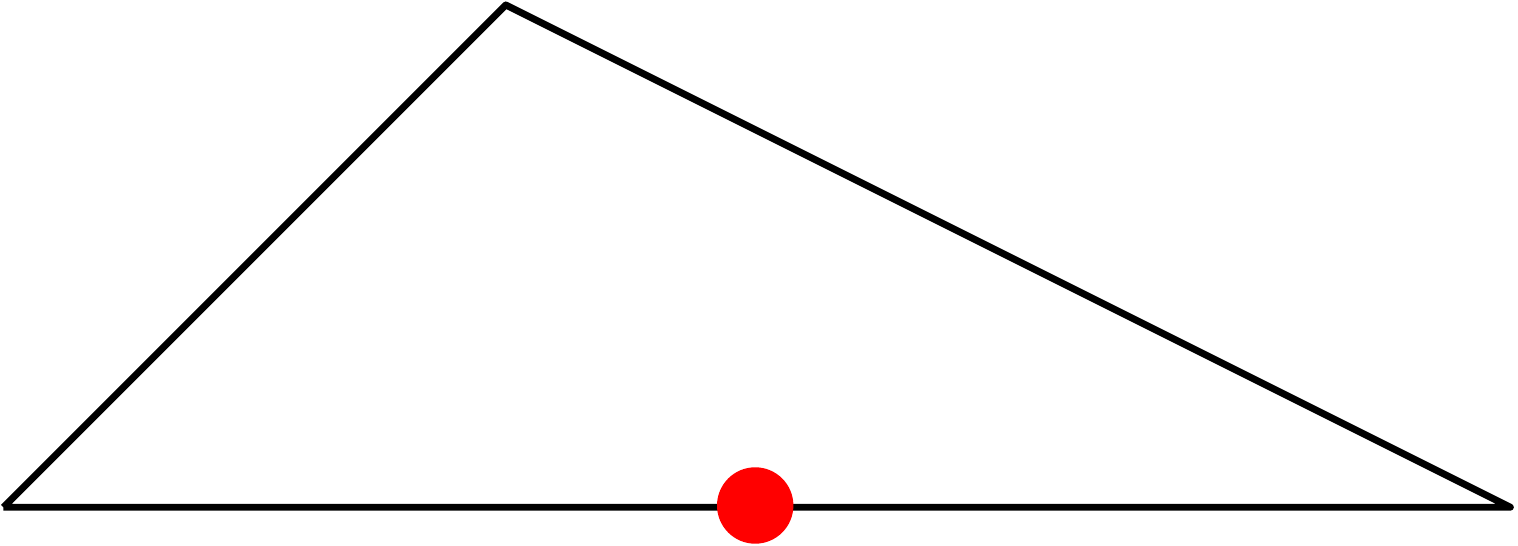}%
\newline%
\includegraphics[width=\textwidth]{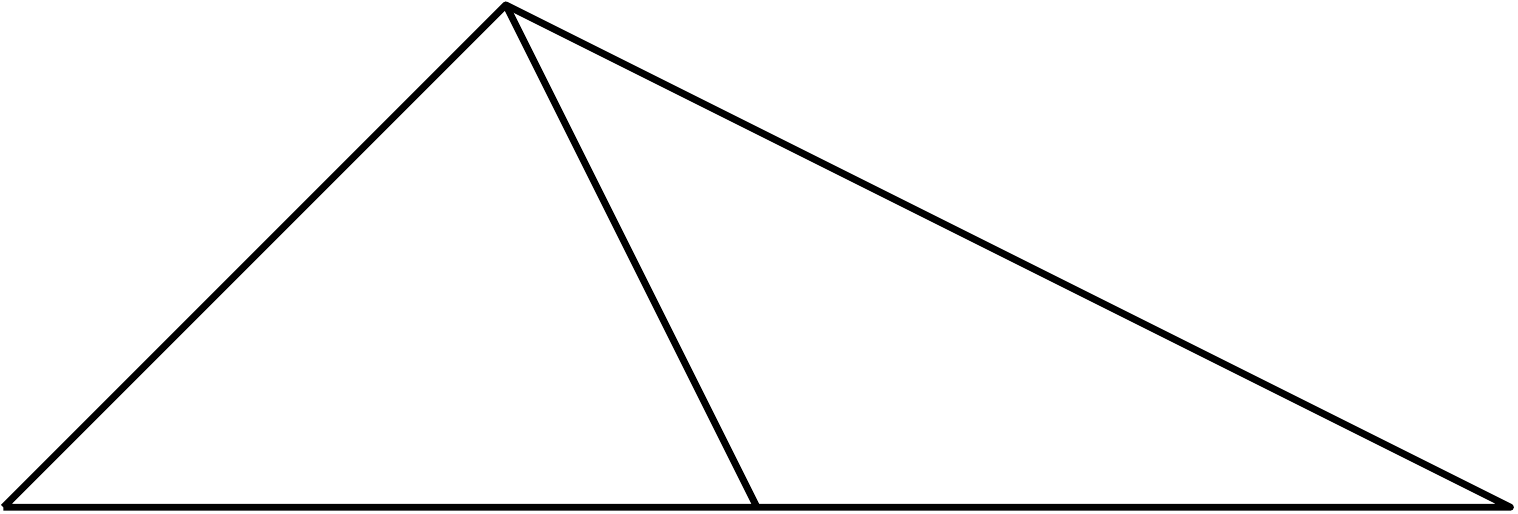}%
\end{minipage}
\begin{minipage}{.19\textwidth}
\centering%
\includegraphics[width=\textwidth]{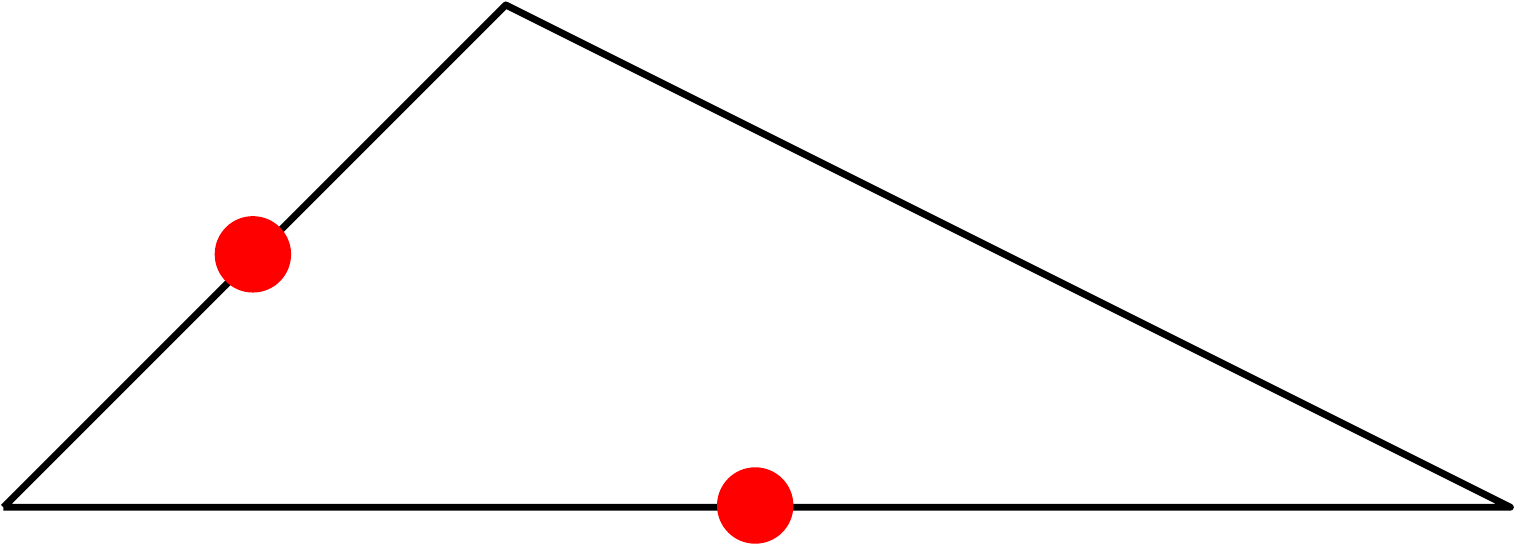}%
\newline%
\includegraphics[width=\textwidth]{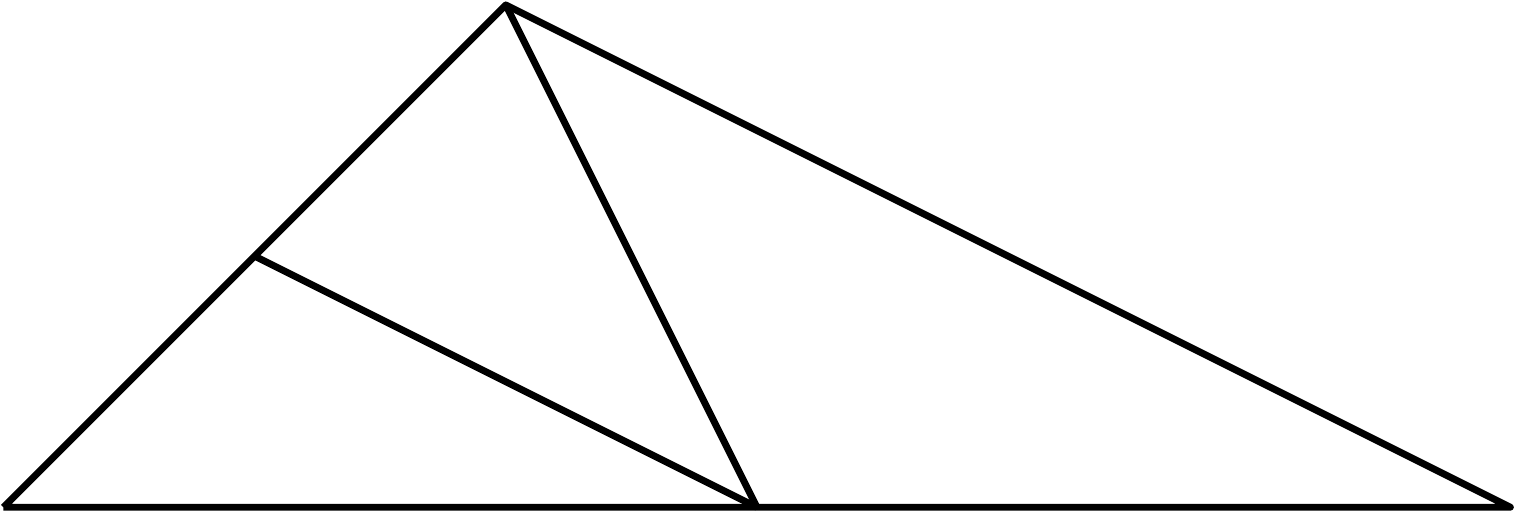}%
\end{minipage}
\begin{minipage}{.19\textwidth}
\centering%
\includegraphics[width=\textwidth]{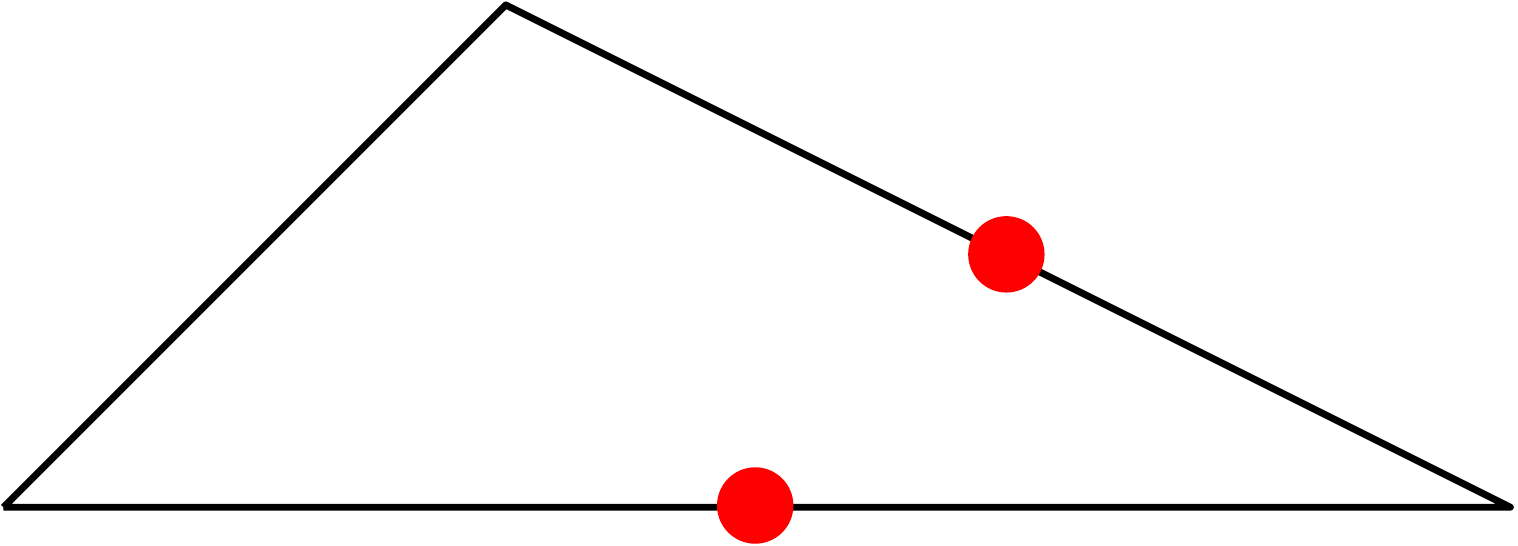}%
\newline%
\includegraphics[width=\textwidth]{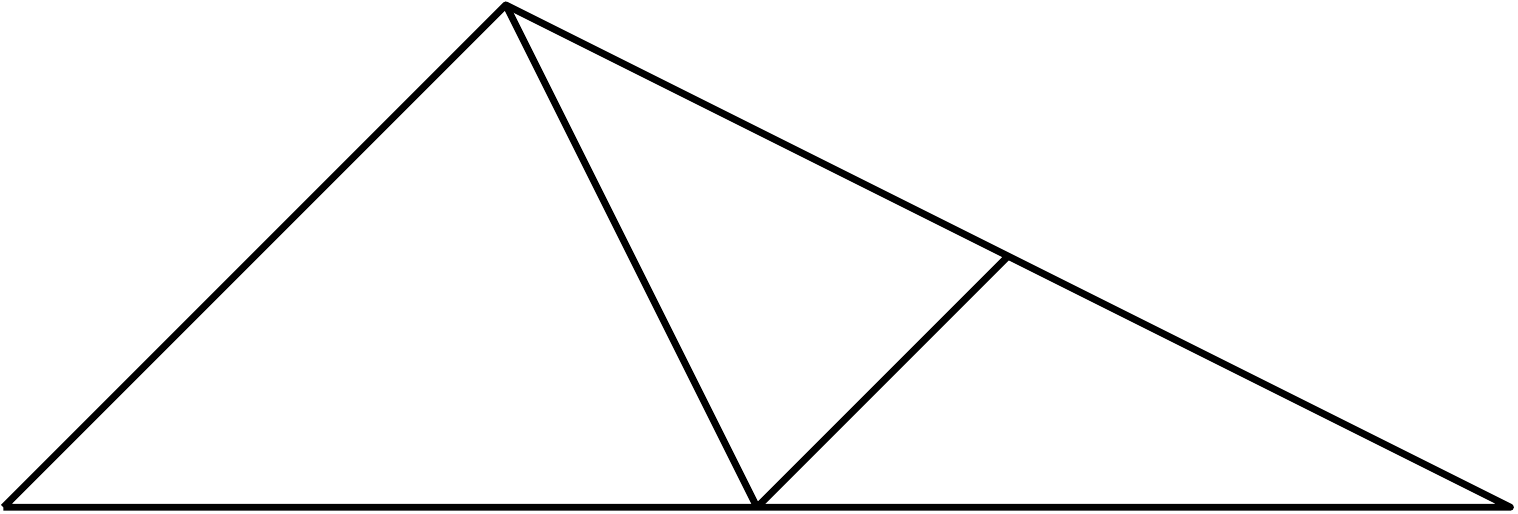}%
\end{minipage}
\begin{minipage}{.19\textwidth}
\centering%
\includegraphics[width=\textwidth]{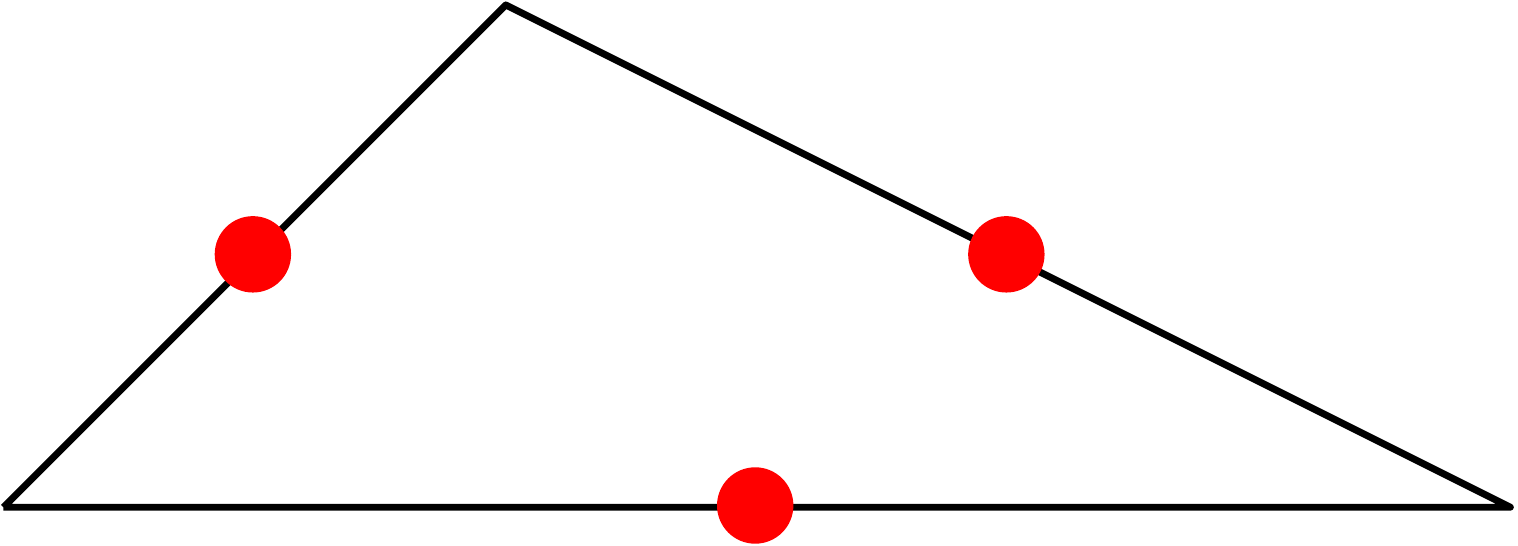}%
\newline%
\includegraphics[width=\textwidth]{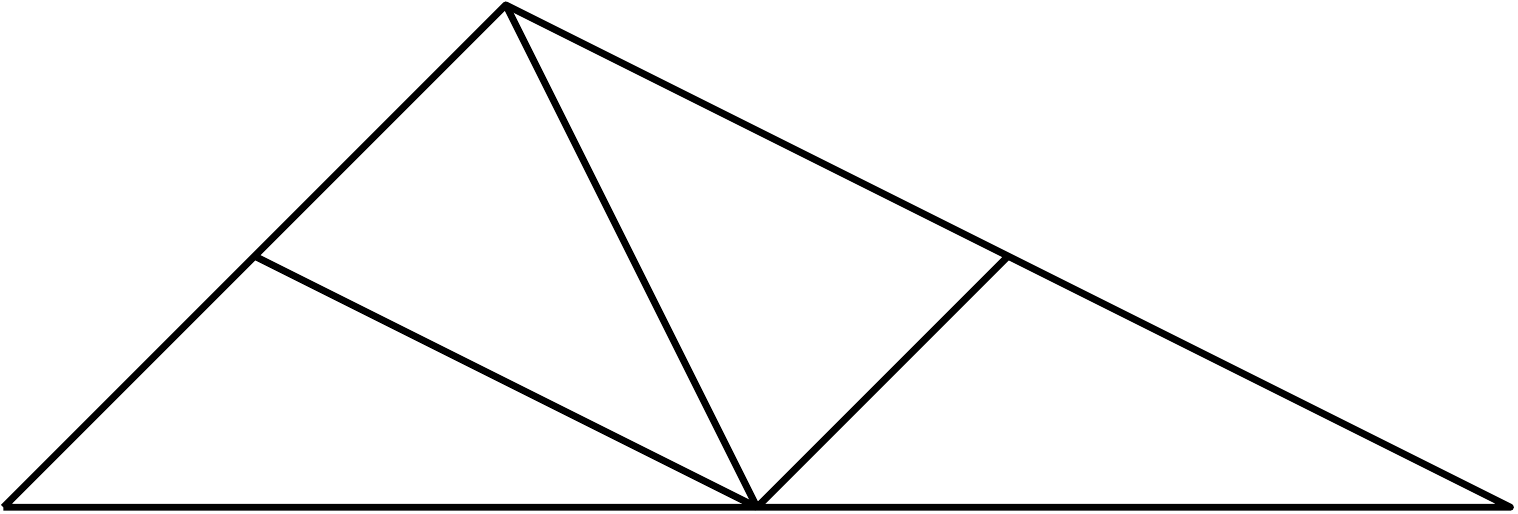}%
\end{minipage}
\caption{Refinement pattern of 2D longest edge bisection:
Coarse-mesh triangles (top row) are refined (bottom row)
by bisection of the edges (at least the longest edge) that
are marked for refinement (top row). The new nodes are the edge midpoints.}
\label{fig:bisection}
\end{figure}

\subsection{Longest edge bisection} 

In what follows, we restrict ourselves to the mesh-refine\-ments performed by 2D longest edge bisection;
see Figure~\ref{fig:bisection}.
Let $\EE_\TT$ be the set of edges of $\TT$. For any set $\MM \subseteq \EE_\TT$, 
the call $\widetilde\TT := \Refine(\TT, \MM)$ returns the coarsest conforming refinement of $\TT$ such that 
all edges $E \in \MM$ are bisected.
In particular, we obtain the uniform refinement $\widehat\TT = \Refine(\TT,\EE_\TT)$ from $\TT$
by three bisections per each element $T \in \TT$.

Let $\mathcal{E}^{\rm{int}}_{\TT} \subset \mathcal{E}_{\TT}$ be the set of \emph{interior} edges,
i.e., $E \in \mathcal{E}^{\rm{int}}_{\TT}$ if and only if there exist two elements $T,T' \in \TT$ such that $E = T\cap T'$.
Then, the above choice of $\widehat\TT$ guarantees
the existence of a one-to-one map between the set $\NN_{\widehat\TT} \backslash \NN_\TT$
of new interior vertices and the set $\mathcal{E}^{\rm{int}}_{\TT}$ of interior edges.
In other words, for any $E \in \mathcal{E}^{\rm{int}}_{\TT}$, there exists a unique $j \in \{1,\ldots,n\}$
such that $\vertexY_j \in \NN_{\widehat\TT} \backslash \NN_\TT$ is the midpoint of $E$.
In this case, we denote by $\varphi_E$ the corresponding hat function (over $\widehat\TT$),
i.e., $\varphi_E := \varphi_j \in \basisY$, where $\varphi_j$ ($j = 1,\ldots,n$)
are defined in \S\ref{sec:two:level:errorestimator}.

\subsection{Local error indicators in the energy norm} 

Consider the primal Galerkin solution $\uXP \in \VXP$ and the associated energy error estimate $\muXP$ given by~\eqref{def:mu}.
We write~\eqref{def:mu} as follows:
\begin{equation} \label{mu:global:contributions}
      \muXP^2 = \mu_{Y\gotP}^2 + \mu_{X\gotQ}^2 = 
      \sum_{E \in \mathcal{E}^{\rm{int}}_{\TT}} \mu^2_{Y\gotP}(E) +  \sum_{\nu \in \gotQ} \mu^2_{X\gotQ}(\nu)
\end{equation}
with the local contributions
\begin{equation} \label{eq:indicators}
      \mu^2_{Y\gotP}(E) :=
      \sum_{\nu \in \gotP} \frac{|F(\varphi_E P_\nu) - B(\uXP,\varphi_E P_\nu)|^2}{\norm{a_0^{1/2}\nabla\varphi_E}{L^2(D)}^2}\quad
      \hbox{and}
      \quad
      \mu_{X\gotQ}(\nu) := \enorm{e_{X\gotQ}^{(\nu)}}{0}
\end{equation}
 (recall that, for each $\nu \in \gotQ$, $e_{X\gotQ}^{(\nu)} \in X \otimes \hull{P_\nu}$ is defined by~\eqref{eq:e_Xnu}).
In explicit terms, $\mu_{Y\gotP}$ (resp., $\mu_{X\gotQ}$) is the global two-level spatial error estimate
(resp., the global parametric error estimate),
$\mu_{Y\gotP}(E)$ denotes the local (spatial) error indicator associated with the edge $E \in \mathcal{E}^{\rm{int}}_{\TT}$,
and $\mu_{X\gotQ}(\nu)$ denotes the individual (parametric) error indicator associated with the index $\nu \in \gotQ$.

A decomposition similar to~\eqref{mu:global:contributions} holds for the error estimate $\zetaXP$
associated with the dual Galerkin solution $\zXP$;
in this case, we denote the corresponding spatial and parametric estimates by $\zeta_{Y\gotP}$ and $\zeta_{X\gotQ}$, respectively;
the definitions of the local contributions $\zeta_{Y\gotP}(E)$ and $\zeta_{X\gotQ}(\nu)$ are analogous to those
of $\mu_{Y\gotP}(E)$ and $\mu_{X\gotQ}(\nu)$ in~\eqref{eq:indicators}.

\subsection{Marking strategy} \label{sec:markstrat}

In order to compute a more accurate Galerkin solution (and, hence, to reduce the error in the quantity of interest),
an enriched approximation space needs to be constructed.
In the algorithm presented below, the approximation space is enriched at
each iteration of the adaptive loop either
by performing local refinement of the underlying triangulation $\TT$ or
by adding new indices into the index set $\gotP$.
In the former case, the refinement is guided by the set $\MM \subseteq \mathcal{E}^{\rm{int}}_{\TT}$ of marked edges,
whereas in the latter case, a set $\gotM \subseteq \gotQ$ of marked indices is added to the index set $\gotP$.

Let us focus on the case of marking edges of triangulation.
We start by using the D\"orfler marking criterion \cite{d96} for the sets 
$\{ \mu_{Y\gotP}(E) \,:\, E \in \mathcal{E}^{\rm{int}}_{\TT} \}$ and
$\{ \zeta_{Y\gotP}(E) \,:\, E \in \mathcal{E}^{\rm{int}}_{\TT} \}$
of (spatial) error indicators (see~\eqref{eq:indicators})
in order to identify two sets of marked edges, independently for the primal and for the dual Galerkin solutions.
Specifically, for a given $0 < \thetaX \leq 1$, we define
\begin{equation*} 
\MM^u := \Doerfler\big( \{ \mu_{Y\gotP}(E) \,:\, E \in \mathcal{E}^{\rm{int}}_{\TT}\}, \thetaX \big) 
\text{ and }
\MM^z := \Doerfler\big( \{ \zeta_{Y\gotP}(E) \,:\, E \in \mathcal{E}^{\rm{int}}_{\TT}\}, \thetaX \big) 
\end{equation*}
as the subsets of $\mathcal{E}^{\rm{int}}_{\TT}$ of minimal cardinality (possibly up to a fixed multiplicative factor) 
such that
\begin{equation*} 
       \thetaX \sum_{E\in\mathcal{E}^{\rm{int}}_{\TT}} \mu^2_{Y\gotP}(E) \leq \sum_{E \in \MM^u} \mu^2_{Y\gotP}(E)
       \quad \text{and} \quad
       \thetaX \sum_{E\in\mathcal{E}^{\rm{int}}_{\TT}} \zeta^2_{Y\gotP}(E) \leq \sum_{E \in \MM^z} \zeta^2_{Y\gotP}(E),
\end{equation*}
respectively.

There exist several strategies of ``combining'' the two sets $\MM^u$ and $\MM^z$ into a single marking set
that is used for refinement in the goal-oriented adaptive algorithm; see~\cite{ms09,bet11,hp16,fpv16}.
For goal-oriented adaptivity in the deterministic setting, \cite{fpv16} proves that the strategies of~\cite{ms09,bet11,fpv16}
lead to convergence with optimal algebraic rates, while the strategy from \cite{hp16} might not.
The marking strategy proposed in~\cite{fpv16} is a modification of the strategy in~\cite{ms09}.
It has been empirically proved that the strategy in~\cite{fpv16} is more effective than the original strategy in~\cite{ms09}
with respect to the overall computational cost.
We employ the following marking strategy adopted from~\cite{fpv16}.
Comparing the cardinality of $\MM^u$ and that of $\MM^z$ we define
\begin{align}
        \MM_\star := \MM^u \quad \hbox{and} \quad \MM^\star := \MM^z \qquad &
        \hbox{if $\#\MM^u \le \#\MM^z$},
        \nonumber
        \\ 
        \MM_\star := \MM^z \quad \hbox{and} \quad \MM^\star := \MM^u \qquad &
        \hbox{otherwise}.
        \nonumber
\end{align}
The set $\MM \subseteq \MM_\star \cup \MM^\star \subseteq \mathcal{E}^{\rm{int}}_{\TT}$ 
is then defined as the union of $\MM_\star$ and 
those $\#\MM_\star$ edges of $\MM^\star$ that have the largest error indicators.
This set $\MM$ of marked edges is the one that is used to guide the local mesh-refinement
in our goal-oriented adaptive algorithm.

In order to identify the set $\gotM \subseteq \gotQ$ of marked indices to be added to the current index set $\gotP$,
we follow the same marking procedure as described above by replacing
$\MM$, $\mathcal{E}^{\rm{int}}_{\TT} $, $E$, $\mu_{Y\gotP}(E)$, $\zeta_{Y\gotP}(E)$, and $\thetaX$
with 
$\gotM$, $\gotQ$, $\nu$, $\mu_{X\gotQ}(\nu)$, $\zeta_{X\gotQ}(\nu)$, and $\thetaP$, respectively
(here, $0 < \thetaP \leq 1$ is a given D\"orfler marking parameter).

\subsection{Mesh-refinement and polynomial enrichment}  \label{sec:refinements}
Let $\MM \subseteq \mathcal{E}^{\rm{int}}_{\TT}$ be a set of marked interior edges and $\widetilde \TT = \Refine(\TT,\MM)$.
We denote by $\RR \subseteq \mathcal{E}^{\rm{int}}_{\TT}$ the set of all edges that are bisected during this refinement,
i.e., $\RR = \mathcal{E}^{\rm{int}}_{\TT} \setminus \mathcal{E}^{\rm{int}}_{\widetilde\TT} \supseteq \MM$.

Since the polynomial space over the parameter domain is fully determined by the associated index set,
the enrichment of the polynomial space is performed 
simply by adding all marked indices $\nu \in \gotM \; \subseteq \; \gotQ$ to the current index set $\gotP$,
i.e., by setting $\widetilde\gotP:= \gotP \cup \gotM$. 

An important feature of the adaptive algorithm presented in the next section
is that it is driven by the estimates of the error reductions
associated with local mesh-refinement and enrichment of the polynomial space on $\Gamma$.
Suppose that the enriched finite-dimensional space is given by
$V_{\widetilde X\widetilde\gotP} \,:=\, \widetilde X \otimes \hull{\set{P_\nu}{\nu \in \widetilde\gotP}}$,
where $\widetilde X := \SS^1_0(\widetilde\TT)$.
Let $u_{\widetilde X \widetilde\gotP} \in V_{\widetilde X \widetilde\gotP}$ be the corresponding Galerkin solution.
We note that Theorem~\ref{theorem:twolevel} applies to $\widehat u_{X\gotP} - u_{X\gotP}$ as well as
to $u_{\widetilde X \widetilde\gotP} - u_{X\gotP}$.
Furthermore, it is important to observe that longest edge bisection ensures that for $E \in \RR$,
the associated hat function $\varphi_E$ is \emph{the same in $\widetilde X$ and $\widehat X$}.
This observation together with the Pythagoras theorem~\eqref{eq:hh2:Pythagoras}
applied to $u_{X\gotP} \in V_{X\gotP}$ and $u_{\widetilde X \widetilde\gotP} \in V_{\widetilde V\widetilde\gotP}$
yield the following estimate of the error reduction
\begin{equation} \label{eq3:theorem:twolevel}
   \bnorm{u - u_{X\gotP}}^2
   - \bnorm{u - u_{\widetilde X \widetilde\gotP}}^2
   \;\reff{eq:hh2:Pythagoras}=\; \bnorm{u_{\widetilde X \widetilde\gotP} - u_{X\gotP}}^2
   \reff{eq1:theorem:twolevel}\simeq \sum_{E \in \RR} \mu_{Y\gotP}^2(E) + \sum_{\nu \in \gotM} \mu_{X\gotQ}^2(\nu).
\end{equation}

\begin{remark}
We note that the estimate~\eqref{eq3:theorem:twolevel} of the error reduction hinges on the mesh-refinement strategy
in the sense that the additional hat functions $\varphi_E$ for $E \in \RR$ must coincide in $\widetilde X$ and $\widehat X$.
As mentioned, this property holds for longest edge bisection as well as for
newest vertex bisection~\cite{steve08,kpp13}, but fails, e.g., for the red-green-blue refinement~\cite{verfuerth}.
\end{remark}%

\subsection{Goal-oriented adaptive algorithm} 

Let us now present a goal-oriented adaptive algorithm for numerical approximation
of $G(u)$, where $u \in V$ is the weak solution to the 
parametric model problem~\eqref{eq:strongform} and
$G$ is the goal functional defined by~\eqref{parametric:goal}.

In the rest of the paper, $\ell \in \N_0$ denotes the iteration counter in the adaptive algorithm
and we use the subscript $\ell$ for triangulations, index sets, Galerkin solutions, error estimates, {\em etc.},
associated with the $\ell$-th iteration of the adaptive loop.
In particular, $V_\ell := V_{X_\ell\gotP_\ell} = X_\ell \otimes \hull{\set{P_\nu}{\nu\in\gotP_\ell}}$
denotes the finite-dimensional subspace of~$V$,
$\ul \in V_\ell$ and $\zl \in V_\ell$ are the primal and dual Galerkin solutions satisfying~\eqref{eq:discreteform}
and~\eqref{parametric:dual:problem:discrete}, respectively, and
$\mu_\ell := \mu_{X_\ell\gotP_\ell}$ and $\zeta_\ell := \zeta_{X_\ell\gotP_\ell}$ are the
associated (global) error estimates (see, e.g.,~\eqref{def:mu}).

\bigskip\hrule\hrule
\begin{algorithm} \label{alg:parametric}
\emph{\bf Goal-oriented adaptive stochastic Galerkin FEM.}
\medskip\hrule\hrule\medskip

\noindent\emph{\textsf{INPUT:}} data $a$, $f$, $g$; initial (coarse) triangulation $\TT_0$, initial index set $\gotP_0$;
marking parameters $0 < \thetaX, \thetaP \leq 1$;
tolerance \emph{$\textsf{tol}$}.
\medskip

\noindent\emph{\textsf{for}} $\ell = 0, 1, 2, \dots$, \emph{\textsf{do}}:
\begin{itemize}
\item[\rm(i)] 
\emph{SOLVE:}
compute $u_\ell, z_\ell \in V_\ell$;
\item[\rm(ii)] 
\emph{ESTIMATE:}
compute four sets of error indicators (see~\eqref{eq:indicators})\\
$\{ \mu_{Y\gotP}(E) : E \in \mathcal{E}^{\rm{int}}_{\TT_\ell} \},
  \{ \mu_{X\gotQ}(\nu) : \nu \in \gotQ_\ell \},
  \{ \zeta_{Y\gotP}(E) : E \in \mathcal{E}^{\rm{int}}_{\TT_\ell} \},
  \{ \zeta_{X\gotQ}(\nu) : \nu \in \gotQ_\ell \}$\\
and two (global) error estimates $\mu_\ell$ and $\zeta_\ell$ (see~\eqref{mu:global:contributions});
\item[\rm(iii)] 
\emph{\textsf{if}} $\mu_\ell\,\zeta_\ell \le \emph{\textsf{tol}}$ \emph{\textsf{then break}}; \emph{\textsf{endif}}
\item[\rm(iv)]
\emph{MARK:} 
use the procedure described in~{\rm \S\ref{sec:markstrat}}
to find the set $\MM_\ell \subseteq \mathcal{E}^{\rm{int}}_{\TT_\ell}$ of marked edges and
the set $\gotM_\ell \subseteq \Ql$ of marked indices.
\item[\rm(v)]
\emph{REFINE:} 
	\begin{itemize}
	\item[\rm(v-a)] 
	Compute two error reduction estimates:
	\begin{subequations}
	\begin{align}
	\rho_{X,\ell}^2
	& := \mu_\ell^2 	\left(\sum_{E \in \RR_\ell} 	\zeta^2_{Y\gotP}(E) \right) +
    	 \zeta_\ell^2 	\left(\sum_{E \in \RR_\ell} 	\mu^2_{Y \gotP}(E) \right),
	\label{eq:rhoX} \\
	\rho_{\gotP,\ell}^2 
	& := \mu_\ell^2 	\left(\sum_{\nu \in \gotM_\ell} \zeta^2_{X\gotQ}(\nu) \right) +
		 \zeta_\ell^2 	\left(\sum_{\nu \in \gotM_\ell}	\mu^2_{X\gotQ}(\nu) \right),
	\label{eq:rhoB}
	\end{align}
	\end{subequations}
	where $\RR_\ell \subseteq \mathcal{E}^{\rm{int}}_{\TT_\ell}$ is 
	 the set of all edges to be bisected if $\TT_\ell$ is refined (see~{\rm \S\ref{sec:refinements}}).
	\item[\rm(v-b)] 
	\emph{\textsf{if}} $\rho_{X,\ell} \ge \rho_{\gotP,\ell}$ \emph{\textsf{then}}
	\begin{itemize}
	\item[] 
	define $\TT_{\ell+1} = \Refine(\TT_\ell,\MM_\ell)$ and $\gotP_{\ell+1} = \gotP_\ell$
	(i.e., refine the spatial triangulation $\TT_\ell$ and keep the index set $\gotP_\ell$);
	\end{itemize}
	
	\emph{\textsf{else}}
	\begin{itemize}
	\item[] 
	define $\TT_{\ell+1} = \TT_\ell$ and $\gotP_{\ell+1} = \gotP_\ell \cup \gotM_\ell$ 
	(i.e., keep the spatial triangulation $\TT_\ell$ and enlarge the index set $\gotP_\ell$).
	\end{itemize}
	\emph{\textsf{endif}}
	\end{itemize}
\end{itemize}
\noindent\emph{\textsf{endfor}}
\medskip

\noindent\emph{\textsf{OUTPUT:}} sequences of nested 
triangulations $\{ \TT_\ell \}$, increasing index sets $\{\gotP_\ell\}$,
primal and dual Galerkin solutions $\{ u_\ell,\, z_\ell \}$,
the corresponding energy error estimates $\{ \mu_\ell,\,\zeta_\ell \}$, and
the estimates $\{ \mu_\ell\,\zeta_\ell \}$ of the error in approximating~$G(u)$. 
\medskip\hrule\hrule
\end{algorithm}
\medskip

Let us give a motivation behind Step~{\rm(v)} in Algorithm~\ref{alg:parametric}.
This relies on the fact that the algorithm employs the product of energy errors $\bnorm{u-u_\ell}\,\bnorm{z-z_\ell}$
in order to control the error in approximating~$G(u)$; see~\eqref{eq:errorInequality}.

Let $\widetilde{V_\ell} \supset V_\ell$ be an enrichment of $V_{\ell}$
(e.g., $\widetilde{V_\ell} = X_{\ell+1} \otimes \hull{\set{P_\nu}{\nu\in\gotP_\ell}}$
or
$\widetilde{V_\ell} = X_{\ell} \otimes \hull{\set{P_\nu}{\nu\in\gotP_{\ell+1}}}$).
Let $\uhatl, \zhatl \in \widetilde{V_\ell}$ denote the enhanced primal and dual Galerkin solutions.
One has (see~\eqref{eq:hh2:Pythagoras})
\[
  \bnorm{u - \uhatl}^2 = \bnorm{u - \ul}^2 - \bnorm{\ul - \uhatl}^2
  \quad \text{and} \quad
  \bnorm{z - \zhatl}^2 = \bnorm{z - \zl}^2 - \bnorm{\zl - \zhatl}^2.
\]
Hence,
\begin{equation} \label{goal:error:reduction}
       \begin{split}
       \bnorm{u - \uhatl}^2 \, \bnorm{z - \zhatl}^2 
	& = \bnorm{u - \ul}^2 \,\bnorm{z - \zl}^2 + \bnorm{\ul - \uhatl}^2 \bnorm{\zl - \zhatl}^2\\ 
	& \quad\;- \left( \bnorm{u - \ul}^2 \,\bnorm{\zl - \zhatl}^2 + \bnorm{\ul - \uhatl}^2 \, \bnorm{z - \zl}^2\right).
       \end{split}
\end{equation}
Equality \eqref{goal:error:reduction} shows that the quantity
\begin{equation} \label{goal:estimator:product:reduction}
       \bnorm{u - \ul}^2 \,\bnorm{\zl - \zhatl}^2 +  \bnorm{\ul - \uhatl}^2\,  \bnorm{z - \zl}^2 
\end{equation}
provides a good approximation of the reduction in the product of energy errors
that would be achieved due to enrichment of the approximation space.
In fact, the true reduction in the product of energy errors also includes the term
$-\bnorm{\ul - \uhatl}^2 \bnorm{\zl - \zhatl}^2$ (see~\eqref{goal:error:reduction}).
This term (in absolute value) is normally much smaller compared to the sum in~\eqref{goal:estimator:product:reduction}
and may thus be neglected.

Now recall that Theorem~\ref{theorem:twolevel} provides computable estimates of the energy errors
(see~\eqref{eq1:theorem:twolevel})
and of the energy error reductions 
(see~\eqref{eq3:theorem:twolevel}).
Using these results to bound each term in~\eqref{goal:estimator:product:reduction}, we obtain
the estimate of reduction in the product of energy errors.
In particular, the reduction due to mesh-refinement (by bisection of all edges in $\RR_\ell$, see~\S\ref{sec:refinements})
is estimated by $\rho^2_{X,\ell}$ defined in~\eqref{eq:rhoX}, i.e.,
\[
  \bnorm{u - \ul}^2 \,\bnorm{\zl - \zhatl}^2 + 
  \bnorm{z - \zl}^2 \,\bnorm{\ul - \uhatl}^2
  \simeq
  \rho^2_{X,\ell}.
\]
Similarly, the reduction due to polynomial enrichment (by adding the set $\gotM_\ell$ of marked indices)
is estimated by $\rho^2_{\gotP,\ell}$ defined in~\eqref{eq:rhoB}.
Thus, by comparing these two estimates ($\rho_{X,\ell}$ and $\rho_{\gotP,\ell}$),
the adaptive algorithm chooses the enrichment of $V_\ell$
(either mesh-refinement or polynomial enrichment)
that corresponds to a larger estimate of the associated error reduction
(see step~{\rm(v-b)} in Algorithm~\ref{alg:parametric}).

\section{Numerical experiments}	\label{sec:numerical:examples}
\noindent
In this section, we report the results of some numerical experiments that demonstrate the performance
of the goal-oriented adaptive algorithm described in Section~\ref{sec:goaloriented:adaptivity}
for parametric model problems.
All experiments were performed using the open source MATLAB toolbox Stochastic T-IFISS~\cite{BespalovR_stoch_tifiss}
on a desktop computer equipped with an Intel Core CPU i5-4590@3.30GHz and 8.00GB RAM.

\subsection{Outline of the experiments} \label{sec:experiments:outline}
Staying within the framework of the parametric model problem~\eqref{eq:strongform}
and the goal functional~\eqref{parametric:goal},
we use the representations of $f$ and $g$ as introduced in~\cite{ms09}
(see also~\cite[Section~4]{fpv16}) to define the corresponding right-hand side functionals
$F(v)$ and $G(v)$ in~\eqref{eq:weakform} and~\eqref{parametric:dual:problem}, respectively.
Specifically, let $f_i, g_i \in L^2(D)$ ($i=0,1,2$) and set $\ff = (f_1,f_2)$ and $\gg = (g_1,g_2)$.
Define
\begin{equation} \label{rhs:primal:problem}
      F(v) = \int_\Gamma \int_D f_0(x)v(x,\y) \, dx \, \dpiy - \int_\Gamma \int_D \ff(x) \cdot \nabla v(x,\y) \, dx \, \dpiy 
      \quad \text{for all $v \in V$}
\end{equation}
and 
\begin{equation} \label{rhs:dual:problem}
      G(v) = \int_\Gamma \int_D g_0(x) v(x,\y) \, dx \, \dpiy - \int_\Gamma \int_D \gg(x) \cdot \nabla v(x,\y) \, dx \, \dpiy
      \quad \text{for all $v \in V$}.
\end{equation}
The motivation behind these representations is to introduce different non-geometric singularities in the primal and dual solutions.
In the context of goal-oriented adaptivity, this emphasizes the need for separate marking to resolve singularities
in both solutions in different regions of the computational domain.

In all experiments, we run Algorithm~\ref{alg:parametric} with the initial index set
\begin{equation*} 
      \gotP_0 := \{ (0,0,0,\dots), \, (1,0,0,\dots) \}
\end{equation*}
and collect the following output data:
\begin{itemize}
\item
the number of iterations $L = L (\mathsf{tol})$ needed to reach the prescribed tolerance~$\mathsf{tol}$;
\item
the final goal-oriented error estimate $\mu_{L} \zeta_{L}$;
\item 
the overall computational time $t$;
\item
the overall computational ``cost''
\[
N_{\rm total} := \sum_{\ell=0}^L \dim(V_\ell), 
\]
which reflects the total amount of work in the adaptive process;
\item
the final number of degrees of freedom
\[
  N_L := \dim (V_L) 
  = \dim\left(X_L\right) \, \#\gotP_L 
  = \#\NN_L \, \#\gotP_L,
\]
where $\NN_L$ denotes the set of interior vertices of $\TT_L$;
\item 
the number of elements $\#\TT_L$ and the number of interior vertices $\#\NN_L$ of the final triangulation $\TT_L$;
\item
the cardinality of the final index set $\gotP_L$ and
the number of active parameters in~$\gotP_L$, denoted by $M^\mathrm{active}_L$; 
\item
the evolution of the index set, i.e., $\{\gotP_\ell \,:\, \ell = 0,1,\ldots,L\}$.
\end{itemize}
In order to test the effectiveness of our goal-oriented error estimation,
we compare the product $\mu_\ell\zeta_\ell$ with the reference error $|G(\uref) - G(u_\ell)|$,
where $\uref \in V_{\rm ref}~:= X_{\rm ref}\otimes \hull{\set{P_\nu}{\nu\in\gotP_{\rm ref}}}$
is an accurate primal solution.
In order to compute $\uref$ we employ quadratic (P2) finite element approximations over a fine triangulation $\TT_{\rm ref}$
and use a large index set $\gotP_{\rm ref}$
($\TT_{\rm ref}$ and $\gotP_{\rm ref}$ are to be specified in each experiment).
%
Then, the effectivity indices are computed as follows:
\begin{equation} \label{eff:indices}
      \Theta_\ell := \frac{\mu_\ell\,\zeta_\ell}{\lvert G(\uref) - G(u_\ell) \lvert}, 
      \qquad \ell = 0,\dots,L.
\end{equation}

\begin{figure}[b!]
\centering
\footnotesize
	\includegraphics[scale=0.36]{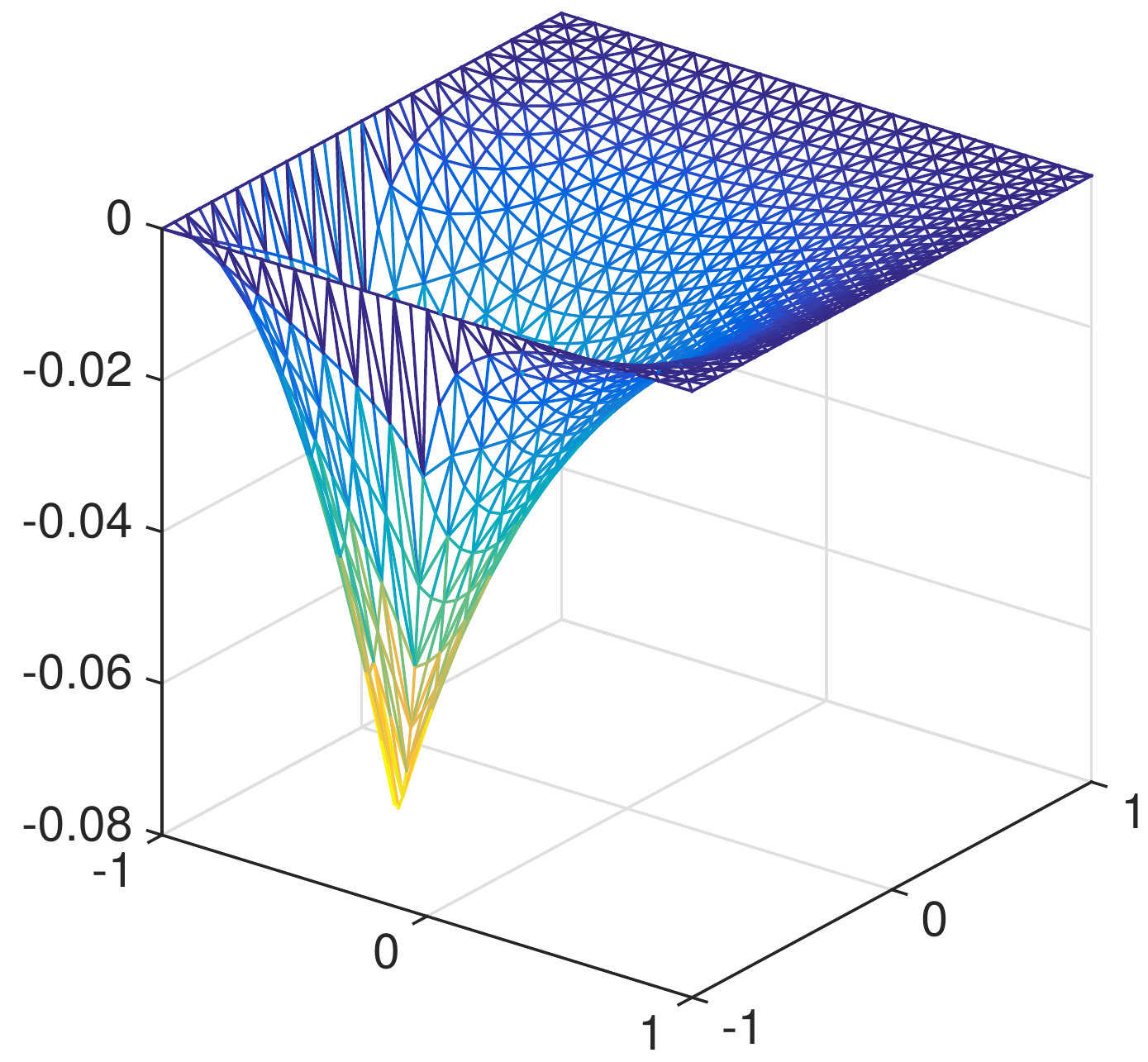} 
	\hspace{0.5cm}
	\includegraphics[scale=0.36]{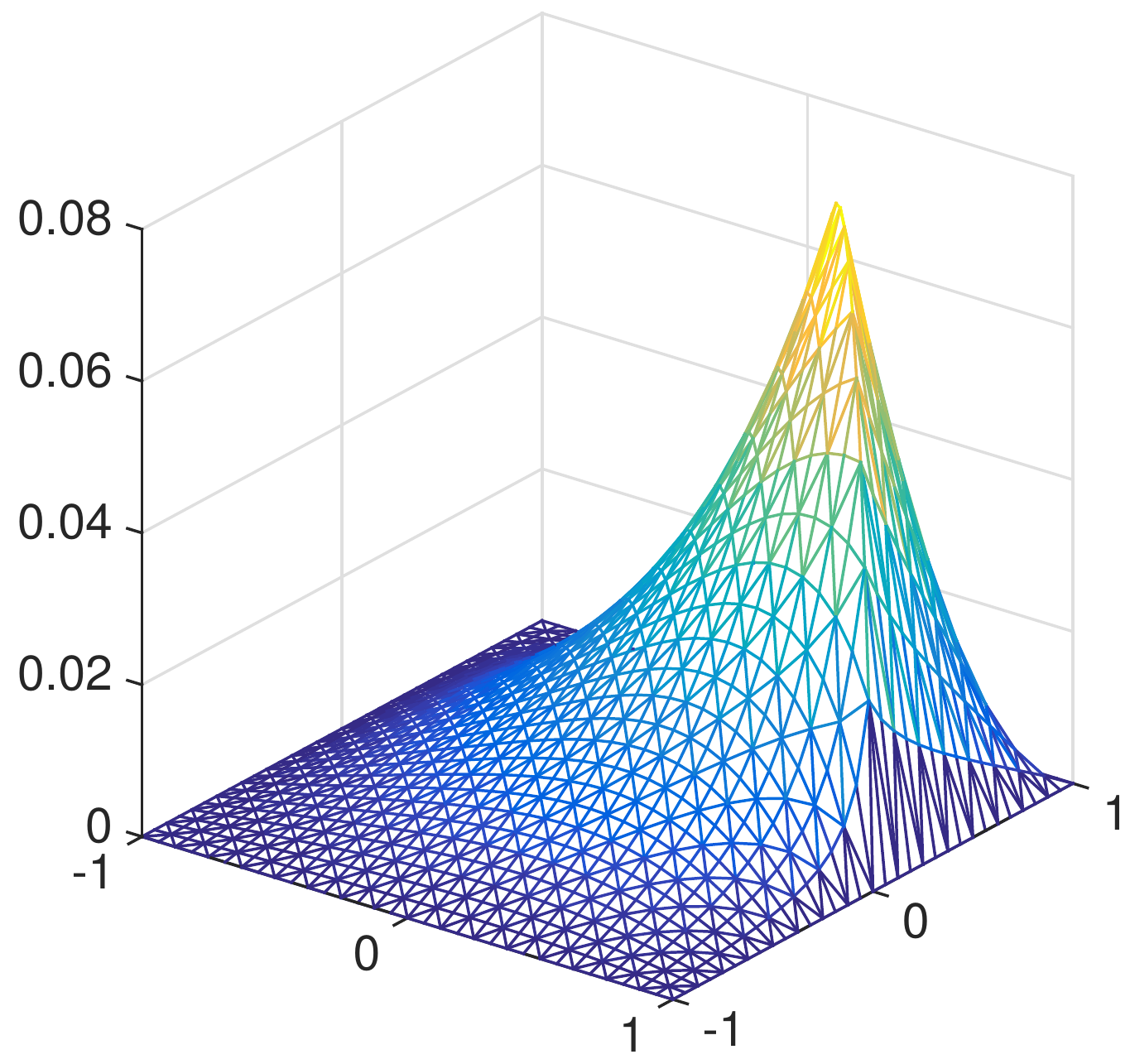}
\caption{
The mean fields of primal (left) and dual (right) Galerkin solutions in Experiment~1. 
}
\label{Ex1:primal:dual:sol}
\end{figure}
\subsection{Experiment~1} 
In the first experiment, we demonstrate the performance of Algorithm~\ref{alg:parametric}
for the parametric model problem \eqref{eq:strongform} posed on the square domain $D=(-1,1)^2$. 
Suppose that the coefficient $a(x,\y)$ in~\eqref{eq:strongform} is a parametric representation
of a second-order random field with prescribed mean $\E[a]$ and covariance function~$\Cov[a]$.
We assume that $\Cov[a]$ is the separable exponential covariance function given by 
\begin{equation*} 
       \Cov[a](x,x') = \sigma^2 \exp\left(-\frac{|x_1 - x_1'|}{l_1} - \frac{|x_2 - x_2'|}{l_2}\right), 
\end{equation*}
where $x=(x_1,x_2) \in D,\ x'=(x_1',x_2') \in D$, $\sigma$ denotes the standard deviation, 
and $l_1, l_2$ are correlation lengths.
In this case, $a(x,\y)$ can be expressed using the Karhunen--L\`oeve expansion 
\begin{equation} \label{random:field:KL}
a(x,\y) = \E[a](x) + c \, \sigma \sum_{m=1}^\infty y_m\,\sqrt{\lambda_m}\,\varphi_m(x), 
\end{equation}
where $\{(\lambda_m, \varphi_m)\}_{m=1}^\infty$ are the eigenpairs of the integral operator 
$\int_D \Cov[a](x,x') \varphi(x')dx'$, 
$y_m$ are the images of pairwise uncorrelated mean-zero random variables,
and the constant $c > 0$ 
is chosen such that ${\rm Var}(c \, y_m)=1$ for all $m \in \N$.
Note that analytical expressions for $\lambda_m$ and $\varphi_m$ exist in the one-dimensional case 
(see, e.g., \cite[pages~28--29]{gs91}); as a consequence, the formulas for rectangular domains follow by tensorization. 
In this experiment, we assume that $y_m$ 
are the images of independent mean-zero 
random variables on $\Gamma_m = [-1,1]$ that have a ``truncated'' Gaussian density:
\begin{equation} \label{trunc:gaussian:density}
\rho(y_m) = (2\Phi(1) - 1)^{-1}\left(\sqrt{2\pi}\right)^{-1}\exp\left(-y_m^2 / 2\right)
\quad\text{for all } m \in \N,
\end{equation}
where $\Phi(\cdot)$ is the Gaussian cumulative distribution function
(in this case, $c \approx 1.8534$ in~\eqref{random:field:KL}).
Thus, in order to construct a polynomial space on $\Gamma$
we employ the set of orthonormal polynomials 
generated by the probability density function~\eqref{trunc:gaussian:density}
and satisfying the three-term recurrence \eqref{three:term:rec}.
These polynomials are known as Rys polynomials; see, e.g., \cite[Example~1.11]{g04}.

We test the performance of Algorithm~\ref{alg:parametric} by considering a parametric version of Example~7.3 in \cite{ms09}.
Specifically, let $f_0 = g_0 = 0$, $\ff =(\rchi_{T_f}, 0)$, and $\gg = (\rchi_{T_g},0)$, 
where $\rchi_{T_f}$ and $\rchi_{T_g}$ denote the characteristic functions of the triangles
\[
  T_f:=\text{conv}\{(-1,-1), (0,-1), (-1,0)\}\quad \hbox{and}\quad T_g := \text{conv}\{(1,1), (0,1), (1,0)\},
\]
respectively (see Figure~\ref{Ex1:meshes} (left)).
Then, the functionals $F$ and $G$ in \eqref{rhs:primal:problem}--\eqref{rhs:dual:problem} read as
\begin{equation*}
          F(v) = 
          -\int_\Gamma \int_{T_f} \frac{\partial v}{\partial x_1}(x,\y) \, dx \, \dpiy, 
\ \
          G(v) = 
          -\int_\Gamma \int_{T_g} \frac{\partial v}{\partial x_1}(x,\y) \, dx \, \dpiy	
\ \ \text{for all } v \in V.
\end{equation*}

\begin{figure}[t!]
\centering
\footnotesize
\includegraphics[scale=0.29]{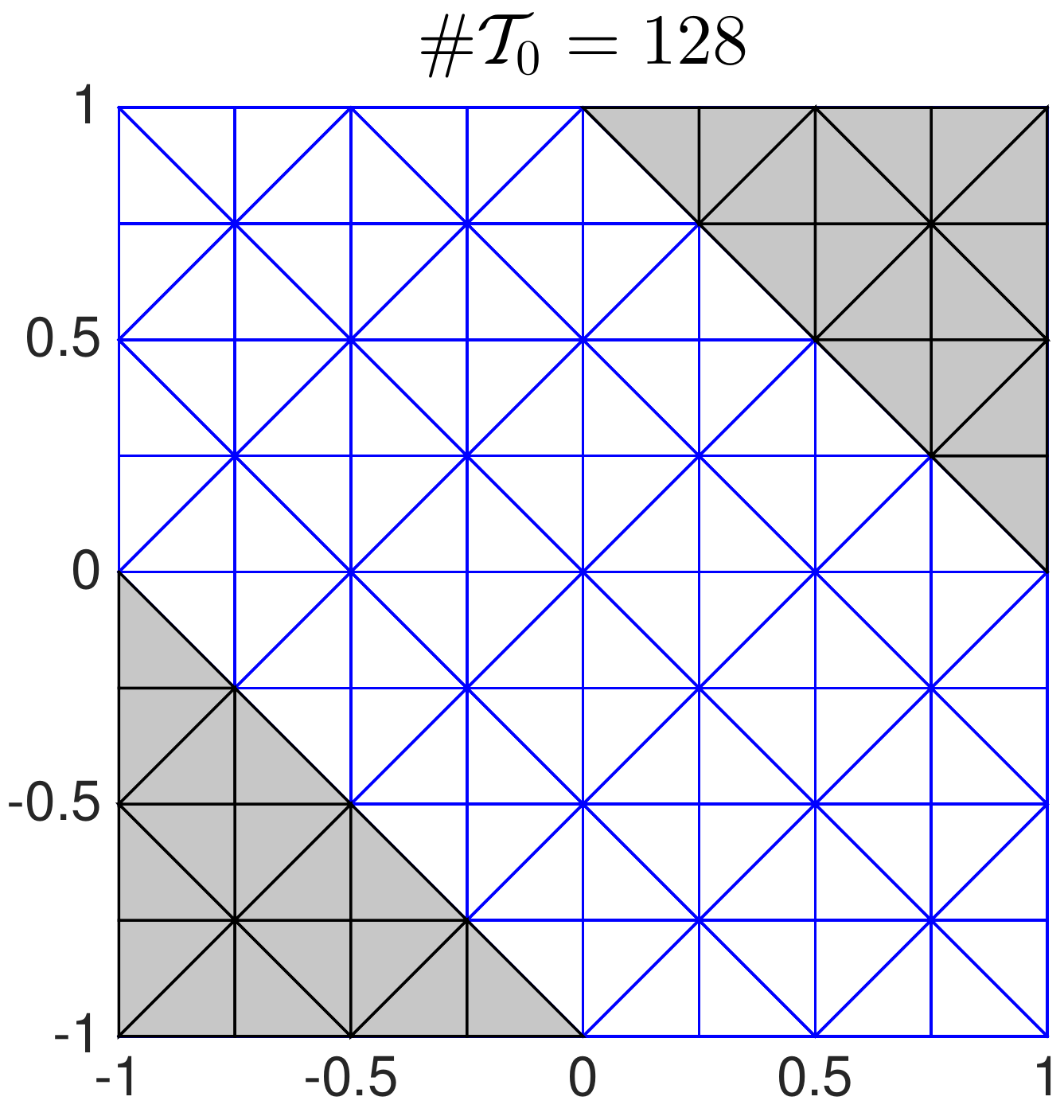} 	\hspace{0.2cm}  
\includegraphics[scale=0.29]{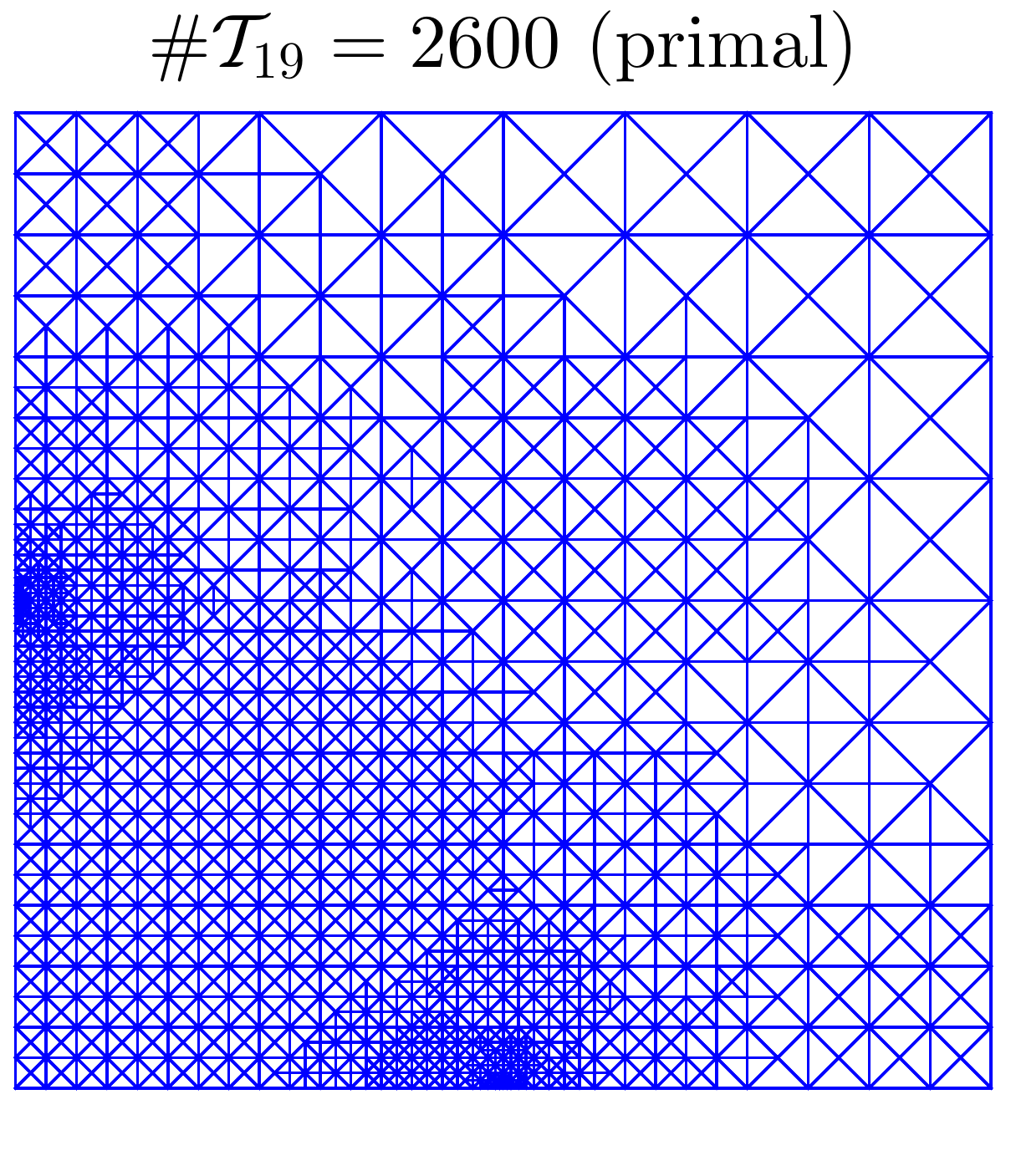}	\hspace{0.2cm}
\includegraphics[scale=0.29]{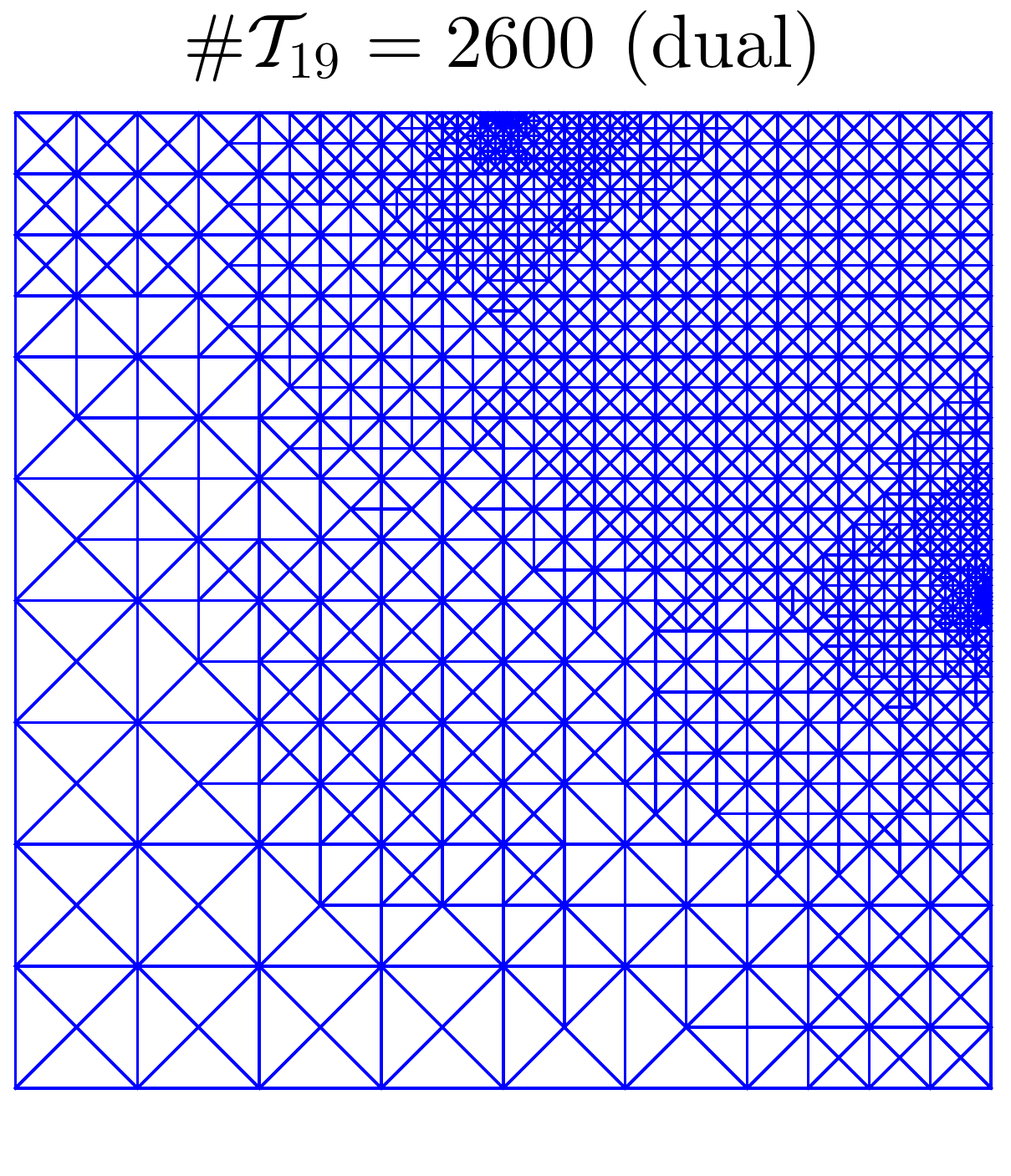} \hspace{0.2cm} 
\includegraphics[scale=0.29]{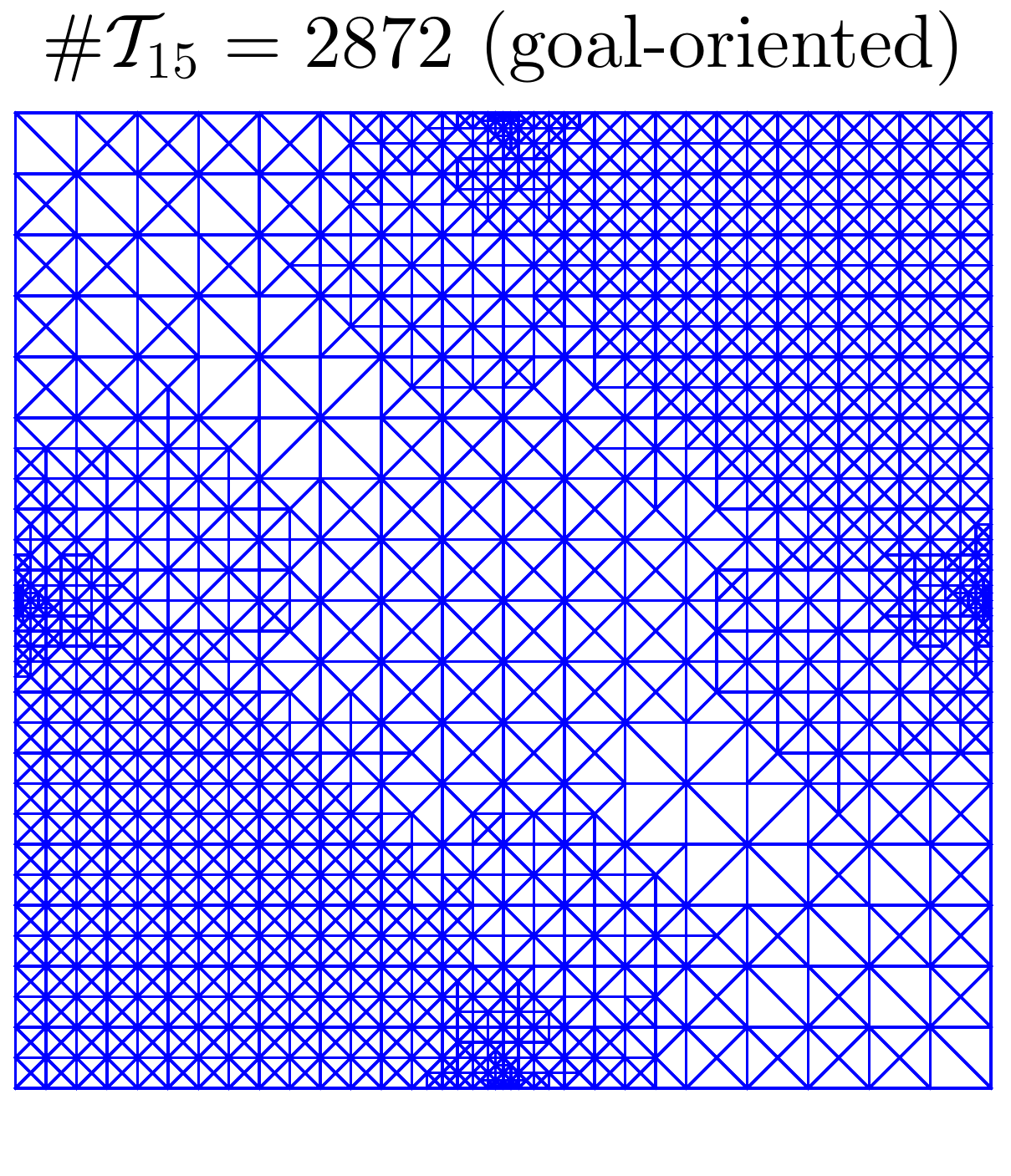} 
\caption{
Experiment~1: Initial triangulation $\TT_0$ with shaded triangles $T_f$ and $T_g$ (left plot);
triangulations generated by the standard adaptive sGFEM algorithm with spatial refinements 
driven either by the error estimates $\mu_\ell$ or
by the error estimates $\zeta_\ell$ (two middle plots); 
triangulation generated by the goal-oriented adaptive algorithm (right plot).
}
\label{Ex1:meshes}
\end{figure}

Setting $\sigma = 0.15$, $l_1 = l_2 = 2.0$, and $\E[a](x) = 2$ for all $x \in D$, 
we compare the performance of Algorithm~\ref{alg:parametric} for different input values of the marking 
parameter $\thetaX$ 
as well as the parameter $\overline{M}$ in~\eqref{finite:index:set:Q}.
More precisely, we consider two sets of marking parameters:
(i) $\thetaX = 0.5$, $\thetaP=0.9$; \ (ii) $\thetaX = 0.25$, $\thetaP=0.9$.
In each case, we run Algorithm~\ref{alg:parametric} with $\overline{M}=1$ and $\overline{M}=2$. 
The same stopping tolerance is set to $\mathsf{tol}=7$e-$6$ in all four computations. 

Figure~\ref{Ex1:primal:dual:sol} (left) shows the mean field of the primal Galerkin solution 
exhibiting a singularity along the line connecting the points $(-1,0)$ and $(0,-1)$.
Similarly, the mean field of the dual Galerkin solution in Figure~\ref{Ex1:primal:dual:sol} (right)
exhibits a singularity along the line connecting the points $(1,0)$ and $(0,1)$. 

Figure~\ref{Ex1:meshes} (left plot) shows the initial triangulation $\TT_0$ used in this experiment.
The two middle plots in Figure~\ref{Ex1:meshes} depict the refined triangulations generated
by an adaptive sGFEM algorithm with spatial refinements
driven either solely by the estimates $\mu_\ell$ for the error in the primal Galerkin solution 
or solely by the estimates $\zeta_\ell$ for the error in the dual Galerkin solution. 
The right plot in Figure~\ref{Ex1:meshes} shows the triangulation produced by Algorithm~\ref{alg:parametric}. 
As expected, this triangulation simultaneously captures spatial features of primal and dual solutions. 

\begin{table}[t!]
\setlength\tabcolsep{12.5pt} 
\begin{center} 
\smallfontthree{
\renewcommand{\arraystretch}{1.4}
\begin{tabular}{r !{\vrule width 1.0pt} c | c !{\vrule width 1.0pt} c | c }
\noalign{\hrule height 1.0pt}
						&\multicolumn{2}{c!{\vrule width 1.0pt}}{case~(i): $\thetaX=0.5$, $\thetaP=0.9$}
						&\multicolumn{2}{c}{case~(ii): $\thetaX=0.25$, $\thetaP=0.9$}\\[2pt]
\cline{2-5} 
						&\multicolumn{1}{c|}{$\overline{M}=1$} 
						&\multicolumn{1}{c!{\vrule width 1.0pt}}{$\overline{M}=2$}
						&\multicolumn{1}{c|}{$\overline{M}=1$} 
						&\multicolumn{1}{c}{$\overline{M}=2$}\\
\hline				 							
$L$						&\multicolumn{1}{c|}{$27$}		
						&\multicolumn{1}{c!{\vrule width 1.0pt}}{$25$}
						&\multicolumn{1}{c|}{$43$}
						&\multicolumn{1}{c}{$42$}\\[-3pt]

$\mu_L\zeta_L$			&\multicolumn{1}{c|}{$5.7152\text{e-}06$}	
						&\multicolumn{1}{c!{\vrule width 1.0pt}}{$6.6884\text{e-}06$}
						&\multicolumn{1}{c|}{$6.4015\text{e-}06$}
						&\multicolumn{1}{c}{$6.4406\text{e-}06$}\\[-3pt]

$t$ (sec)	 			&\multicolumn{1}{c|}{$307$}	
						&\multicolumn{1}{c!{\vrule width 1.0pt}}{$312$}
						&\multicolumn{1}{c|}{$323$}
						&\multicolumn{1}{c}{$468$}\\[-3pt]

$N_{\rm total}$			&\multicolumn{1}{c|}{$2,825,160$}	
						&\multicolumn{1}{c!{\vrule width 1.0pt}}{$3,201,507$}
						&\multicolumn{1}{c|}{$3,433,577$}
						&\multicolumn{1}{c}{$5,702,589$}\\[-3pt]
														
$N_L$					&\multicolumn{1}{c|}{$710,467$}
						&\multicolumn{1}{c!{\vrule width 1.0pt}}{$786,390$}
						&\multicolumn{1}{c|}{$552,442$}
						&\multicolumn{1}{c}{$979,017$}\\[-3pt]
									
$\#\TT_L$				&\multicolumn{1}{c|}{$75,568$}
						&\multicolumn{1}{c!{\vrule width 1.0pt}}{$53,044$}
						&\multicolumn{1}{c|}{$50,808$}
						&\multicolumn{1}{c}{$50,792$}\\[-3pt]
						
$\#\NN_L$				&\multicolumn{1}{c|}{$37,393$}
						&\multicolumn{1}{c!{\vrule width 1.0pt}}{$26,213$}
						&\multicolumn{1}{c|}{$25,111$}
						&\multicolumn{1}{c}{$25,103$}\\[-3pt]
				
$\#\gotP_L$ 			&\multicolumn{1}{c|}{$19$}
						&\multicolumn{1}{c!{\vrule width 1.0pt}}{$30$}
						&\multicolumn{1}{c|}{$22$}
						&\multicolumn{1}{c}{$39$}\\[-3pt]						
				
$M^\mathrm{active}_L$	&\multicolumn{1}{c|}{$8$}
						&\multicolumn{1}{c!{\vrule width 1.0pt}}{$15$}
						&\multicolumn{1}{c|}{$10$}
						&\multicolumn{1}{c}{$17$}\\[2pt]
\noalign{\hrule height 1.0pt}
\end{tabular}
\vspace{8pt}
\caption{
The outputs obtained by running Algorithm~\ref{alg:parametric}
with $\thetaX = 0.5$, $\thetaP = 0.9$ (case (i)) and $\thetaX = 0.25$, $\thetaP = 0.9$ (case~(ii)) 
in Experiment~1.
}
\label{Ex1:table:i:ii}
}
\end{center}                                                                   
\end{table}

\begin{table}[t!]
\setlength\tabcolsep{10pt} 
\begin{center} 
\smallfontthree{
\renewcommand{\arraystretch}{1.4}
\begin{tabular}{r !{\vrule width 1.0pt} c c !{\vrule width 1.0pt} c c }
\noalign{\hrule height 1.0pt}
&\multicolumn{4}{c}{case~(i): $\thetaX = 0.5$, $\thetaP = 0.9$}\\[2pt]
\cline{2-5}
&\multicolumn{2}{c!{\vrule width 1.0pt}}{$\overline{M} = 1$} 
&\multicolumn{2}{c}{$\overline{M} = 2$}\\
\hline
$\gotP_\ell$	
&\quad$\ell=0$	&$(0\ 0)$\quad				&$\ell=0$	&$(0\ 0)$	\\[-4pt]	
&			&$(1\ 0)$						&			&$(1\ 0)$	\\
					
&\quad$\ell=11$	&$\eps{2}{2}$				&$\ell=10$	&$\eps{3}{3}$	\\[-4pt]
&			&								&			&$\eps{2}{3}$	\\

&\quad$\ell=12$	&$\eps{3}{3}$				&$\ell=15$	&$\eps{5}{5}$	\\[-4pt]
&			&								&			&$\eps{4}{5}$	\\
		
&\quad$\ell=17$	&$\eps{4}{4}$				&$\ell=17$	&$\eps{7}{7}$	\\[-4pt]
&			&$(1\ 0\ 1\ 0)$					&			&$\eps{6}{7}$	\\[-4pt]
&			&$(2\ 0\ 0\ 0)$					&			&$(1\ 0\ 1\ \zero{4})$\\[-4pt]
&			&								&			&$(2\ \zero{6})$\\

&\quad$\ell=18$	&$\eps{5}{5}$				&$\ell=20$	&$\eps{9}{9}$\\[-4pt]
&			&								&			&$\eps{8}{9}$\\[-4pt]
&			&								&			&$(1\ 1\ \zero{7})$\\

&\quad$\ell=22$	&$\eps{6}{6}$				&$\ell=22$	&$\eps{11}{11}$\\[-4pt]
&			&$(1\ 1\ 0\ 0\ 0\ 0)$			&			&$\eps{10}{11}$\\[-4pt]
&			&								&			&$(0\ 1\ 1\ \zero{8})$	\\[-4pt]
&\quad$\ell=23$	&$\eps{7}{7}$				&			&$(1\ \zero{4}\ 1\ \zero{5})$\\[-4pt]
&			&$(1\ 0\ 0\ 1\ 0\ 0\ 0)$		&			&$(1\ \zero{3}\ 1\ \zero{6})$\\[-4pt]
&			&								&			&$(1\ \zero{2}\ 1\ \zero{7})$\\

&\quad$\ell=24$	&$\eps{8}{8}$				&$\ell=24$	&$\eps{13}{13}$\\[-4pt]	
&			&$(1\ 0\ 0\ 0\ 1\ 0\ 0\ 0)$		&			&$\eps{12}{13}$\\[-4pt]
&			&								&			&$(0\ 0\ 2\ \zero{10})$\\[-4pt]
&\quad$\ell=26$	&$\eps{9}{9}$				&			&$(0\ 2\ 0\ \zero{10})$\\[-4pt]
&			&$(0\ 1\ 1\ \zero{6})$			&			&$(1\ \zero{6}\ 1\ \zero{5})$\\[-4pt]
&			&$(1\ \zero{6}\ 1\ 0)$			&			&$(1\ \zero{5}\ 1\ \zero{6})$\\[-4pt]
&			&$(1\ \zero{5}\ 1\ 0\ 0)$		&			&\\[-4pt]
&			&$(1\ \zero{4}\ 1\ 0\ 0\ 0)$	&$\ell=25$	&$\eps{15}{15}$\\[-4pt]
&			&								&			&$\eps{14}{15}$\\[-4pt]
&			&								&			&$(0\ 1\ 0\ 1\ \zero{11})$\\[-4pt]
&			&								&			&$(1\ \zero{8}\ 1\ \zero{5})$\\[-4pt]
&			&								&			&$(1\ \zero{7}\ 1\ \zero{6})$\\[2pt]			
\noalign{\hrule height 1.0pt}
\end{tabular}
\vspace{8pt}
\caption{
The evolution of the index set 
obtained by running Algorithm~\ref{alg:parametric} with 
$\thetaX = 0.5$, $\thetaP = 0.9$ (case~(i)) in Experiment~1.
}
\label{Ex1:table:05:09:indexset}
}
\end{center}                                                                   
\end{table}

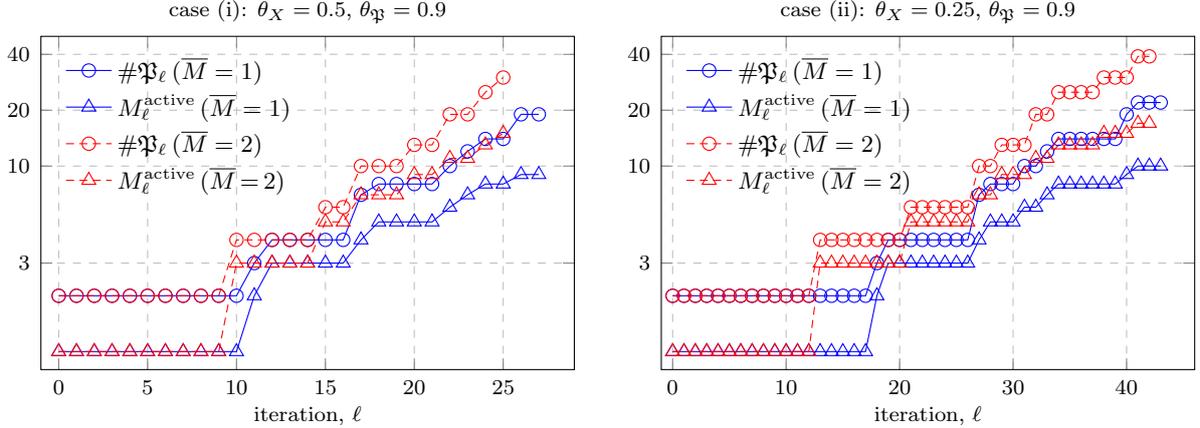
\begin{figure}[t!]
\begin{tikzpicture}
\pgfplotstableread{data/parametric/newex1/thetaX0.5_thetaP0.9_extrarv=1.dat}{\first}
\pgfplotstableread{data/parametric/newex1/thetaX0.5_thetaP0.9_extrarv=2.dat}{\second}
\begin{semilogyaxis}
[
width = 8.6cm, height=6.0cm,								
title={case (i): $\thetaX = 0.5$, $\thetaP = 0.9$},			
xlabel={iteration, $\ell$}, 
ylabel style={font=\tiny,yshift=-3.0ex}, 					
ymajorgrids=true, xmajorgrids=true, grid style=dashed,		
xmin = -1, 			xmax=29, 								
ymin = 0.8,			ymax=50,								
ytick={0,3,10,20,40},										
yticklabel style={/pgf/number format/.cd,fixed},
yticklabel={%
	\pgfmathfloatparsenumber{\tick}%
    \pgfmathfloatexp{\pgfmathresult}%
    \pgfmathprintnumber{\pgfmathresult}%
},
legend columns = 1,
cycle multi list = {color list\nexlist [2 of]mark list},
legend style={legend pos=north west, legend cell align=left, fill=none, draw=none, font={\fontsize{9pt}{12pt}\selectfont}}
]
%
\addplot[blue,mark=o,mark size=2.5pt]		table[x=iter, y=cardP]{\first};	
\addplot[blue,mark=triangle,mark size=3.0pt]table[x=iter, y=arv]{\first};	
%
%
%
\addplot[red,densely dashed,mark=o,mark size=2.5pt,mark options={solid}]		table[x=iter, y=cardP]{\second};
\addplot[red,densely dashed,mark=triangle,mark size=3.0pt,mark options={solid}]	table[x=iter, y=arv]{\second};  
%
\legend{
{$\#\gotP_\ell\, (\overline{M}= 1)$},
{$M^\mathrm{active}_\ell\, (\overline{M} = 1)$},
{$\#\gotP_\ell\, (\overline{M} = 2)$},
{$M^\mathrm{active}_\ell\, (\overline{M}\,{=}\,2)$},
}
\end{semilogyaxis}
\end{tikzpicture}
\hspace{8pt}
\begin{tikzpicture}
\pgfplotstableread{data/parametric/newex1/thetaX0.25_thetaP0.9_extrarv=1.dat}{\first}
\pgfplotstableread{data/parametric/newex1/thetaX0.25_thetaP0.9_extrarv=2.dat}{\second}
\begin{semilogyaxis}
[
width = 8.6cm, height=6.0cm,								
title={case (ii): $\thetaX = 0.25$, $\thetaP = 0.9$},		
xlabel={iteration, $\ell$},									
ylabel style={font=\tiny,yshift=-3.0ex}, 					
ymajorgrids=true, xmajorgrids=true, grid style=dashed,		
xmin=-1, 			xmax=46, 								
ymin=0.8,			ymax=50,								
ytick={3,10,20,40},											
yticklabel style={/pgf/number format/.cd,fixed},
yticklabel={%
	\pgfmathfloatparsenumber{\tick}%
    \pgfmathfloatexp{\pgfmathresult}%
    \pgfmathprintnumber{\pgfmathresult}%
},
legend columns = 1,
cycle multi list = {color list\nexlist [2 of]mark list},
legend style={legend pos=north west, legend cell align=left, fill=none, draw=none, font={\fontsize{9pt}{12pt}\selectfont}}
]
%
\addplot[blue,mark=o,mark size=2.5pt]		table[x=iter, y=cardP]{\first};	
\addplot[blue,mark=triangle,mark size=3.0pt]table[x=iter, y=arv]{\first};	
%
%
%
\addplot[red,densely dashed,mark=o,mark size=2.5pt,mark options={solid}]		table[x=iter, y=cardP]{\second};
\addplot[red,densely dashed,mark=triangle,mark size=3.0pt, mark options={solid}]table[x=iter, y=arv]{\second};  
%
\legend{
{$\#\gotP_\ell\, (\overline{M} = 1)$},
{$M^\mathrm{active}_\ell\, (\overline{M} = 1)$},
{$\#\gotP_\ell\, (\overline{M} = 2)$},
{$M^\mathrm{active}_\ell\, (\overline{M} = 2)$},
}
\end{semilogyaxis}
\end{tikzpicture}
\caption{
Characteristics of the index sets~$\gotP_\ell$  in Experiment~1 at each iteration 
of Algorithm~\ref{alg:parametric}.
}
\label{Ex1:dataindexset}
\end{figure}

In Table~\ref{Ex1:table:i:ii}, we collect the final outputs of computations 
in cases~(i) and~(ii) for $\overline{M} = 1$ and $\overline{M} = 2$, 
whereas Table~\ref{Ex1:table:05:09:indexset} shows the index set enrichments in case~(i).  
Recall that choosing a larger $\overline{M}$ in~\eqref{finite:index:set:Q}
leads to a larger detail index set, and hence, a larger set of marked indices, at each iteration.
As a result, more random variables are activated in the final index set and
the total number of iterations is reduced 
(compare the values of $L$, $\#\gotP_L$, and $M^\mathrm{active}_L$ in Table~\ref{Ex1:table:i:ii}). 
This can be also observed by looking at Figure~\ref{Ex1:dataindexset} that 
visualizes the evolution of the index set in cases~(i) and~(ii).
Furthermore, in case~(ii), due to a smaller marking parameter $\thetaX$, the algorithm
produces less refined triangulations
but takes more iterations to reach the tolerance than it does in case~(i)
(see the values of $\#\TT_L$ and $L$ in Table~\ref{Ex1:table:i:ii}).

Figure~\ref{Ex1:decay:est:05:09} shows the convergence history 
of three error estimates ($\mu_\ell$, $\zeta_\ell$, and $\mu_\ell \zeta_\ell$) and
the reference error $|G(\uref) - G(u_\ell)|$ in case~(i) for both~$\overline{M} = 1$
and~$\overline{M} = 2$ (see the end of this subsection for details on how the reference solution $\uref$ is computed). 
We observe that the estimates of the error in approximating~$G(u)$ (i.e., the products $\mu_\ell \zeta_\ell$)
decay with an overall rate of about $\mathcal{O}(N^{-0.55})$ 
for both $\overline{M}=1$ and $\overline{M}=2$. 
We notice that choosing $\overline{M}=2$ has a ``smoothing'' effect on the decay of $\mu_\ell\zeta_\ell$ 
(see Figure~\ref{Ex1:decay:est:05:09} (right));
this is due to larger index set enrichments in this case compared to those
in the case of $\overline{M} = 1$ 
(see the evolution of $\gotP_\ell$ in Table~\ref{Ex1:table:05:09:indexset}).

\begin{figure}[t!]
\begin{tikzpicture}
\pgfplotstableread{data/parametric/newex1/thetaX0.5_thetaP0.9_extrarv=1.dat}{\first}
\begin{loglogaxis}
[
width = 8.0cm, height=7cm,									
title={case (i) \,|\, $\overline{M} = 1$},					
xlabel={degree of freedom, $N_\ell$}, 						
ylabel={error estimate},									
ylabel style={font=\tiny,yshift=-1.0ex}, 					
ymajorgrids=true, xmajorgrids=true, grid style=dashed,		
xmin = 10^(1.5), 	xmax = 10^(6.3),						
ymin = 5*10^(-8),	ymax = (1.0)*10^(-1),					
legend style={legend pos=south west, legend cell align=left, fill=none, draw=none, font={\fontsize{9pt}{12pt}\selectfont}}
]
\addplot[blue,mark=diamond,mark size=3.0pt]		table[x=dofs, y=error_primal]{\first};
\addplot[red,mark=square,mark size=2.5pt]		table[x=dofs, y=error_dual]{\first};
\addplot[darkGreen,mark=o,mark size=2.5pt]		table[x=dofs, y=error_product]{\first};
\addplot[teal,mark=triangle,mark size=3.5pt]	table[x=dofs, y=truegerr]{\first};
\addplot[black,solid,domain=10^(1.5):10^(7.8)] { 0.04*x^(-1/4) };
\addplot[black,solid,domain=10^(1.5):10^(7.8)] { 0.0006*x^(-0.55) };
\node at (axis cs:1.0e5,2.5e-4) [anchor=south west] {$\mathcal{O}(N_\ell^{-1/4})$};
\node[black] at (axis cs:1.0e5,7e-8) [anchor=south west] {$\mathcal{O}(N_\ell^{-0.55})$};
\legend{
{$\mu_\ell$ (primal)},
{$\zeta_\ell$ (dual)},
{$\mu_\ell\,\zeta_\ell$},
{$|G(\uref) - G(u_\ell)|$}
}
\end{loglogaxis}
\end{tikzpicture}
\hspace{5pt}
\begin{tikzpicture}
\pgfplotstableread{data/parametric/newex1/thetaX0.5_thetaP0.9_extrarv=2.dat}{\first}
\begin{loglogaxis}
[
width = 8.0cm, height=7cm,									
title={case (i) \,|\, $\overline{M} = 2$},					
xlabel={degree of freedom, $N_\ell$}, 						
ylabel={error estimate},									
ylabel style={font=\tiny,yshift=-1.0ex}, 					
ymajorgrids=true, xmajorgrids=true, grid style=dashed,		
xmin = 10^(1.5), 	xmax = 10^(6.3),						
ymin = 5*10^(-8),	ymax = (1.0)*10^(-1),					
legend style={legend pos=south west, legend cell align=left, fill=none, draw=none, font={\fontsize{9pt}{12pt}\selectfont}}
]
\addplot[blue,mark=diamond,mark size=3.0pt]	table[x=dofs, y=error_primal]{\first};
\addplot[red,mark=square,mark size=2.5pt]	table[x=dofs, y=error_dual]{\first};
\addplot[darkGreen,mark=o,mark size=2.5pt]	table[x=dofs, y=error_product]{\first};
\addplot[teal,mark=triangle,mark size=3.5pt]	table[x=dofs, y=truegerr]{\first};
\addplot[black,solid,domain=10^(1.5):10^(7.8)] { 0.04*x^(-1/4) };
\addplot[black,solid,domain=10^(1.5):10^(7.8)] { 0.0006*x^(-0.55) };
\node at (axis cs:1.0e5,2.5e-4) [anchor=south west] {$\mathcal{O}(N_\ell^{-1/4})$};
\node[black] at (axis cs:1.0e5,7e-8) [anchor=south west] {$\mathcal{O}(N_\ell^{-0.55})$};
\legend{
{$\mu_\ell$ (primal)},
{$\zeta_\ell$ (dual)},
{$\mu_\ell\,\zeta_\ell$},
{$|G(\uref) - G(u_\ell)|$}
}
\end{loglogaxis}
\end{tikzpicture}\caption{
Error estimates $\mu_\ell$, $\zeta_\ell$, $\mu_\ell\,\zeta_\ell$
and the reference error $|G(\uref) - G(u_\ell)|$ at each iteration of
Algorithm~\ref{alg:parametric} with $\thetaX = 0.5$, $\thetaP = 0.9$ (case~(i)) in Experiment~1
(here, $G(\uref) = -3.180377$e-$03$).
}
\label{Ex1:decay:est:05:09}
\end{figure}
\begin{figure}[t!]
\begin{tikzpicture}
\pgfplotstableread{data/parametric/newex1/thetaX0.25_thetaP0.9_extrarv=1.dat}{\first}
\begin{loglogaxis}
[
width = 7.9cm, height=7cm,									
title={case (ii) \,|\, $\overline{M} =1$},					
xlabel={degree of freedom, $N_\ell$}, 						
ylabel={error estimate},									
ylabel style={font=\tiny,yshift=-1.0ex}, 					
ymajorgrids=true, xmajorgrids=true, grid style=dashed,		
xmin = 10^(1.5), 	xmax = 10^(6.3),						
ymin = 5*10^(-8),	ymax = (1.0)*10^(-1),					
legend style={legend pos=south west, legend cell align=left, fill=none, draw=none, font={\fontsize{9pt}{12pt}\selectfont}}
]
\addplot[blue,mark=diamond,mark size=3.0pt]		table[x=dofs, y=error_primal]{\first};
\addplot[red,mark=square,mark size=2.5pt]		table[x=dofs, y=error_dual]{\first};
\addplot[darkGreen,mark=o,mark size=2.5pt]		table[x=dofs, y=error_product]{\first};
\addplot[teal,mark=triangle,mark size=3.5pt]	table[x=dofs, y=truegerr]{\first};
\addplot[black,solid,domain=10^(1.5):10^(7.6)] { 0.04*x^(-1/4) };
\addplot[black,solid,domain=10^(1.5):10^(7.6)] { 0.0006*x^(-0.55) };
\node at (axis cs:1e5,2.5e-4) [anchor=south west] {$\mathcal{O}(N_\ell^{-1/4})$};
\node[black] at (axis cs:1e5,7e-8) [anchor=south west] {$\mathcal{O}(N_\ell^{-0.55})$};
\legend{
{$\mu_\ell$ (primal)},
{$\zeta_\ell$ (dual)},
{$\mu_\ell\,\zeta_\ell$},
{$|G(\uref) - G(u_\ell)|$}
}
\end{loglogaxis}
\end{tikzpicture}
\hspace{5pt}
\begin{tikzpicture}
\pgfplotstableread{data/parametric/newex1/thetaX0.25_thetaP0.9_extrarv=2.dat}{\first}
\begin{loglogaxis}
[
width = 7.9cm, height=7.0cm,								
title={case (ii) \,|\, $\overline{M} = 2$},					
xlabel={degree of freedom, $N_\ell$}, 						
ylabel={error estimate},									
ylabel style={font=\tiny,yshift=-1.0ex}, 					
ymajorgrids=true, xmajorgrids=true, grid style=dashed,		
xmin = 10^(1.5), 	xmax = 10^(6.3),						
ymin = 5*10^(-8),	ymax = (1.0)*10^(-1),					
legend style={legend pos=south west, legend cell align=left, fill=none, draw=none, font={\fontsize{9pt}{12pt}\selectfont}}
]
\addplot[blue,mark=diamond,mark size=3.0pt]		table[x=dofs, y=error_primal]{\first};
\addplot[red,mark=square,mark size=2.5pt]		table[x=dofs, y=error_dual]{\first};
\addplot[darkGreen,mark=o,mark size=2.5pt]		table[x=dofs, y=error_product]{\first};
\addplot[teal,mark=triangle,mark size=3.5pt]	table[x=dofs, y=truegerr]{\first};
\addplot[black,solid,domain=10^(1.5):10^(7.7)] { 0.04*x^(-1/4) };
\addplot[black,solid,domain=10^(1.5):10^(7.7)] { 0.0006*x^(-0.55) };
\node at (axis cs:1e5,2.5e-4) 		[anchor=south west] {$\mathcal{O}(N_\ell^{-1/4})$};
\node[black] at (axis cs:1e5,7e-8) [anchor=south west] {$\mathcal{O}(N_\ell^{-0.55})$};
\legend{
{$\mu_\ell$ (primal)},
{$\zeta_\ell$ (dual)},
{$\mu_\ell\,\zeta_\ell$},
{$|G(\uref) - G(u_\ell)|$}
}
\end{loglogaxis}
\end{tikzpicture}
\caption{
Error estimates $\mu_\ell$, $\zeta_\ell$, $\mu_\ell\,\zeta_\ell$
and the reference error $|G(\uref) - G(u_\ell)|$ at each iteration of
Algorithm~\ref{alg:parametric} with $\thetaX = 0.25$, $\thetaP = 0.9$ (case~(ii)) in Experiment~1
(here, $G(\uref) = -3.180377$e-$03$).
}
\label{Ex1:decay:est:025:09}
\end{figure}

\begin{figure}[t!]
\begin{tikzpicture}
\pgfplotstableread{data/parametric/newex1/thetaX0.5_thetaP0.9_extrarv=1.dat}{\first}
\pgfplotstableread{data/parametric/newex1/thetaX0.5_thetaP0.9_extrarv=2.dat}{\second}
\begin{semilogxaxis}
[
title={case (i): $\thetaX = 0.5$, $\thetaP = 0.9$},	
xlabel={degree of freedom, $N_\ell$}, 				
ylabel={effectivity index, $\Theta_\ell$},			
ylabel style={font=\tiny,yshift=-2.5ex}, 			
ymajorgrids=true, xmajorgrids=true, grid style=dashed,		
xmin = (9)*10^1, 	xmax = (5)*10^5,				
ymin = 6,			ymax = 14.5,						
ytick={7,10,12,14},								
width = 8.0cm, height=6.0cm,
legend style={legend pos=north east, legend cell align=left, fill=none, draw=none, font={\fontsize{9pt}{12pt}\selectfont}}
]
\addplot[blue,mark=o,mark size=2.5pt]		table[x=dofs, y=effindices]{\first};
\addplot[red,mark=triangle,mark size=3.0pt]	table[x=dofs, y=effindices]{\second};
\legend{
{$\overline{M} = 1$},
{$\overline{M} = 2$}
}
\end{semilogxaxis}
\end{tikzpicture}
\hspace{5pt}
\begin{tikzpicture}
\pgfplotstableread{data/parametric/newex1/thetaX0.25_thetaP0.9_extrarv=1.dat}{\first}
\pgfplotstableread{data/parametric/newex1/thetaX0.25_thetaP0.9_extrarv=2.dat}{\second}
\begin{semilogxaxis}
[
title={case (ii): $\thetaX = 0.25$, $\thetaP = 0.9$},
xlabel={degree of freedom, $N_\ell$}, 				
ylabel={effectivity index, $\Theta_\ell$},			
ylabel style={font=\tiny,yshift=-2.5ex}, 			
ymajorgrids=true, xmajorgrids=true, grid style=dashed,		
xmin = (9)*10^1, 	xmax = (5)*10^5,				
ymin = 6,			ymax = 14.5,					
ytick={7,10,12,14},									
width = 8.0cm, height=6.0cm,
legend style={legend pos=north east, legend cell align=left, fill=none, draw=none, font={\fontsize{9pt}{12pt}\selectfont}}
]
\addplot[blue,mark=o,mark size=2.5pt]		table[x=dofs, y=effindices]{\first};
\addplot[red,mark=triangle,mark size=3.0pt]	table[x=dofs, y=effindices]{\second};
\legend{
{$\overline{M} = 1$},
{$\overline{M} = 2$}
}
\end{semilogxaxis}
\end{tikzpicture}
\caption{
The effectivity indices for the goal-oriented error estimates in Experiment~1
at each iteration of Algorithm~\ref{alg:parametric}.
}
\label{Ex1:effindices}
\end{figure}

In Figure~\ref{Ex1:decay:est:025:09}, we plot three error 
estimates 
as well as the reference error in the goal functional in case~(ii).  
We observe that $\mu_\ell\zeta_\ell$ 
decay with about the same overall rate as in case~(i), i.e., $\mathcal{O}(N^{-0.55})$. 
On the other hand, the ``smoothing'' effect due to a larger $\overline{M}$ is less evident in case (ii),
compared to case~(i). This is likely due to a smaller value of the (spatial) 
marking parameter $\thetaX$ in case~(ii), which provides a more balanced refinement of 
spatial and parametric components of the generated Galerkin approximations
(note that the (parametric) marking parameter $\thetaP$ is the same in both cases).

Finally, for all cases considered in this experiment, we compute
the effectivity indices as explained in \S\ref{sec:experiments:outline}; see~\eqref{eff:indices}.
Here, we employ a reference Galerkin solution computed
using the triangulation $\TT_{\rm ref}$ 
obtained by a uniform refinement of $\TT_L$ from case~(i) with $\overline{M}=1$ and 
a large index set $\gotP_{\rm ref}$ which includes all indices generated in this experiment.
The effectivity indices are plotted in Figure~\ref{Ex1:effindices}.
Overall, they oscillate within the interval $(7.0,\,14.0)$ in all cases.

\begin{figure}[b!]
\centering
\footnotesize
	\includegraphics[scale=0.36]{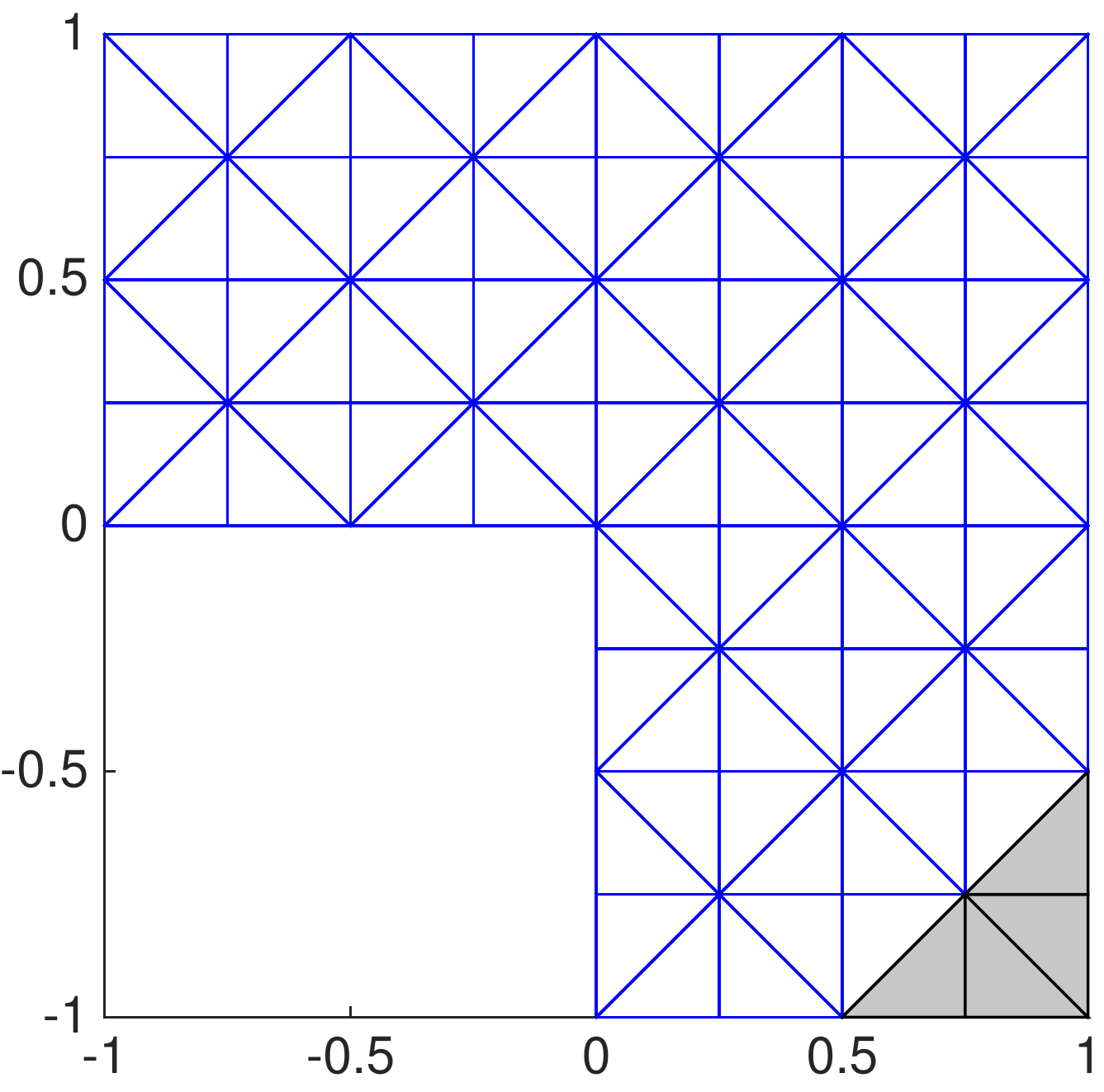} 
	\hspace{0.6cm} 
	\includegraphics[scale=0.36]{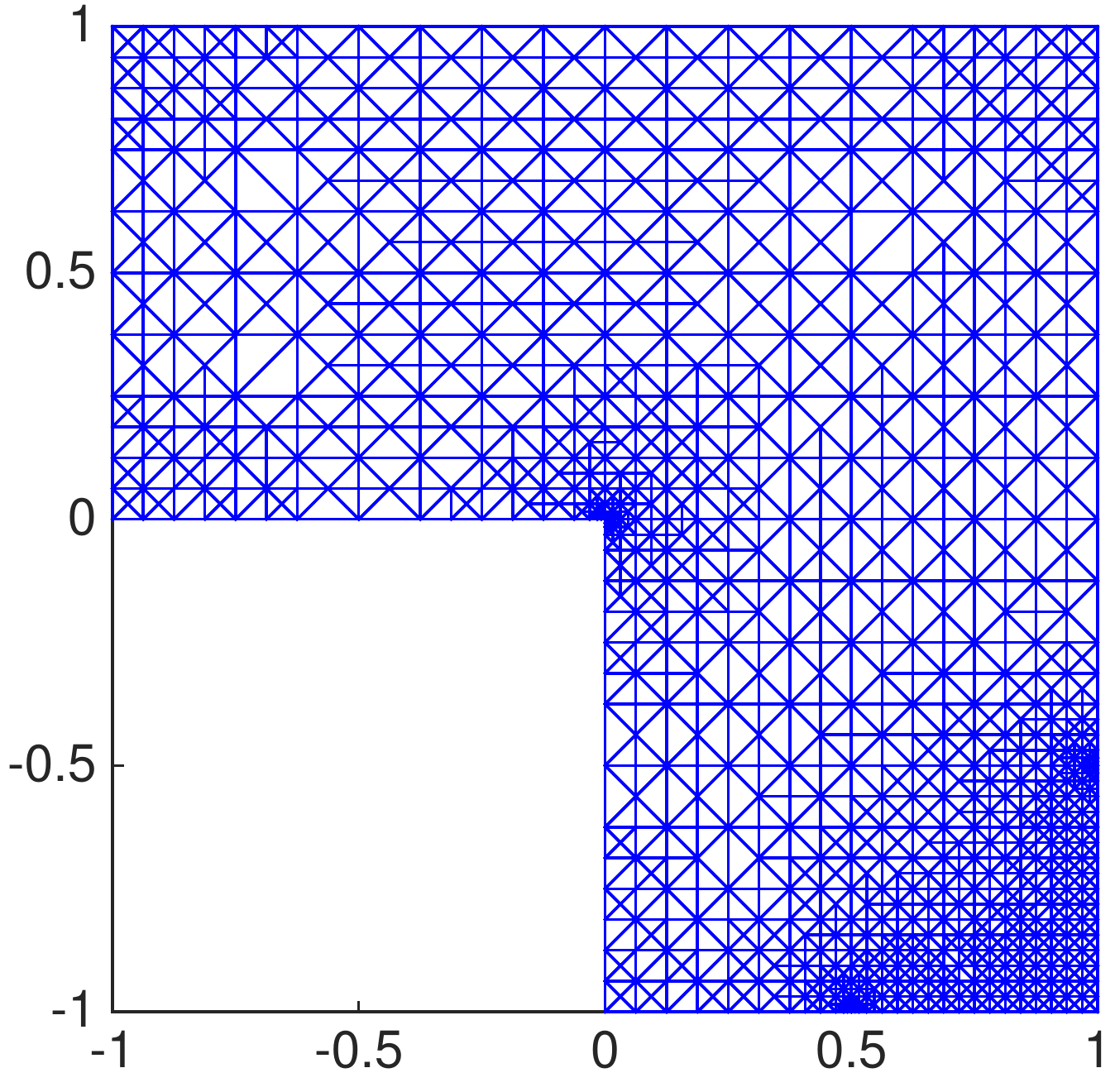}
\caption{
Initial triangulation $\TT_0$ 
with the shaded triangle $T_g$ (left) and the triangulation 
generated by the goal-oriented adaptive algorithm 
for an intermediate tolerance (right).
}
\label{Ex2:meshes}
\end{figure}
\subsection{Experiment~2} 
%
In this experiment, we consider the parametric model problem~\eqref{eq:strongform} posed 
on the L-shaped domain $D = (-1,1)^2 \setminus (-1,0]^2$ and we choose the parametric coefficient $a(x,\y)$ 
as the one introduced in~\cite[Section~11.1]{egsz14}.
Let $\sigma > 1$ and $0 < A < 1/\zeta(\sigma)$, where $\zeta$ denotes the Riemann zeta function.
For every $x = (x_1,x_2) \in D$, we fix $a_0(x) := 1$ and choose the coefficients $a_m(x)$ 
in~\eqref{eq1:a} to represent planar Fourier modes of increasing total order: 
\begin{equation} \label{Eigel:coeff}
a_m(x) := \alpha_m \cos(2\pi\beta_1(m) \, x_1) \cos(2\pi\beta_2(m)\, x_2)
\quad \text{for all $m \in \N$},
\end{equation}
where $\alpha_m := A m^{-\sigma}$ are the amplitudes of the coefficients and $\beta_1$, $\beta_2$ are defined as
\[
\beta_1(m) := m - k(m)(k(m) + 1)/2\ \ \hbox{and}\ \
\beta_2(m) := k(m) - \beta_1(m),
\]
with $k(m) := \lfloor -1/2  + \sqrt{1/4+2m}\rfloor$ for all $m \in \N$. 
Note that under these assumptions, both conditions \eqref{eq2:a} and \eqref{eq3:a} are satisfied
with $\azeromin = \azeromax = 1$ and $\tau = A \zeta(\sigma)$, respectively.

We assume that the parameters $y_m$ in \eqref{eq1:a} are the images of uniformly distributed independent 
mean-zero random variables on $\Gamma_m = [-1,1]$, so that $d\pi_m =~dy_m/2$ for all $m \in \N$.
Then, the orthonormal polynomial basis in $L^2_{\pi_m}(\Gamma_m)$ is comprised of scaled Legendre polynomials.
Note that the same parametric coefficient as described above
was also used in numerical experiments in~\cite{egsz15, bs16, em16, eps17, br18}.

In this experiment, we choose the quantity of interest that
involves the average value of a directional derivative of the primal solution
over a small region of the domain away from the reentrant corner.
More precisely, we set $f_0 = 1$, $\ff = (0,0)$, $g_0 = 0$, and $\gg = (\rchi_{T_g},0)$, 
where $\rchi_{T_g}$ denotes the characteristic function of the triangle 
$T_g := \text{conv}\{(1/2,-1), (1,-1), (1,-1/2)\}$ (see Figure~\ref{Ex2:meshes} (left)), 
so that the functionals in~\eqref{rhs:primal:problem}--\eqref{rhs:dual:problem} read as
\begin{equation*}
F(v) = 
\int_\Gamma \int_D v(x,\y) \, dx \, \dpiy,
\quad
G(v) = 
-\int_\Gamma \int_{T_g} \frac{\partial v}{\partial x_1}(x,\y) \, dx \, \dpiy	
\quad \text{for all } v \in V.
\end{equation*}

\begin{figure}[t!]
\begin{tikzpicture}
\pgfplotstableread{data/parametric/newex2/compar_adaptivity/thetaX0.0_thetaP1.0.dat}{\zeroone}
\pgfplotstableread{data/parametric/newex2/compar_adaptivity/thetaX1.0_thetaP0.0.dat}{\onezero}
\pgfplotstableread{data/parametric/newex2/compar_adaptivity/thetaX1.0_thetaP1.0_longer3.dat}{\oneone}
\pgfplotstableread{data/parametric/newex2/compar_adaptivity/thetaX0.6_thetaP1.0_longer.dat}{\zerosixone}
\pgfplotstableread{data/parametric/newex2/compar_adaptivity/thetaX1.0_thetaP0.6_longer.dat}{\onezerosix}
\pgfplotstableread{data/parametric/newex2/compar_adaptivity/thetaX0.2_thetaP0.8_longer.dat}{\zerotwozeroeight}
\begin{loglogaxis}
[
width = 10cm, height=7cm,									
xlabel={degree of freedom, $N_\ell$},						
ylabel={error estimate, $\mu_\ell\,\zeta_\ell$},			
ylabel style={font=\tiny,yshift=-1.0ex},					
ymajorgrids=true, xmajorgrids=true, grid style=dashed,		
xmin = (8)*10^(0), 	xmax = (3)*10^(7), 
ymin = (3)*10^(-6), ymax = (1.5)*10^(-2),					
legend style={legend pos=south west, legend cell align=left, fill=none, draw=none, font={\fontsize{9pt}{12pt}\selectfont}}
]
\addplot[myViolet,mark=diamond,mark size=3.5pt]	table[x=dofs, y=error_product]{\zeroone};
\addplot[myGray,mark=asterisk,mark size=3.0pt]	table[x=dofs, y=error_product]{\onezero};
\addplot[red,mark=o,mark size=2.5pt]			table[x=dofs, y=error_product]{\oneone};
\addplot[myOrange,mark=pentagon,mark size=3.0pt]table[x=dofs, y=error_product]{\zerosixone};
\addplot[blue,mark=square,mark size=2.0pt]		table[x=dofs, y=error_product]{\onezerosix};
\addplot[teal,mark=triangle,mark size=3.5pt]	table[x=dofs, y=error_product]{\zerotwozeroeight};
%
\node[diamond,fill=myViolet,fill opacity=0.3,scale=0.4] at (axis cs:165,0.008222426) {};
\node[diamond,fill=myViolet,fill opacity=0.3,scale=0.4] at (axis cs:462,0.008233172) {};
\node[diamond,fill=myViolet,fill opacity=0.3,scale=0.4] at (axis cs:1386,0.008234037) {};
\node[diamond,fill=myViolet,fill opacity=0.3,scale=0.4] at (axis cs:4356,0.008235500) {};
\node[diamond,fill=myViolet,fill opacity=0.3,scale=0.4] at (axis cs:14157,0.008235698) {};
%
\node[regular polygon,regular polygon sides=4,fill=blue,fill opacity=0.3,scale=0.4] at (axis cs:5892,0.000454121) {};
\node[regular polygon,regular polygon sides=4,fill=blue,fill opacity=0.3,scale=0.4] at (axis cs:145926,0.000062051) {};
\node[regular polygon,regular polygon sides=4,fill=blue,fill opacity=0.3,scale=0.4] at (axis cs:437769,0.000033884) {};
\node[regular polygon,regular polygon sides=4,fill=blue,fill opacity=0.3,scale=0.4] at (axis cs:2151435,0.000013762) {};
\node[regular polygon,regular polygon sides=4,fill=blue,fill opacity=0.3,scale=0.4] at (axis cs:5490702,0.000007713) {};
%
\node[regular polygon,regular polygon sides=5,fill=myOrange,fill opacity=0.3,scale=0.4] at (axis cs:2340,0.000563223) {};
\node[regular polygon,regular polygon sides=5,fill=myOrange,fill opacity=0.3,scale=0.4] at (axis cs:65436,0.000061840) {};
\node[regular polygon,regular polygon sides=5,fill=myOrange,fill opacity=0.3,scale=0.4] at (axis cs:744702,0.000018773) {};
\node[regular polygon,regular polygon sides=5,fill=myOrange,fill opacity=0.3,scale=0.4] at (axis cs:5667420,0.000008367) {};
%
\node[circle,fill=red,fill opacity=0.3,scale=0.5] at (axis cs:7365,0.000436295) {};
\node[circle,fill=red,fill opacity=0.3,scale=0.5] at (axis cs:340494,0.000045785) {};
\node[circle,fill=red,fill opacity=0.3,scale=0.5] at (axis cs:8214570,0.000011227) {};
%
\node[regular polygon,regular polygon sides=3,fill=teal,fill opacity=0.3,scale=0.3] at (axis cs:708,0.001260268) {};
\node[regular polygon,regular polygon sides=3,fill=teal,fill opacity=0.3,scale=0.3] at (axis cs:6300,0.000240472) {};
\node[regular polygon,regular polygon sides=3,fill=teal,fill opacity=0.3,scale=0.3] at (axis cs:21170,0.000120506) {};
\node[regular polygon,regular polygon sides=3,fill=teal,fill opacity=0.3,scale=0.3] at (axis cs:104145,0.000037954) {};
\node[regular polygon,regular polygon sides=3,fill=teal,fill opacity=0.3,scale=0.3] at (axis cs:261560,0.000020747) {};
\node[regular polygon,regular polygon sides=3,fill=teal,fill opacity=0.3,scale=0.3] at (axis cs:816648,0.000009197) {};
\legend{
{$\thetaX=0$, $\thetaP=1$},
{$\thetaX=1$, $\thetaP=0$},
{$\thetaX=1$, $\thetaP=1$},
{$\thetaX=0.6$, $\thetaP=1$},
{$\thetaX=1$, $\thetaP=0.6$},
{$\thetaX=0.2$, $\thetaP=0.8$},
}
\end{loglogaxis}
\end{tikzpicture}
\caption{
Error estimates $\mu_\ell\,\zeta_\ell$ at each iteration of Algorithm~\ref{alg:parametric}
for different sets of marking parameters in Experiment~2. Filled markers indicate iterations at which 
parametric enrichments occur.
}
\label{Ex2:compar:adaptivity}
\end{figure}
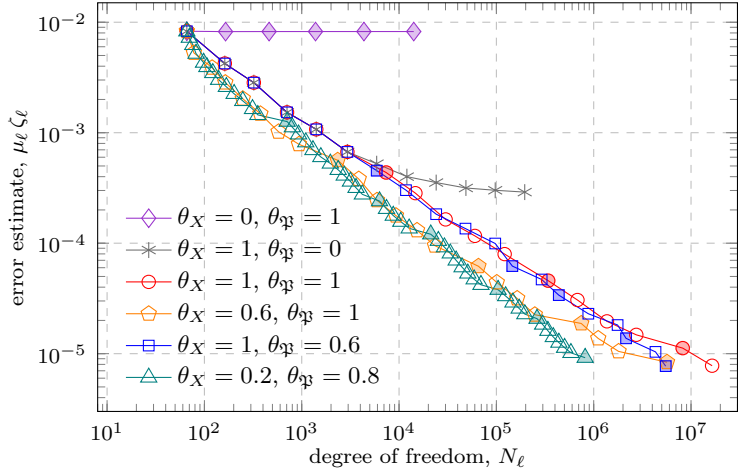

Note that in this example,
the primal and dual solutions both exhibit a geometric singularity at the reentrant corner
(see the left plots in Figures~\ref{Ex2:primal:sol:var} and~\ref{Ex2:dual:sol:var}).
In addition, the dual solution exhibits also a singularity along the line 
connecting the points $(1/2,-1)$ and $(1,-1/2)$
(see the left plot in Figure~\ref{Ex2:dual:sol:var});
the latter singularity is due to a low regularity of the goal functional $G(v)$.

Our first aim in this experiment is to show the advantages of using adaptivity in \emph{both}
components of Galerkin approximations.
To this end, we consider the expansion coefficients in~\eqref{Eigel:coeff} with 
slow ($\sigma = 2$) decay of the amplitudes $\alpha_m$ (fixing $\tau =~A \zeta(\sigma) = 0.9$, 
this results in $A \approx 0.547$) and we choose $\overline{M} = 1$ in \eqref{finite:index:set:Q}.
Starting with the coarse triangulation $\TT_0$ depicted in Figure~\ref{Ex2:meshes} (left)
and setting the tolerance to $\mathsf{tol} = 1$e-$05$, we run Algorithm~\ref{alg:parametric}
for six different sets of marking parameters and plot the error estimates $\mu_\ell\,\zeta_\ell$ computed at each iteration;
see Figure~\ref{Ex2:compar:adaptivity}.

In the cases where only one component of the Galerkin approximation is enriched
(i.e., either $\thetaX=0$ or $\thetaP=0$ as in the first two sets of parameters in Figure~\ref{Ex2:compar:adaptivity}),
the error estimates $\mu_\ell\,\zeta_\ell$ quickly stagnate as iterations progress,
and the set tolerance cannot be reached.
If both components are enriched but no adaptivity is used
(i.e., $\thetaX=\thetaP=1$, see the third set of parameters in Figure~\ref{Ex2:compar:adaptivity}), then 
the error estimates decay throughout all iterations.
However, in this case, the overall decay rate is slow and eventually deteriorates due to
the number of degrees of freedom growing very fast,
in particular, during the iterations with parametric enrichments (see the filled circle markers in Figure~\ref{Ex2:compar:adaptivity}).
The deterioration of the decay rate is also
observed for the fourth set of marking parameters in Figure~\ref{Ex2:compar:adaptivity} ($\thetaX=0.6$, $\thetaP=1$),
where adaptivity is only used for enhancing the spatial component of approximations.
If adaptivity is only used for enriching the parametric component
(e.g., $\thetaX=1$ and $\thetaP=0.6$ as in the fifth set in Figure~\ref{Ex2:compar:adaptivity}),
then the error estimates decay throughout all iterations without deterioration of the rate.
However, the decay rate in this case is slower than the one
for the sixth set of marking parameters $\thetaX = 0.2$, $\thetaP = 0.8$,
where adaptivity is used for both components of Galerkin approximations.
Thus, we conclude, that for the same level of accuracy,
adaptive enrichment in \emph{both} components
provides more balanced approximations with less degrees of freedom and
leads to a faster convergence rate than in all other cases considered in this experiment.

\begin{figure}[t!]
\centering
\footnotesize
	\includegraphics[scale=0.36]{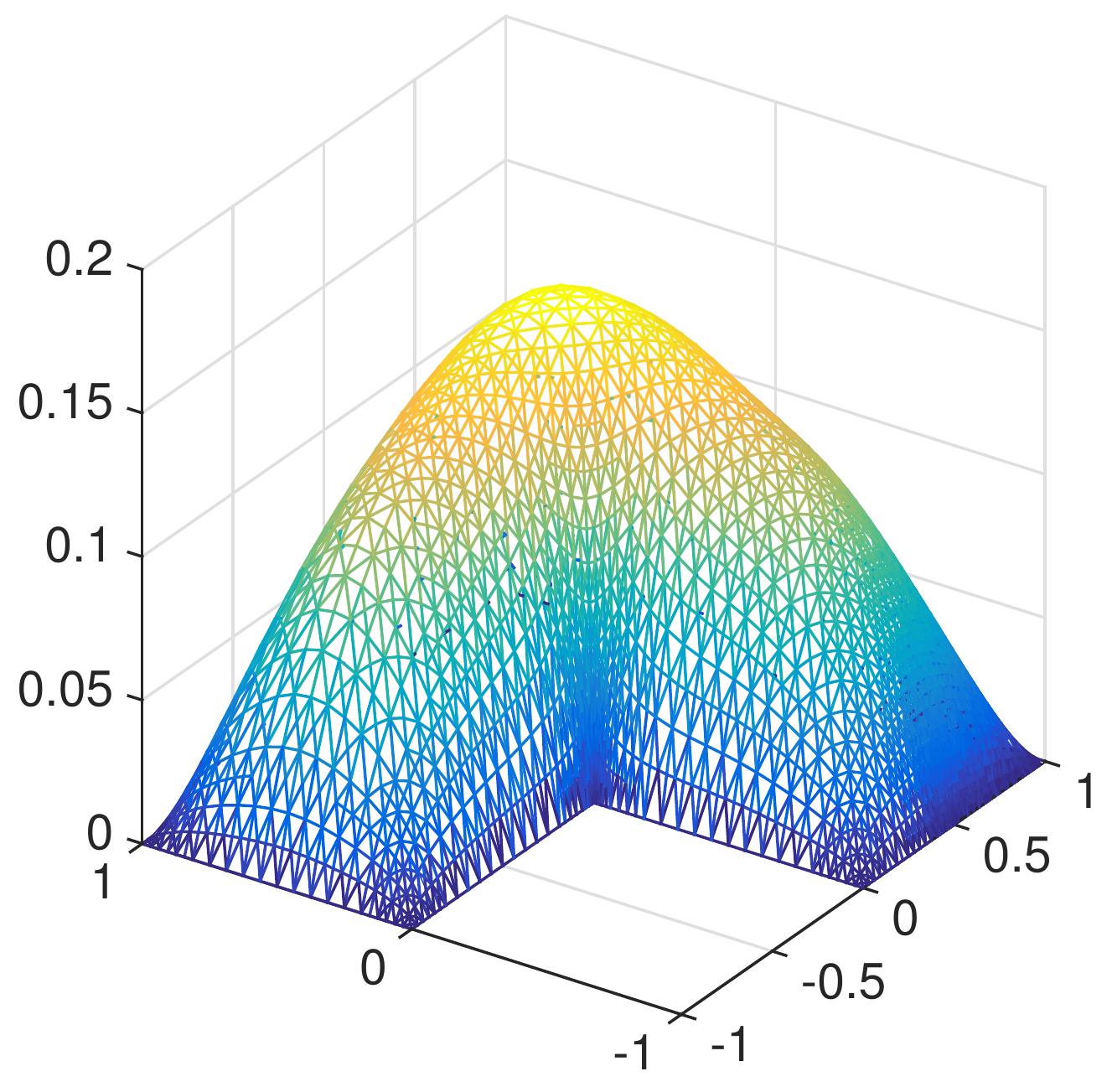} 
	\hspace{0.6cm}
	\includegraphics[scale=0.36]{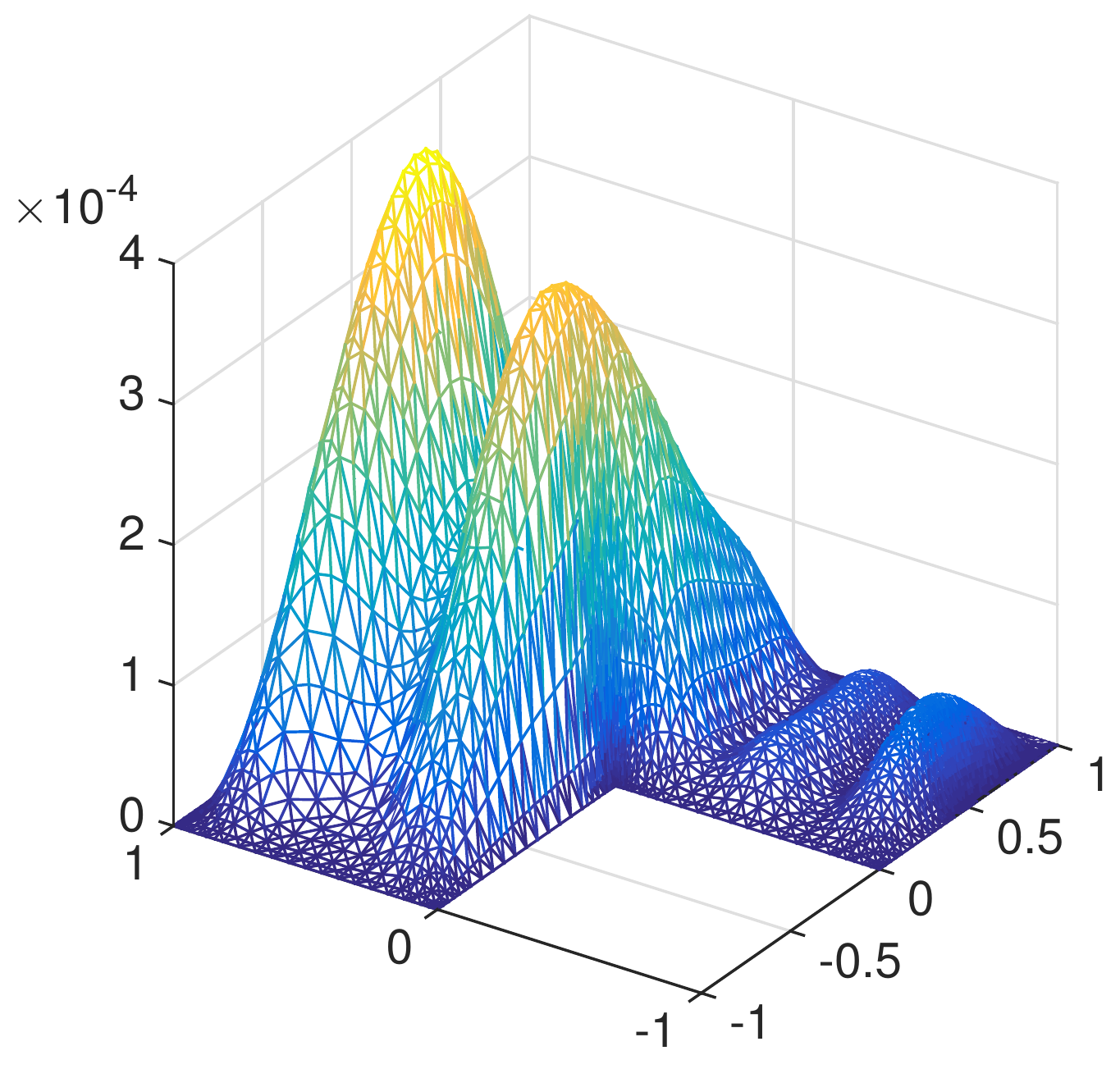}
\caption{
The mean field (left) and the variance (right) of the primal Galerkin solution in Experiment~2.
}
\label{Ex2:primal:sol:var}
\end{figure}
\begin{figure}[t!]
\centering
\footnotesize
	\includegraphics[scale=0.36]{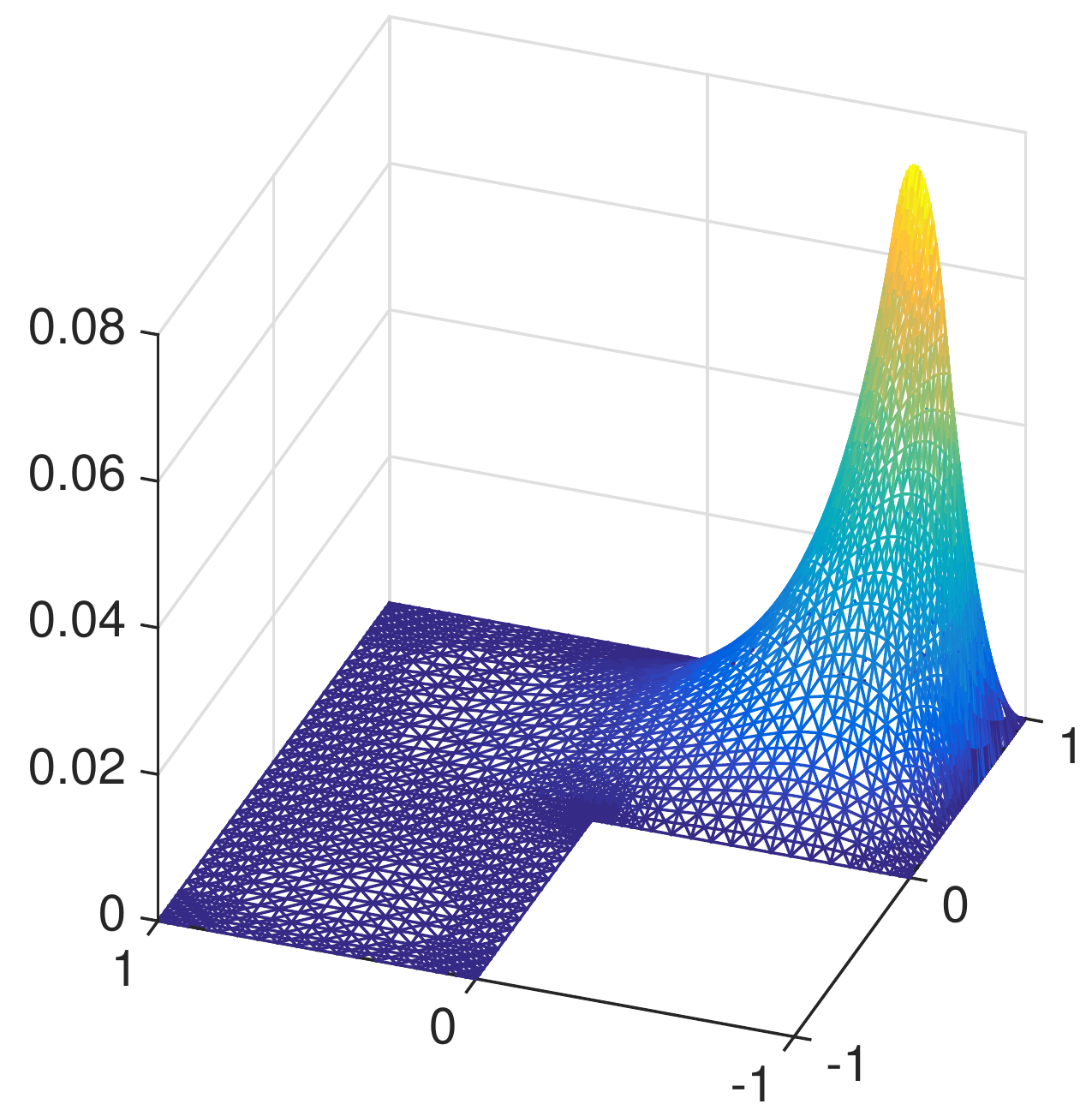} \hspace{0.6cm}
	\includegraphics[scale=0.36]{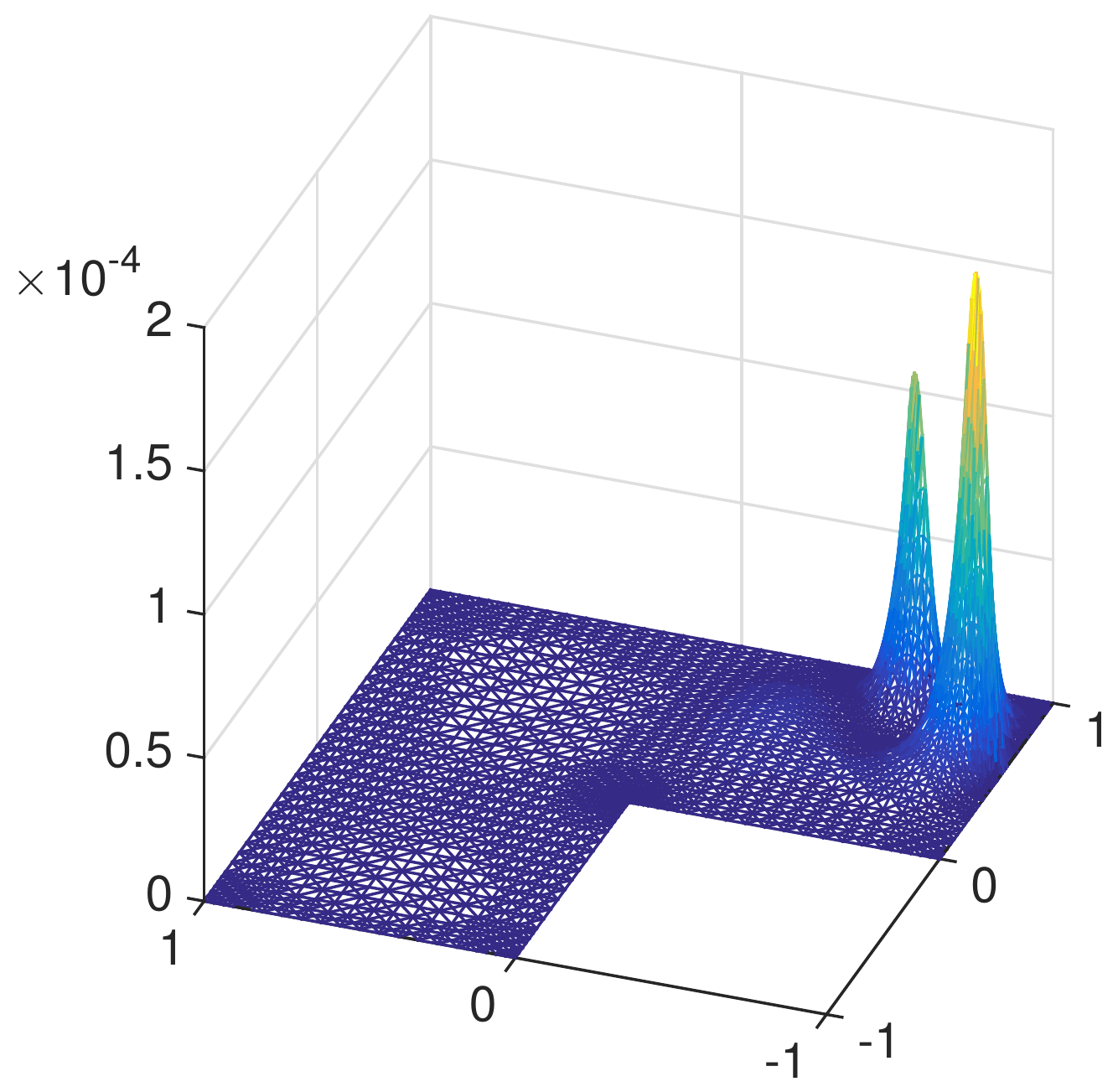}
\caption{
The mean field (left) and the variance (right) of the dual Galerkin solution in Experiment~2.
}
\label{Ex2:dual:sol:var}
\end{figure}

\begin{table}[t!]
\setlength\tabcolsep{8.5pt} 
\begin{center} 
\smallfontthree{ 
\renewcommand{\arraystretch}{1.3}
\begin{tabular}{r !{\vrule width 1.0pt} c c | c c !{\vrule width 1.0pt} c c | c c }
\noalign{\hrule height 1.0pt}
%
& \multicolumn{4}{c!{\vrule width 1.0pt}}{$\hbox{case (i):} \;\;\thetaX = 0.3,\  \thetaP = 0.8$}
& \multicolumn{4}{c}{$\hbox{case (ii):} \;\;\thetaX = 0.15,\ \thetaP = 0.95$}\\
\cline{2-9} 
&\multicolumn{2}{c|}{$\sigma = 2$} 
&\multicolumn{2}{c!{\vrule width 1.0pt}}{$\sigma = 4$}
&\multicolumn{2}{c|}{$\sigma = 2$} 
&\multicolumn{2}{c}{$\sigma = 4$}\\
\hline
$L$						&\multicolumn{2}{c|}{$37$}
						&\multicolumn{2}{c!{\vrule width 1.0pt}}{$37$}	
						&\multicolumn{2}{c|}{$62$}
						&\multicolumn{2}{c}{$63$}\\[-3pt]

$\mu_L\zeta_L$			&\multicolumn{2}{c|}{$8.8378\text{e-}06$}	
						&\multicolumn{2}{c!{\vrule width 1.0pt}}{$9.4638\text{e-}06$}	
						&\multicolumn{2}{c|}{$9.6811\text{e-}06$}
						&\multicolumn{2}{c}{$9.1013\text{e-}06$}\\[-3pt]

$t$ (sec) 				&\multicolumn{2}{c|}{$372$}	
						&\multicolumn{2}{c!{\vrule width 1.0pt}}{$345$}	
						&\multicolumn{2}{c|}{$604$}	
						&\multicolumn{2}{c}{$602$}\\[-3pt]
				
$N_{\rm total}$			&\multicolumn{2}{c|}{$3,507,551$}	
						&\multicolumn{2}{c!{\vrule width 1.0pt}}{$2,260,897$}	
						&\multicolumn{2}{c|}{$6,617,365$}	
						&\multicolumn{2}{c}{$4,289,136$}\\[-3pt]
					
$N_L$					&\multicolumn{2}{c|}{$725,800$}	
						&\multicolumn{2}{c!{\vrule width 1.0pt}}{$458,568$}	
						&\multicolumn{2}{c|}{$789,670$}	
						&\multicolumn{2}{c}{$552,857$}\\[-3pt]		

$\#\TT_L$				&\multicolumn{2}{c|}{$73,393$}
						&\multicolumn{2}{c!{\vrule width 1.0pt}}{$77,249$}
						&\multicolumn{2}{c|}{$55,183$}
						&\multicolumn{2}{c}{$65,800$}\\[-3pt]
																					
$\#\NN_L$				&\multicolumn{2}{c|}{$36,290$}
						&\multicolumn{2}{c!{\vrule width 1.0pt}}{$38,214$}	
						&\multicolumn{2}{c|}{$27,230$}	
						&\multicolumn{2}{c}{$32,521$}\\[-3pt]
				
$\#\gotP_L$				&\multicolumn{2}{c|}{$20$} 	
						&\multicolumn{2}{c!{\vrule width 1.0pt}}{$12$}	
						&\multicolumn{2}{c|}{$29$}	
						&\multicolumn{2}{c}{$17$}\\[-3pt]
				
$M^\mathrm{active}_L$	&\multicolumn{2}{c|}{$6$} 	
						&\multicolumn{2}{c!{\vrule width 1.0pt}}{$3$}	
						&\multicolumn{2}{c|}{$6$}	
						&\multicolumn{2}{c}{$4$}\\[1pt]
 
\hline 
\multicolumn{9}{l}{Evolution of the index set}\\
\hline

$\gotP_\ell$	
&$\ell=0$	&$(0\ 0)$				&$\ell=0$	&$(0\ 0)$		&$\ell=0$	&$(0\ 0)$			&$\ell=0$	&$(0\ 0)$\\[-4pt]       
&			&$(1\ 0)$				&		  	&$(1\ 0)$		&			&$(1\ 0)$			&			&$(1\ 0)$\\
								
&$\ell=12$	&$(0\ 1)$				&$\ell=8$	&$(2\ 0)$		&$\ell=14$	&$(0\ 1)$			&$\ell=10$	&$(2\ 0)$\\[-4pt]          
&			&$(2\ 0)$				&		   	&				&			&$(2\ 0)$			&			&\\
							
&$\ell=19$	&$(0\ 0\ 1)$			&$\ell=15$	&$(3\ 0)$		&$\ell=26$	&$(0\ 0\ 1)$		&$\ell=21$	&$(0\ 1)$\\[-4pt]
&			&$(1\ 1\ 0)$			&		  	&				&           &$(1\ 1\ 0)$     	&			&$(3\ 0)$\\[-4pt]
&			&						&		  	&				&           &$(3\ 0\ 0)$     	&			&\\

&$\ell=23$	&$(0\ 0\ 0\ 1)$			&$\ell=20$	&$(0\ 1)$		&$\ell=37$	&$(0\ 0\ 0\ 1)$		&$\ell=32$	&$(1\ 1)$\\[-4pt]  
& 			&$(1\ 0\ 1\ 0)$			&		  	&$(4\ 0)$		&			&$(0\ 2\ 0\ 0)$		&			&$(4\ 0)$\\[-4pt]
&			&$(2\ 1\ 0\ 0)$			&		  	&				&			&$(1\ 0\ 1\ 0)$     &			&\\[-4pt]
&			&$(3\ 0\ 0\ 0)$			&$\ell=26$	&$(1\ 1)$		&			&$(2\ 1\ 0\ 0)$     &$\ell=40$	&$(0\ 0\ 1)$\\[-4pt]
&			&						&		  	&$(5\ 0)$		&			&$(4\ 0\ 0\ 0)$     &			&$(2\ 1\ 0)$\\[-4pt]
&$\ell=30$	&$(0\ 0\ 0\ 0\ 1)$		&		  	&				&			&				    &			&$(5\ 0\ 0)$\\[-4pt]
&			&$(0\ 2\ 0\ 0\ 0)$		&			&				&$\ell=46$	&$(0\ 0\ 0\ 0\ 1)$	&			&\\[-4pt]
&			&$(1\ 0\ 0\ 1\ 0)$		&$\ell=30$	&$(2\ 1)$		&			&$(0\ 1\ 1\ 0\ 0)$	&$\ell=51$	&$(1\ 0\ 1)$\\[-4pt]
&			&$(2\ 0\ 1\ 0\ 0)$		&		  	&$(6\ 0)$		&			&$(1\ 2\ 0\ 0\ 0)$	&			&$(3\ 1\ 0)$\\[-4pt]
&			&$(3\ 1\ 0\ 0\ 0)$		&		  	&				&			&$(2\ 0\ 1\ 0\ 0)$	&			&$(6\ 0\ 0)$\\[-4pt]
&			&						&		  	&				&			&$(3\ 1\ 0\ 0\ 0)$	&			&\\
                                         
&$\ell=34$	&$(0\ 0\ 0\ 0\ 0\ 1)$	&$\ell=35$	&$(0\ 0\ 1)$	&$\ell=53$&$(0\ 0\ 0\ 0\ 0\ 1)$	&$\ell=60$	&$(0\ 0\ 0\ 1)$\\[-4pt]          
&			&$(0\ 1\ 1\ 0\ 0\ 0)$	&			&$(3\ 1\ 0)$	&      	&$(0\ 1\ 0\ 0\ 1\ 0)$	&			&$(2\ 0\ 1\ 0)$\\[-4pt] 
&			&$(1\ 0\ 0\ 0\ 1\ 0)$	&			&				&		&$(0\ 1\ 0\ 1\ 0\ 0)$	&			&$(4\ 1\ 0\ 0)$\\[-4pt]
&			&$(1\ 2\ 0\ 0\ 0\ 0)$	&			&				&		&$(1\ 0\ 0\ 0\ 0\ 1)$	&			&$(7\ 0\ 0\ 0)$\\[-4pt] 
&			&$(4\ 0\ 0\ 0\ 0\ 0)$	&			&				&		&$(1\ 0\ 0\ 0\ 1\ 0)$	&			&\\[-4pt]
&			&						&			&				&		&$(1\ 1\ 1\ 0\ 0\ 0)$	&			&\\[-4pt]
&			&						&			&				&		&$(2\ 0\ 0\ 1\ 0\ 0)$	&			&\\[-4pt]
&			&						&			&				&		&$(2\ 2\ 0\ 0\ 0\ 0)$	&			&\\[-4pt]
&			&						&			&				&		&$(3\ 0\ 1\ 0\ 0\ 0)$	&			&\\[-4pt]
&			&						&			&				&		&$(4\ 1\ 0\ 0\ 0\ 0)$	&			&\\[-4pt]
&			&						&			&				&		&$(5\ 0\ 0\ 0\ 0\ 0)$	&			&\\
\noalign{\hrule height 1.0pt}
\end{tabular}
\vspace{8pt}
\caption{
The outputs obtained by running Algorithm~\ref{alg:parametric} in Experiment~2 with 
$\thetaX = 0.3$, $\thetaP = 0.8$ (case~(i)) and $\thetaX = 0.15$, $\thetaP = 0.95$ (case~(ii)) 
for both slow ($\sigma=2$) and fast ($\sigma=4$) decay of the amplitude coefficients.
}
\label{Ex2:table}       
} 
\end{center}                                                                   
\end{table}

Let us now run Algorithm~\ref{alg:parametric} with the following two sets of marking parameters:
(i)~$\thetaX = 0.3$, $\thetaP = 0.8$; (ii)~$\thetaX = 0.15$, $\thetaP = 0.95$. 
In each case, we consider the expansion coefficients in~\eqref{Eigel:coeff} with 
slow ($\sigma = 2$) 
and fast ($\sigma = 4$) decay of the amplitudes $\alpha_m$ 
(in the latter case, fixing $\tau =~A \zeta(\sigma) = 0.9$ results in $A \approx 0.832$). 
In all computations, we choose $\overline{M} = 1$ in \eqref{finite:index:set:Q}
and set the tolerance to $\mathsf{tol} = 1$e-$05$. 

Figure~\ref{Ex2:meshes} (right) depicts an adaptively refined triangulation produced by 
Algorithm~\ref{alg:parametric} in case~(i) for the problem with slow decay of the amplitude coefficients
(similar triangulations were obtained in other cases). 
Observe that the triangulation effectively captures spatial features of primal and dual solutions.
Indeed, it is refined in the vicinity of the reentrant corner
and, similarly to Experiment~1,  in the vicinity of points $(1/2,-1)$ and $(1,-1/2)$.

\begin{figure}[t!]
\begin{tikzpicture}
\pgfplotstableread{data/parametric/newex2/slow_decay/thetaX0.30_thetaP0.80.dat}{\first}
\begin{loglogaxis}
[
title={case~(i) \,|\, $\sigma = 2$},								
xlabel={degree of freedom, $N_\ell$},								
ylabel={error estimate},											
ylabel style={font=\tiny,yshift=-1.0ex},							
ymajorgrids=true, xmajorgrids=true, grid style=dashed,				
xmin = 10^(1.5), 	xmax = 10^(6.4),								
ymin = 7*10^(-7),	ymax = 3*10^(-1),								
width = 8.0cm, height=7.0cm,
legend style={legend pos=south west, legend cell align=left, fill=none, draw=none, font={\fontsize{9pt}{12pt}\selectfont}}
]
\addplot[blue,mark=diamond,mark size=3.0pt]	table[x=dofs, y=error_primal]{\first};
\addplot[red,mark=square,mark size=2.5pt]	table[x=dofs, y=error_dual]{\first};
\addplot[darkGreen,mark=o,mark size=2.5pt]	table[x=dofs, y=error_product]{\first};
\addplot[teal,mark=triangle,mark size=3.5pt]table[x=dofs, y=truegerr]{\first};
\addplot[black,solid,domain=10^(1.5):10^(6.5)] { 1.3*x^(-0.35) };
\addplot[black,solid,domain=10^(1.5):10^(6.5)] { 0.3*x^(-2/3) };
\node at (axis cs:1e5,2e-2) [anchor=south west] {$\mathcal{O}(N_\ell^{-0.35})$};
\node at (axis cs:1e5,9e-5) [anchor=south west] {$\mathcal{O}(N_\ell^{-2/3})$};
\legend{
{$\mu_\ell$ (primal)},
{$\zeta_\ell$ (dual)},
{$\mu_\ell\,\zeta_\ell$},
{$|G(u_{\rm ref}) - G(u_\ell)|$}
}
\end{loglogaxis}
\end{tikzpicture}
\begin{tikzpicture}
\pgfplotstableread{data/parametric/newex2/fast_decay/thetaX0.30_thetaP0.80.dat}{\first}
\begin{loglogaxis}
[
title={case~(i) \,|\, $\sigma = 4$},						
xlabel={degree of freedom, $N_\ell$}, 						
ylabel={error estimate},									
ylabel style={font=\tiny,yshift=-1.0ex}, 					
ymajorgrids=true, xmajorgrids=true, grid style=dashed,		
xmin = 10^(1.5), 	xmax = 10^(6.3),						
ymin = 7*10^(-7),	ymax = 3*10^(-1),						
width = 8.0cm, height=7.0cm,
legend style={legend pos=south west, legend cell align=left, fill=none, draw=none, font={\fontsize{9pt}{12pt}\selectfont}}
]
\addplot[blue,mark=diamond,mark size=3.0pt]	table[x=dofs, y=error_primal]{\first};
\addplot[red,mark=square,mark size=2.5pt]	table[x=dofs, y=error_dual]{\first};
\addplot[darkGreen,mark=o,mark size=2.5pt]	table[x=dofs, y=error_product]{\first};
\addplot[teal,mark=triangle,mark size=3.5pt]table[x=dofs, y=truegerr]{\first};
\addplot[black,solid,domain=10^(1.5):10^(6.5)] { 1.3*x^(-0.35) };
\addplot[black,solid,domain=10^(1.5):10^(6.5)] { 0.6*x^(-3/4) };
\node at (axis cs:1e5,2e-2) [anchor=south west] {$\mathcal{O}(N_\ell^{-0.35})$};
\node at (axis cs:1e5,9e-5) [anchor=south west] {$\mathcal{O}(N_\ell^{-3/4})$};
\legend{
{$\mu_\ell$ (primal)},
{$\zeta_\ell$ (dual)},
{$\mu_\ell\,\zeta_\ell$},
{$|G(u_{\rm ref}) - G(u_\ell)|$}
}
\end{loglogaxis}
\end{tikzpicture}
\caption{
Error estimates $\mu_\ell$, $\zeta_\ell$, $\mu_\ell\,\zeta_\ell$ and the reference error $|G(u_{\rm ref}) - G(u_\ell)|$ 
at each iteration of Algorithm~\ref{alg:parametric} in Experiment~2 with $\thetaX=0.3$, $\thetaP=0.8$ (case~(i)) 
for $\sigma=2$ (left) and $\sigma=4$ (right)
(here, $G(\uref) = 1.789774$e-$2$ for $\sigma=2$ and $G(\uref) = 1.855648$e-$2$  for $\sigma=4$).
}
\label{Ex2:decay:case:i}
\end{figure}

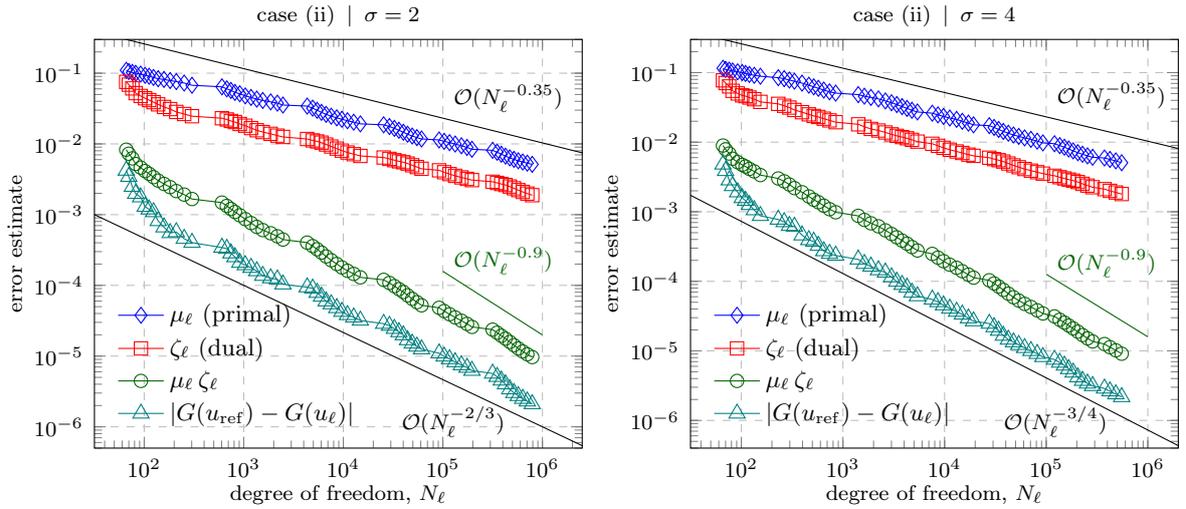
\begin{figure}[t!]
\begin{tikzpicture}
\pgfplotstableread{data/parametric/newex2/slow_decay/thetaX0.15_thetaP0.95.dat}{\first}
\begin{loglogaxis}
[
title={case~(ii) \,|\, $\sigma = 2$},						
xlabel={degree of freedom, $N_\ell$},						
ylabel={error estimate},									
ylabel style={font=\tiny,yshift=-1.0ex}, 					
ymajorgrids=true, xmajorgrids=true, grid style=dashed,		
xmin = 10^(1.5), 	xmax = 10^(6.4),						
ymin = 5*10^(-7),	ymax = 3*10^(-1),						
width = 8.0cm, height=7.0cm,
legend style={legend pos=south west, legend cell align=left, fill=none, draw=none, font={\fontsize{9pt}{12pt}\selectfont}}
]
\addplot[blue,mark=diamond,mark size=3.0pt]	table[x=dofs, y=error_primal]{\first};
\addplot[red,mark=square,mark size=2.5pt]	table[x=dofs, y=error_dual]{\first};
\addplot[darkGreen,mark=o,mark size=2.5pt]	table[x=dofs, y=error_product]{\first};
\addplot[teal,mark=triangle,mark size=3.5pt]table[x=dofs, y=truegerr]{\first};
\addplot[black,solid,domain=10^(1.5):10^(7.1)] { 1.3*x^(-0.35) };
\addplot[darkGreen,solid,domain=10^(5):10^(6)] {5*x^(-0.9) };
\addplot[black,solid,domain=10^(1.5):10^(7.1)] { 0.01*x^(-2/3) };
\node at (axis cs:1e5,2e-2) [anchor=south west] {$\mathcal{O}(N_\ell^{-0.35})$};
\node[darkGreen] at (axis cs:1e5,1e-4) [anchor=south west]{$\mathcal{O}(N_\ell^{-0.9})$};
\node at (axis cs:3e4,5e-7) [anchor=south west] {$\mathcal{O}(N_\ell^{-2/3})$};
\legend{
{$\mu_\ell$ (primal)},
{$\zeta_\ell$ (dual)},
{$\mu_\ell\,\zeta_\ell$},
{$|G(u_{\rm ref}) - G(u_\ell)|$}
}
\end{loglogaxis}
\end{tikzpicture}
\begin{tikzpicture}
\pgfplotstableread{data/parametric/newex2/fast_decay/thetaX0.15_thetaP0.95.dat}{\first}
\begin{loglogaxis}
[
title={case~(ii) \,|\, $\sigma=4$},
xlabel={degree of freedom, $N_\ell$}, 						
ylabel={error estimate},									
ylabel style={font=\tiny,yshift=-1.0ex}, 					
ymajorgrids=true, xmajorgrids=true, grid style=dashed,		
xmin = 10^(1.5), 	xmax = 10^(6.3),						
ymin = 4*10^(-7),	ymax = 3*10^(-1),						
width = 8.0cm, height=7.0cm,
legend style={legend pos=south west, legend cell align=left, fill=none, draw=none, font={\fontsize{9pt}{12pt}\selectfont}}
]
\addplot[blue,mark=diamond,mark size=3.0pt]	table[x=dofs, y=error_primal]{\first};
\addplot[red,mark=square,mark size=2.5pt]	table[x=dofs, y=error_dual]{\first};
\addplot[darkGreen,mark=o,mark size=2.5pt]	table[x=dofs, y=error_product]{\first};
\addplot[teal,mark=triangle,mark size=3.5pt]table[x=dofs, y=truegerr]{\first};
\addplot[black,solid,domain=10^(1.5):10^(7)] { 1.3*x^(-0.35) };
\addplot[darkGreen,solid,domain=10^(5):10^(6)] { 4*x^(-0.9) };
\addplot[black,solid,domain=10^(1.5):10^(7)] { 0.023*x^(-3/4) };
\node at (axis cs:1e5,2e-2) [anchor=south west] {$\mathcal{O}(N_\ell^{-0.35})$};
\node at (axis cs:1e5,8e-5) [anchor=south west] {\textcolor{darkGreen}{$\mathcal{O}(N_\ell^{-0.9})$}};
\node at (axis cs:3e4,4e-7) [anchor=south west] {$\mathcal{O}(N_\ell^{-3/4})$};
\legend{
{$\mu_\ell$ (primal)},
{$\zeta_\ell$ (dual)},
{$\mu_\ell\,\zeta_\ell$},
{$|G(u_{\rm ref}) - G(u_\ell)|$}
}
\end{loglogaxis}
\end{tikzpicture}
\caption{
Error estimates $\mu_\ell$, $\zeta_\ell$, $\mu_\ell\,\zeta_\ell$ and the reference error $|G(u_{\rm ref}) - G(u_\ell)|$ 
at each iteration of Algorithm~\ref{alg:parametric} in Experiment~2 
with $\thetaX=0.15$, $\thetaP=0.95$ (case~(ii)) for $\sigma=2$ (left) and $\sigma=4$ (right)
(here, $G(\uref) = 1.789774$e-$2$ for $\sigma=2$ and $G(\uref) = 1.855648$e-$2$ for $\sigma=4$).
}
\label{Ex2:decay:case:ii}
\end{figure}

Table~\ref{Ex2:table} collects the outputs of all computations. 
On the one hand, we observe that in case~(i), for both slow and fast decay of the amplitude coefficients, the algorithm took fewer 
iterations compared to case~(ii) 
($37$ versus $62$ for $\sigma=2$ and $37$ versus $63$ for $\sigma=4$) and
reached the tolerance faster (see the final times $t$ in Table~\ref{Ex2:table}). 
On the other hand, due to a larger $\thetaX$ in case~(i), the algorithm produced more refined triangulations
(see the values of $\#\TT_L$ in Table~\ref{Ex2:table}). 
Also, we observe that final index sets generated for the problem with slow decay ($\sigma=2$) are larger than those
for the problem with fast decay ($\sigma=4$)
($20$ indices versus $12$ in case~(i) and $29$ indices versus $17$ in case~(ii)). 
Furthermore, the algorithm tends to activate more parameters and to generate polynomial
approximations of lower degree for the problem with slow decay
(e.g., in case (i), polynomials of total degree $4$ in $6$ parameters for $\sigma=2$ versus
polynomials of total degree $6$ in $3$ parameters for $\sigma=4$).
Note that this behavior has been previously observed in numerical experiments for
parametric problems on the square domain; see, e.g.,~\cite{bs16}.

\begin{figure}[t!]
\begin{tikzpicture}
\pgfplotstableread{data/parametric/newex2/slow_decay/thetaX0.30_thetaP0.80.dat}{\first}
\pgfplotstableread{data/parametric/newex2/slow_decay/thetaX0.15_thetaP0.95.dat}{\second}
\begin{semilogxaxis}
[
width = 8.0cm, height=6.0cm,						
title={slow decay ($\sigma=2$)},					
xlabel={degree of freedom, $N_\ell$}, 				
ylabel={effectivity index, $\Theta_\ell$},			
ylabel style={font=\tiny,yshift=-3.0ex}, 			
ymajorgrids=true, xmajorgrids=true, grid style=dashed,		
xmin = (5)*10^1, 	xmax = 10^6,					
ymin = 1.7,	ymax = 5.1,								
ytick={2,3,4,5},									
legend style={legend pos=south east, legend cell align=left, fill=none, draw=none, font={\fontsize{9pt}{12pt}\selectfont}}
]
\addplot[blue,mark=o,mark size=3pt]				table[x=dofs, y=effindices]{\first};
\addplot[red,mark=triangle,mark size=3.5pt]		table[x=dofs, y=effindices]{\second};
\legend{
{$\thetaX = 0.3$, $\thetaP = 0.8$},
{$\thetaX = 0.15$, $\thetaP = 0.95$}
}
\end{semilogxaxis}
\end{tikzpicture}
\hspace{2pt}
\begin{tikzpicture}
\pgfplotstableread{data/parametric/newex2/fast_decay/thetaX0.30_thetaP0.80.dat}{\first}
\pgfplotstableread{data/parametric/newex2/fast_decay/thetaX0.15_thetaP0.95.dat}{\second}
\begin{semilogxaxis}
[
width = 8.0cm, height=6.0cm,						
title={fast decay ($\sigma=4$)},					
xlabel={degree of freedom, $N_\ell$}, 				
ylabel={effectivity index, $\Theta_\ell$},			
ylabel style={font=\tiny,yshift=-3.0ex}, 			
ymajorgrids=true, xmajorgrids=true, grid style=dashed,		
xmin = (5)*10^1, 	xmax = 10^6,					
ymin = 1.7,	ymax = 5.1,								
ytick={2,3,4,5},									
legend style={legend pos=south east, legend cell align=left, fill=none, draw=none, font={\fontsize{9pt}{12pt}\selectfont}}
]
\addplot[blue,mark=o,mark size=3pt]				table[x=dofs, y=effindices]{\first};
\addplot[red,mark=triangle,mark size=3.5pt]		table[x=dofs, y=effindices]{\second};
\legend{
{$\thetaX = 0.3$, $\thetaP = 0.8$},
{$\thetaX = 0.15$, $\thetaP = 0.95$}
}
\end{semilogxaxis}
\end{tikzpicture}
\caption{
The effectivity indices for the goal-oriented error estimates in Experiment~2 at each iteration of 
Algorithm~\ref{alg:parametric}.
}
\label{Ex2:effindices}
\end{figure}
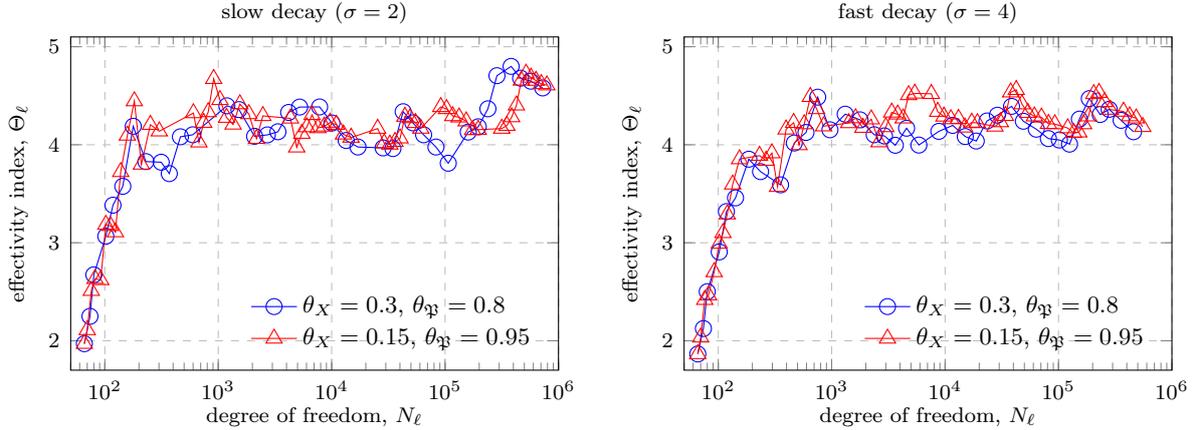

Figure~\ref{Ex2:decay:case:i} (resp., Figure~\ref{Ex2:decay:case:ii})
shows the convergence history of three error estimates
($\mu_\ell$, $\zeta_\ell$, and $\mu_\ell \zeta_\ell$) and the reference error in the goal functional 
in case~(i) (resp., case~(ii)) of marking parameters.
Firstly, we can see that the estimates $\mu_\ell \zeta_\ell$ 
converge with a faster rate for the problem with $\sigma=4$ than for the problem with $\sigma=2$.
This is true in both cases of marking parameters.
In particular, the overall convergence rate is about $\mathcal{O}(N^{-3/4})$ when $\sigma=4$ and 
about $\mathcal{O}(N^{-2/3})$ when $\sigma=2$.
Secondly, we observe an improved convergence rate during mesh refinement steps in case (ii)
(i.e., for smaller $\thetaX$ and larger $\thetaP$).
For both problems, i.e., for $\sigma=2$ and $\sigma=4$, this rate is about $\mathcal{O}(N^{-0.9})$
(see Figure~\ref{Ex2:decay:case:ii}), i.e., very close to the optimal one.

We conclude this experiment by testing the effectivity of the goal-oriented error estimation
at each iteration of Algorithm~\ref{alg:parametric}.
To this end, we compute the effectivity indices $\Theta_\ell$ (see~\eqref{eff:indices}) by
employing reference Galerkin solutions to problems with slow ($\sigma=2$) and fast ($\sigma=4$)
decay of the amplitude coefficients.
Specifically, for both problems we employ the same reference triangulation $\TT_{\rm ref}$
(obtained by a uniform refinement of the final triangulation $\TT_L$ generated
in case~(i) for the problem with slow decay),
but use two reference index sets
(namely, for $\sigma=2$, we set $\gotP_{\rm ref} := \gotP_L$, where $\gotP_L$ is generated
for the problem with slow decay in case~(ii) and
for $\sigma=4$, we set $\gotP_{\rm ref} := \gotP_L \cup \gotM_L$ with the corresponding 
$\gotP_L$ and $\gotM_L$ generated for the problem with fast decay in case~(ii)).
The effectivity indices are plotted in Figure~\ref{Ex2:effindices}.
As iterations progress, they tend to concentrate within the interval $(3.5,\,5.0)$ in all cases.

For the parametric model problem considered in this experiment, 
we conclude that Algorithm~\ref{alg:parametric} performs better if
the (spatial) marking threshold~$\thetaX$ is sufficiently small and
the (parametric) marking threshold $\thetaP < 1$ is sufficiently large
(see the results of experiments in case~(ii)).
In fact, in case~(ii), the estimates $\mu_\ell\zeta_\ell$ 
converge with nearly optimal rates during spatial refinement steps for problems with
slow and fast decay of the amplitude coefficients.
Furthermore, in this case, the algorithm generates 
richer index sets, which leads to more accurate parametric approximations. 

\subsection{Experiment~3} 
In the final experiment, we test the performance of Algorithm~\ref{alg:parametric} 
for the parametric model problem \eqref{eq:strongform} posed on the slit domain 
$D = (-1,1)^2 \setminus ([-1,0]\,\times\,\{0\})$.
The boundary of this domain is non-Lipschitz; however,
the problem on $D$ can be seen as a limit case of the problem on the Lipschitz domain
$D_\delta = (-1,1)^2 \setminus \overline{T}_\delta$ as $\delta \to 0$,
where $T_\delta = \text{conv}\{(0,0), (-1,\delta), (-1,-\delta)\}$
(cf.~\cite[p.~259]{SF08}).\footnote{We refer to~\cite[Section 2.7, p.~83]{Grisward92}
for a discussion about the well-posedness of the standard weak formulation
for the deterministic Poisson problem on the slit domain,
in particular, in the case of homogeneous Dirichlet boundary conditions.}
In fact, all computations in this experiment were performed for the domain
$D = D_\delta$ with $\delta = 0.005$.

Following \cite{emn16}, we consider a modification of the parametric coefficient used in Experiment~2. 
For all $m \in \N$ and $x \in D$, let $a_m(x)$ be the coefficients defined in \eqref{Eigel:coeff} with
$\alpha_m := A m^{-\sigma}$ for some $\sigma > 1$ and $0 < A < 1/\zeta(\sigma)$, 
where $\zeta$ is the Riemann zeta function.
Then, given two constants $c, \, \varepsilon > 0$, we define
\begin{equation} \label{emn:random:field}
a(x,\y) := \frac{c}{\alpha_{\rm min}}\left(\sum_{m=1}^\infty y_m a_m(x) + \alpha_{\rm min} \right) + \varepsilon, 
\end{equation}
where $\alpha_{\rm min} := A \zeta(\sigma)$ and
the parameters $y_m$ are the images of uniformly distributed independent mean-zero random variables on $[-1,1]$.

\begin{figure}[t!]
\centering
\footnotesize
	\includegraphics[scale=0.33]{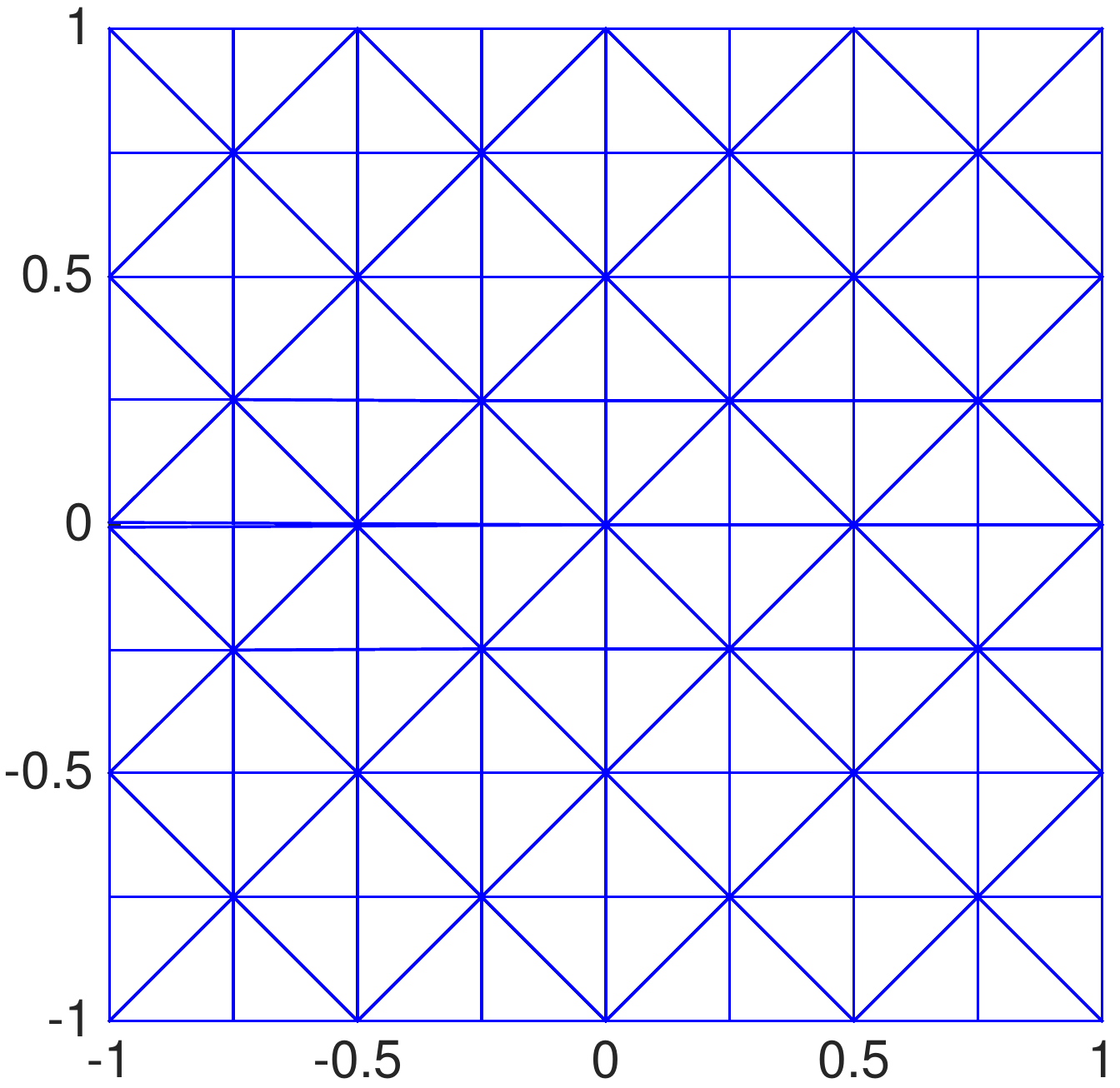} 
	\hspace{0.6cm}
	\includegraphics[scale=0.33]{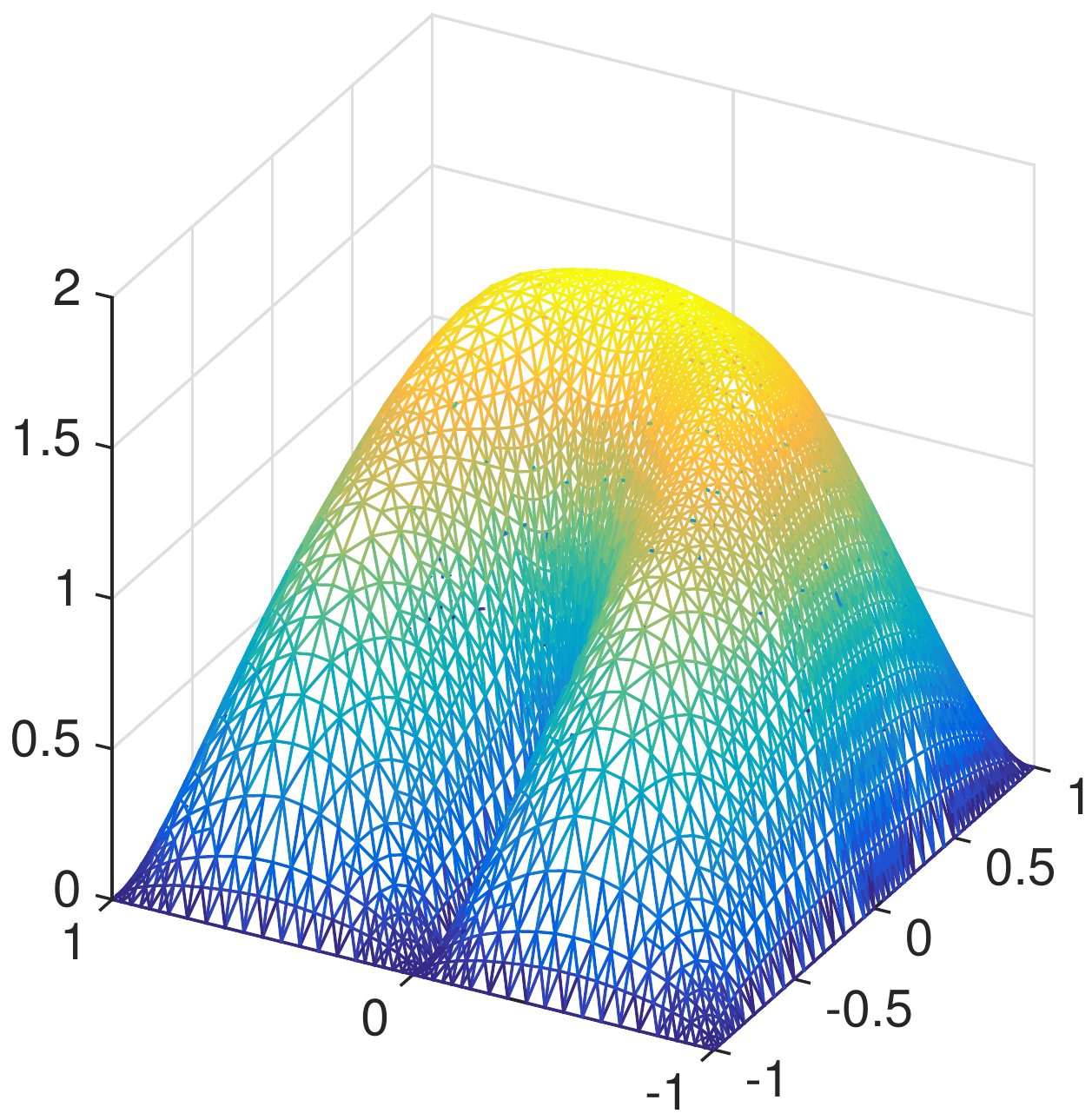} 
	\hspace{0.6cm}
	\includegraphics[scale=0.33]{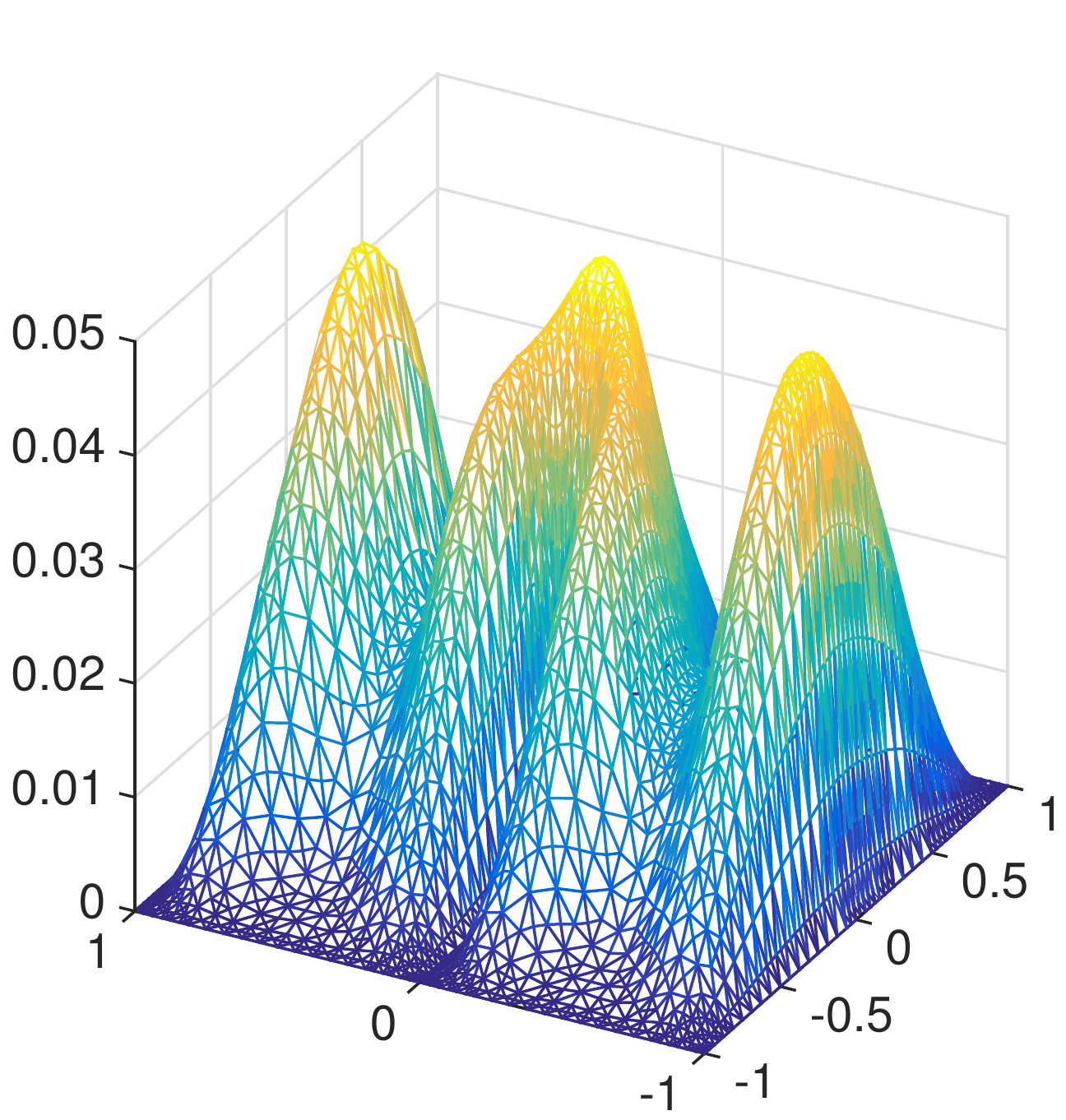} 
\caption{
Initial triangulation $\TT_0$ (left) 
as well as the mean field (middle) and the variance (right) of the primal Galerkin solution in 
Experiment~3.
}
\label{Ex3:T0:primal:sol:var}
\end{figure}

It is easy to see that $a(x,\y) \in [\varepsilon, \, 2c + \varepsilon]$ for all $x \in D$ and $\y \in \Gamma$.
Note that \eqref{emn:random:field} can be written in the form~\eqref{eq1:a} 
with $a_0(x) = c + \varepsilon$ and the expansion coefficients given by $(c \,a_m(x))/\alpha_{\rm min}$. 
Furthermore, conditions \eqref{eq2:a} and \eqref{eq3:a} are satisfied 
with $\azeromin =~\azeromax =~c + \varepsilon$ and $\tau =c/(c + \varepsilon)$, respectively.

It is known that solution $u$ to problem~\eqref{eq:strongform} in this example
exhibits a singularity induced by the slit in the domain.
Our aim in this experiment is to approximate the value of $u$ at some fixed point $x_0 \in D$
away from the slit.
To that end (and to stay within the framework of the bounded goal functional $G$ in~\eqref{rhs:dual:problem}), 
we fix a sufficiently small $r >0$ and define
$g_0$ as the \emph{mollifier} (see~\cite{prudhommeOden1999}):
\begin{equation} \label{mollifier}
g_0(x) = g_0(x; x_0, r) := 
\begin{cases}
C \exp\left(- \frac{r^2}{r^2 - \norm{x - x_0}{2}^2} \right)	& \text{if $\norm{x - x_0}{2} < r$}, \\
0 								& \text{otherwise.}
\end{cases}
\end{equation}
Here, $\norm{\cdot}{2}$ denotes the $\ell_2$-norm and the constant $C$ is chosen such that
\[
   \int_D g_0(x) dx = 1.
\]
Note that the value of the constant $C$ is independent of the location of $x_0 \in D$,
provided that $r$ is chosen sufficiently small such that $\supp(g_0(x;x_0,r)) \subset D$.
In this case, $C \approx 2.1436\,r^{-2}$ (see, e.g., \cite{prudhommeOden1999}).

Setting $f_0 = 1$, $\ff = (0,0)$ and $\gg = (0,0)$,
the functionals in~\eqref{rhs:primal:problem}--\eqref{rhs:dual:problem} read as
\begin{equation*}
F(v) = 
\int_\Gamma \int_D v(x,\y) \, dx \, \dpiy,
\quad
G(v) = 
\int_\Gamma \int_D g_0(x) v(x,\y) \, dx \, \dpiy	
\quad \text{for all } v \in V. 
\end{equation*}
Note that if $u(x,\y)$ is continuous in the spatial neighborhood of $x_0$, then
$G(u)$ converges to the mean value $\E[u(x_0,\y)]$ as $r$ tends to zero.

We fix $c = 10^{-1}$, $\varepsilon = 5\cdot 10^{-3}$, $\sigma=2$, $A = 0.6$ 
and set $x_0 = (0.4,-0.5) \in D$. 
In all computations performed in this experiment,
we use the coarse triangulation $\TT_0$ depicted in~Figure~\ref{Ex3:T0:primal:sol:var} (left plot).
Figure~\ref{Ex3:T0:primal:sol:var} also shows the mean field (middle plot) and the variance (right plot)
of the primal Galerkin solution.

\begin{figure}[t!]
\centering
\footnotesize
	\includegraphics[scale=0.31]{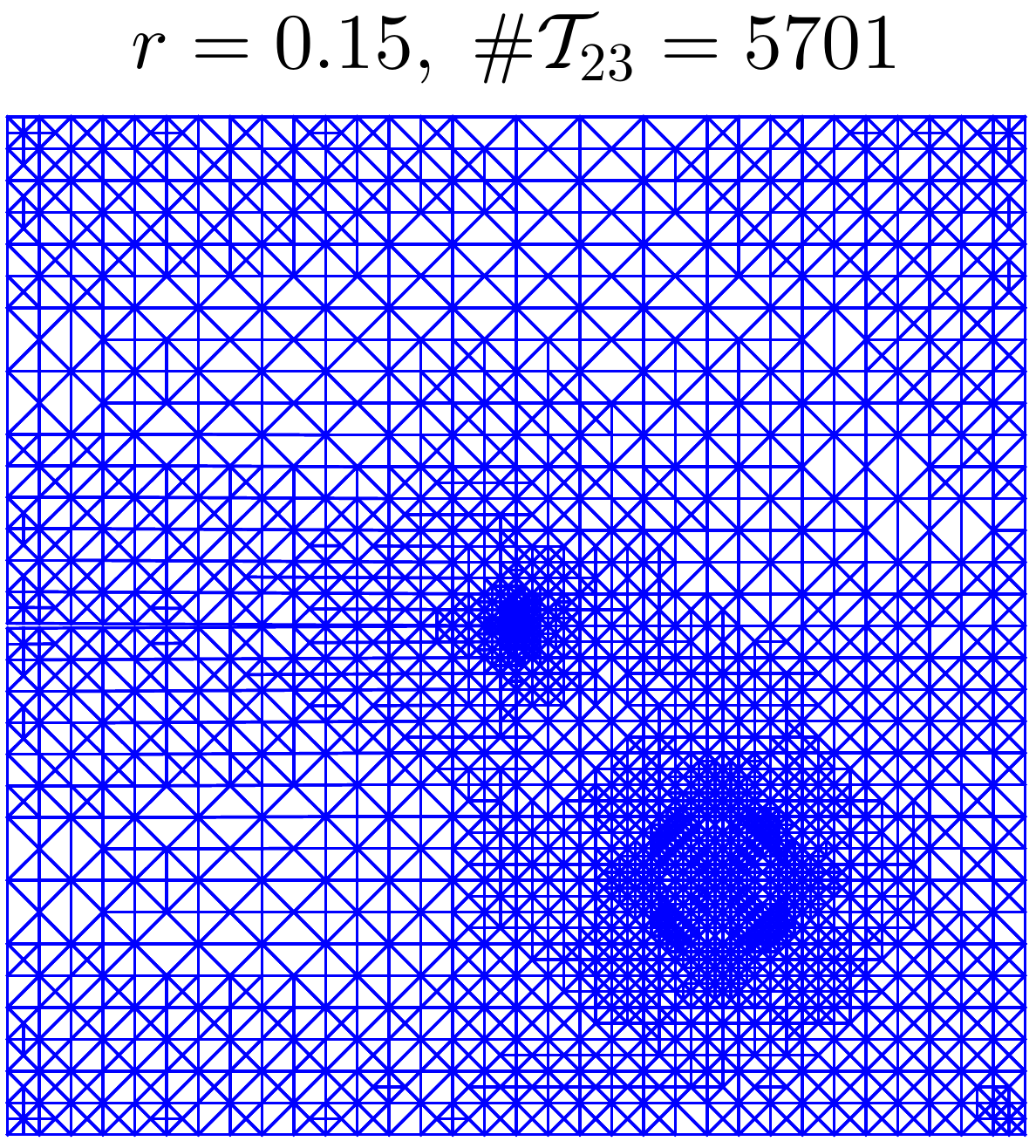} 	\hspace{0.01cm}
	\includegraphics[scale=0.31]{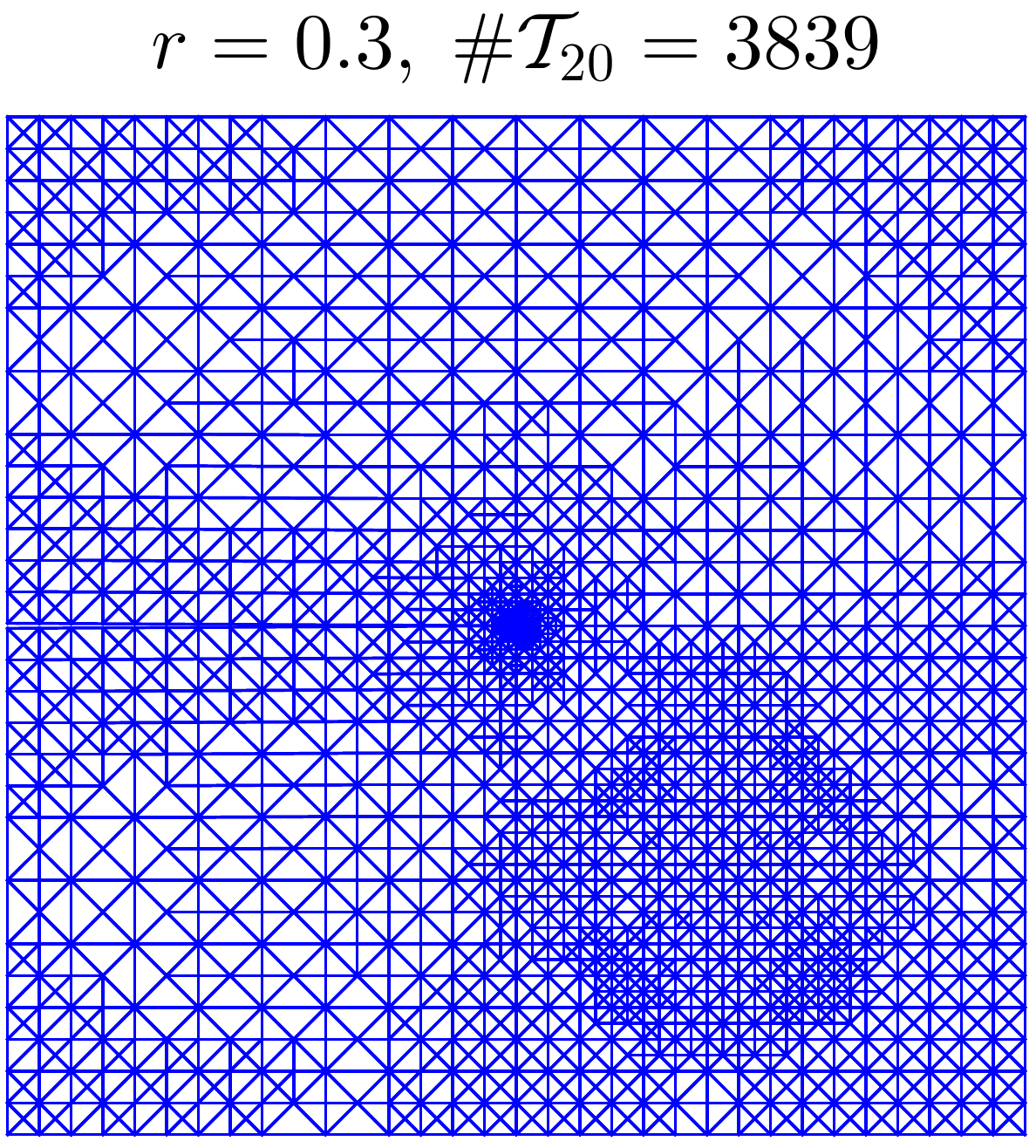} 	\hspace{0.01cm}
	\includegraphics[scale=0.31]{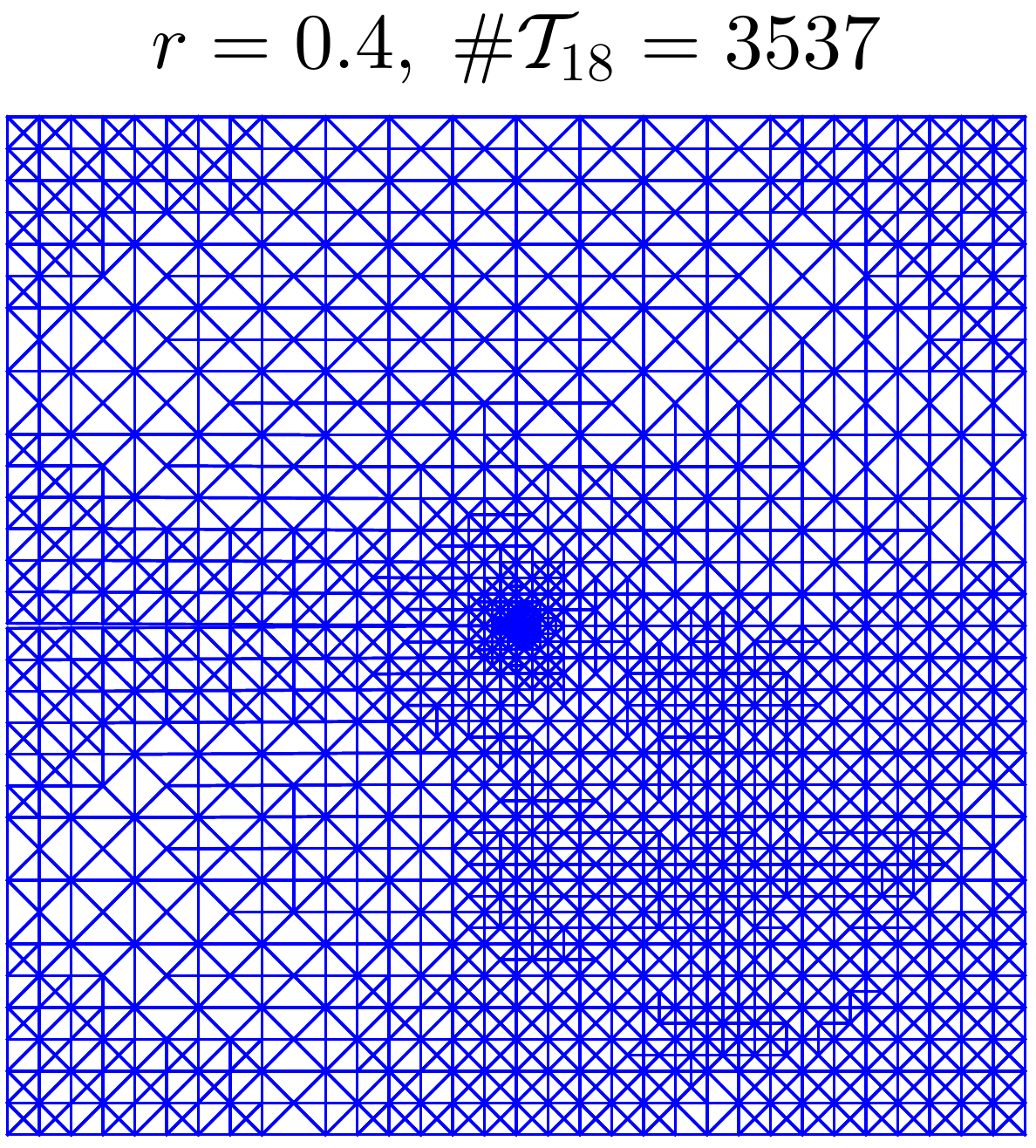} \hspace{0.01cm}
	\includegraphics[scale=0.31]{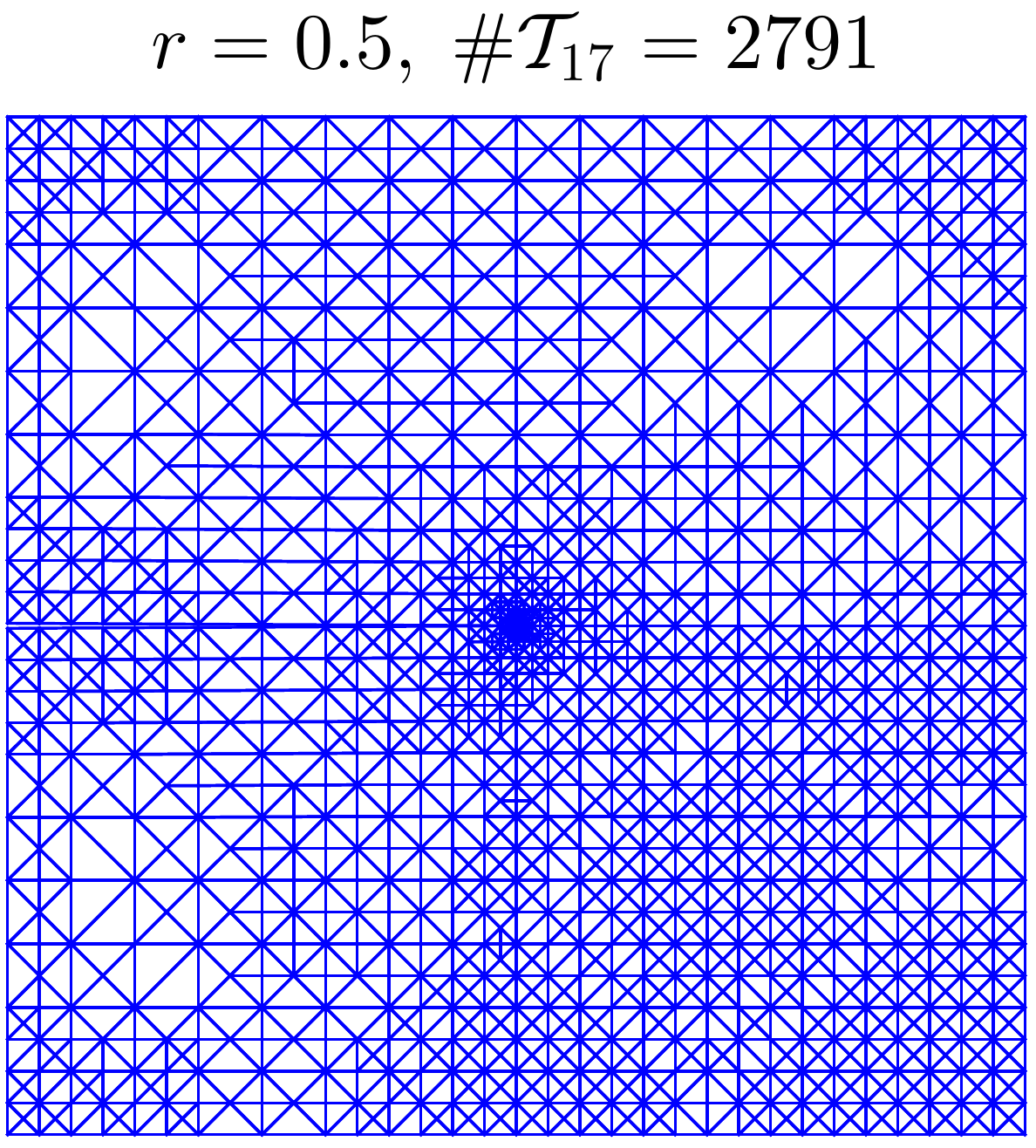}	\\
	\vspace{6pt}
	\includegraphics[scale=0.3]{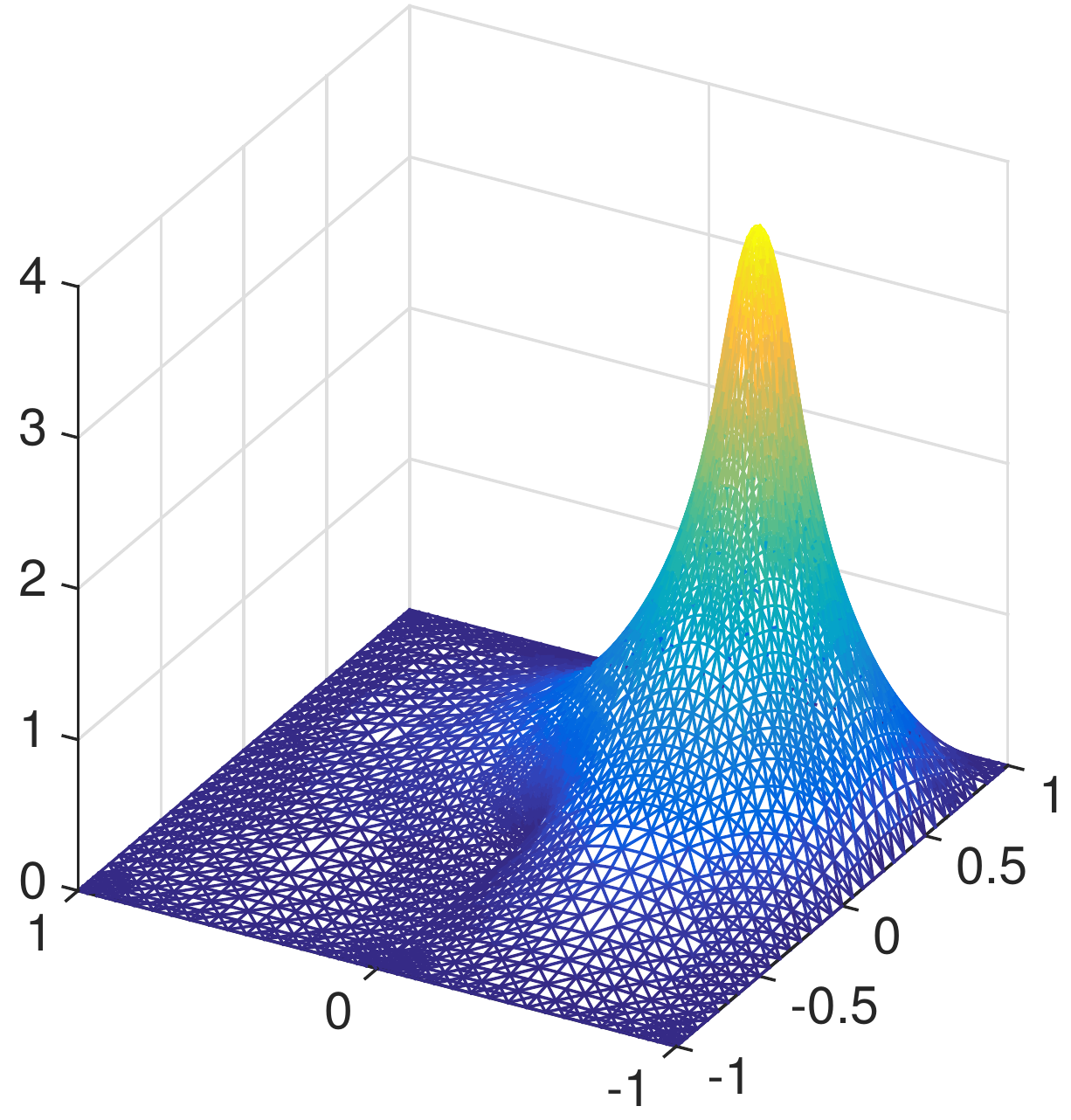} 	\hspace{0.1cm}
	\includegraphics[scale=0.3]{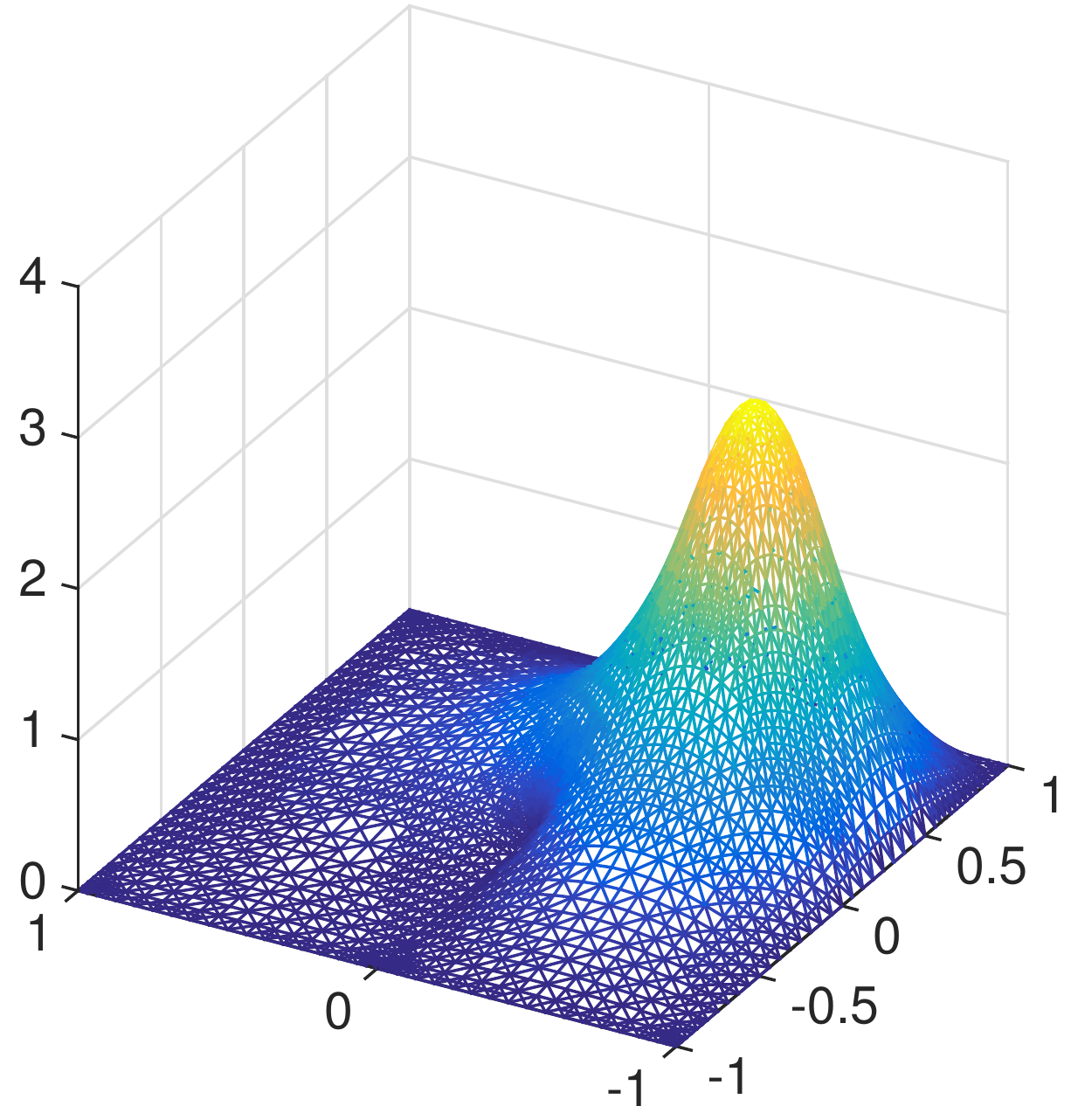} 	\hspace{0.1cm}
	\includegraphics[scale=0.3]{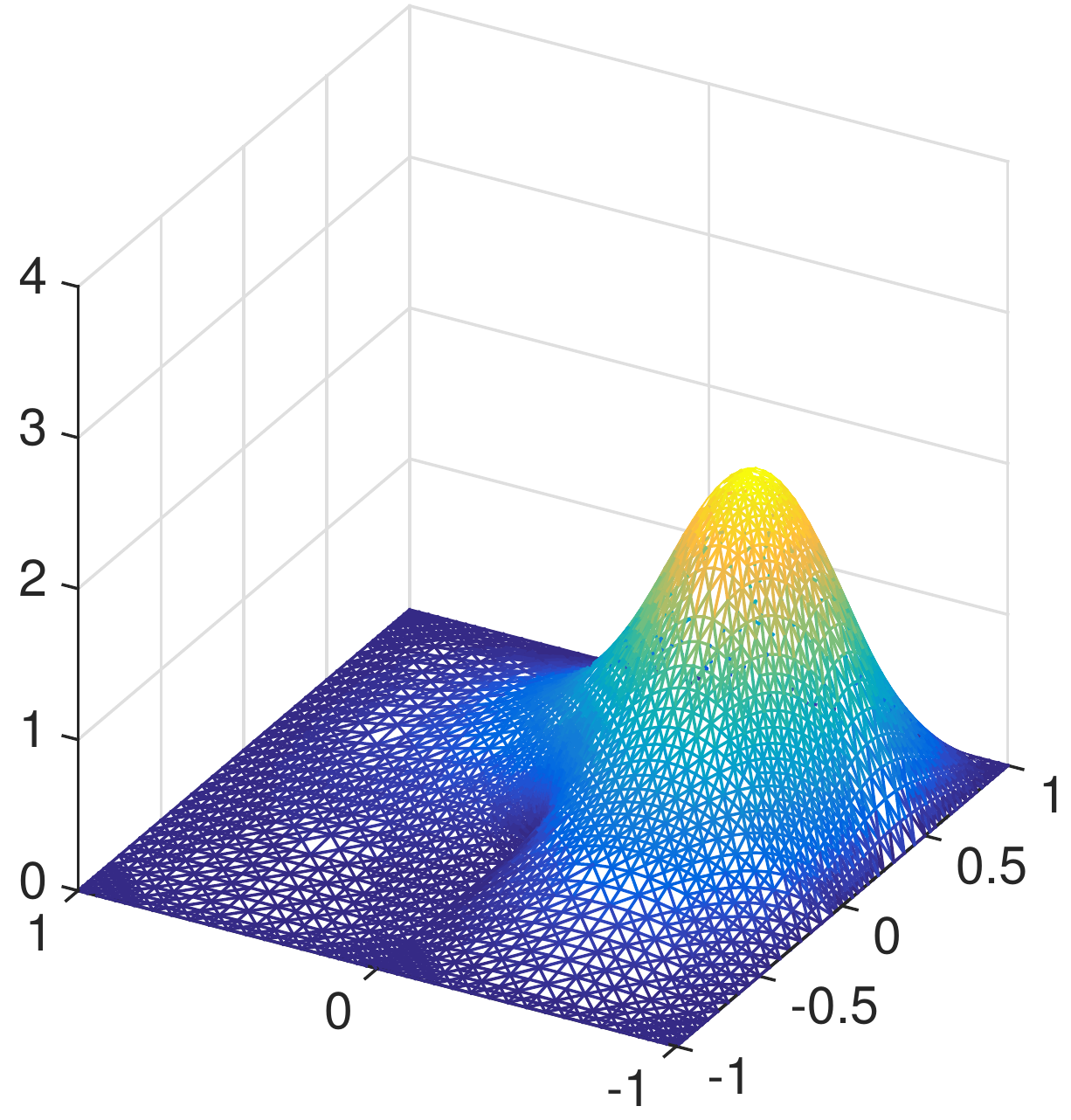} 	\hspace{0.1cm}
	\includegraphics[scale=0.3]{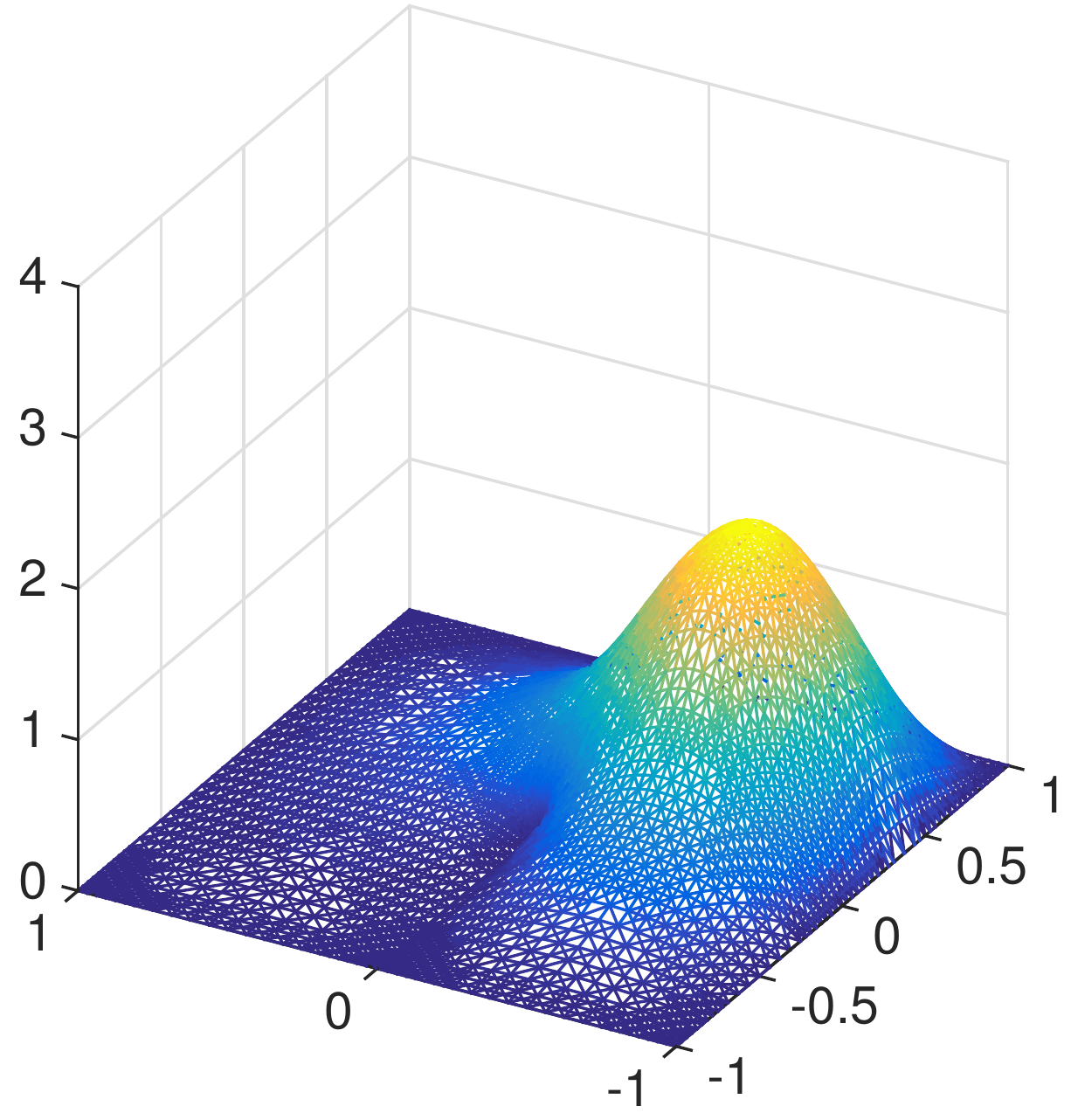} 
\caption{
Adaptively refined triangulations (top row) and the mean fields of dual Galerkin solutions (bottom row)
computed using the mollifier $g_0$ in~\eqref{mollifier} with $r = 0.15, 0.3, 0.4, 0.5$ (Experiment~3).
}
\label{Ex3:sequence:meshes}
\end{figure}

First, we fix $\textsf{tol} = 7$e-$03$ 
and run Algorithm~\ref{alg:parametric} to compute dual Galerkin solutions
for different values of radius $r$ in~\eqref{mollifier}.
Figure~\ref{Ex3:sequence:meshes} shows the refined triangulations (top row) and the corresponding
mean fields of dual Galerkin solutions (bottom row) for $r = 0.15, 0.3, 0.4, 0.5$.
As observed in previous experiments, the triangulations generated by the algorithm simultaneously capture
spatial features of primal and dual solutions.
In this experiment, the triangulations are refined in the vicinity of each corner,
with particularly strong refinement near the origin, where the primal solution exhibits a singularity;
in addition to that, for smaller values of $r$ ($r=0.15,\; 0.3$),
the triangulation is strongly refined in a neighborhood of~$x_0$
due to sharp gradients in the corresponding dual solutions
(note that the refinements in the neighborhood of~$x_0$ become coarser as $r$ increases).

Let us now fix $r=0.15$ (which gives $C\approx 95.271$ in~\eqref{mollifier}) 
and run Algorithm~\ref{alg:parametric} with two sets of marking parameters: 
(i)~$\thetaX=0.3$, $\thetaP=0.8$; (ii)~$\thetaX=0.15$, $\thetaP=0.8$.
In both computations we choose $\overline{M}=1$ in \eqref{finite:index:set:Q} and set
the tolerance to $\textsf{tol} = 6.0$e-$04$.

In Table~\ref{Ex3:table}, we collect the outputs of computations in both cases.
In agreement with results of previous experiments, we see that running the algorithm
with a smaller value of $\thetaX$ (i.e., in case (ii)) requires
more iterations to reach the tolerance (see the values of $L$ in both columns in Table~\ref{Ex3:table}).
We also observe that, for a fixed $\thetaP$, choosing a smaller $\thetaX$
naturally results in a less refined final triangulation
($\#\TT_L=54,819$ in case~(ii) versus $\#\TT_L=67,955$ in case~(i));
interestingly, this under-refinement in the spatial approximation
is balanced by a more accurate polynomial approximation on the parameter domain
(i.e., a larger final index set $\#\gotP_L$ is generated:
$28$ indices with $7$ active parameters in case (ii) versus
$21$ indices with $6$ active parameters in case~(i)).

\begin{table}[t!]
\setlength\tabcolsep{4.9pt} 
\begin{center} 
\smallfontthree{
\renewcommand{\arraystretch}{1.45}
\begin{tabular}{r !{\vrule width 1.0pt} c  c !{\vrule width 1.0pt} c c } 
\noalign{\hrule height 1.0pt}
%
&\multicolumn{2}{c!{\vrule width 1.0pt}}{case (i): $\thetaX=0.3,\ \thetaP=0.8$} 
&\multicolumn{2}{c}{case (ii): $\thetaX=0.15,\ \thetaP=0.8$} \\	
\hline
$L$						&\multicolumn{2}{c!{\vrule width 1.0pt}}{$33$}		
						&\multicolumn{2}{c}{$56$}\\[-3pt]

$\mu_L\zeta_L$			&\multicolumn{2}{c!{\vrule width 1.0pt}}{$5.5663\text{e-}04$}	
						&\multicolumn{2}{c}{5.9495$\text{e-}04$}\\[-3pt]

$t$ (sec)			 	&\multicolumn{2}{c!{\vrule width 1.0pt}}{$434$}	
						&\multicolumn{2}{c}{$659$}\\[-3pt]

$N_{\rm total}$			&\multicolumn{2}{c!{\vrule width 1.0pt}}{$3,006,114$}	
						&\multicolumn{2}{c}{$4,890,073$}\\[-3pt]	
									
$N_L$					&\multicolumn{2}{c!{\vrule width 1.0pt}}{$706,398$}
						&\multicolumn{2}{c}{$759,136$}\\[-3pt]

$\#\TT_L$				&\multicolumn{2}{c!{\vrule width 1.0pt}}{$67,955$}
						&\multicolumn{2}{c}{$54,819$}\\[-3pt]

$\#\NN_L$				&\multicolumn{2}{c!{\vrule width 1.0pt}}{$33,638$}
						&\multicolumn{2}{c}{$27,112$}\\[-3pt]								
					
$\#\gotP_L$ 			&\multicolumn{2}{c!{\vrule width 1.0pt}}{$21$}
						&\multicolumn{2}{c}{$28$}\\[-3pt]
				
$M^\mathrm{active}_L$	&\multicolumn{2}{c!{\vrule width 1.0pt}}{$6$}
						&\multicolumn{2}{c}{$7$}\\

\hline 
\multicolumn{5}{l}{Evolution of the index set} \\
\hline

$\gotP_\ell$	
&$\ell=0$	&$(0\ 0)$				&$\ell=0$	&$(0\ 0)$	\\[-4.5pt]
&			&$(1\ 0)$				&			&$(1\ 0)$	\\
				
&$\ell=11$	&$(0\ 1)$				&$\ell=13$	&$(0\ 1)$\\[-4.5pt]
&			&$(2\ 0)$				&			&$(2\ 0)$\\

&$\ell=17$	&$(0\ 0\ 1)$			&$\ell=25$	&$(0\ 0\ 1)$\\[-4.5pt]
&			&$(1\ 1\ 0)$			&			&$(1\ 1\ 0)$\\[-4.5pt]
&			&$(3\ 0\ 0)$			&			&$(3\ 0\ 0)$\\
	
&$\ell=23$	&$(0\ 0\ 0\ 1)$			&$\ell=36$	&$(0\ 0\ 0\ 1)$\\[-4.5pt]
&			&$(1\ 0\ 1\ 0)$			&			&$(1\ 0\ 1\ 0)$\\[-4.5pt]
&			&$(2\ 1\ 0\ 0)$			&			&$(2\ 1\ 0\ 0)$\\

&$\ell=29$	&$(0\ 0\ 0\ 0\ 1)$		&$\ell=45$	&$(0\ 0\ 0\ 0\ 1)$\\[-4.5pt]
&			&$(1\ 0\ 0\ 1\ 0)$		&			&$(1\ 0\ 0\ 1\ 0)$\\[-4.5pt]
&			&$(2\ 0\ 1\ 0\ 0)$		&			&$(2\ 0\ 1\ 0\ 0)$\\[-4.5pt]
&			&$(4\ 0\ 0\ 0\ 0)$		&			&$(4\ 0\ 0\ 0\ 0)$\\
		
&$\ell=31$	&$(0\ 0\ 0\ 0\ 0\ 1)$	&$\ell=49$	&$(0\ 0\ 0\ 0\ 0\ 1)$\\[-4.5pt]
&			&$(0\ 1\ 1\ 0\ 0\ 0)$	&			&$(0\ 1\ 1\ 0\ 0\ 0)$\\[-4.5pt]
&			&$(0\ 2\ 0\ 0\ 0\ 0)$	&			&$(0\ 2\ 0\ 0\ 0\ 0)$\\[-4.5pt]
&			&$(1\ 0\ 0\ 0\ 1\ 0)$	&			&$(1\ 0\ 0\ 0\ 1\ 0)$\\[-4.5pt]
&			&$(2\ 0\ 0\ 1\ 0\ 0)$	&			&$(2\ 0\ 0\ 1\ 0\ 0)$\\[-4.5pt]
&			&$(3\ 0\ 1\ 0\ 0\ 0)$	&			&$(3\ 0\ 1\ 0\ 0\ 0)$\\[-4.5pt]
&			&$(3\ 1\ 0\ 0\ 0\ 0)$	&			&$(3\ 1\ 0\ 0\ 0\ 0)$\\

&			&						&$\ell=56$	&$(0\ 0\ 0\ 0\ 0\ 0\ 1)$\\[-4.5pt]	
&			&						&			&$(0\ 1\ 0\ 1\ 0\ 0\ 0)$\\[-4.5pt]
&			&						&			&$(1\ 0\ 0\ 0\ 0\ 1\ 0)$\\[-4.5pt]
&			&						&			&$(1\ 1\ 1\ 0\ 0\ 0\ 0)$\\[-4.5pt]
&			&						&			&$(1\ 2\ 0\ 0\ 0\ 0\ 0)$\\[-4.5pt]
&			&						&			&$(4\ 1\ 0\ 0\ 0\ 0\ 0)$\\[-4.5pt]
&			&						&			&$(5\ 0\ 0\ 0\ 0\ 0\ 0)$\\

\noalign{\hrule height 1.0pt}
\end{tabular}
\vspace{10pt}
\caption{
The outputs obtained by running Algorithm~\ref{alg:parametric} in Experiment~3 with 
$\thetaX = 0.3$, $\thetaP = 0.8$ (case~(i)) and $\thetaX = 0.15$, $\thetaP = 0.8$ (case~(ii)).
}
\label{Ex3:table}
}
\end{center}                                                                   
\end{table}

\begin{figure}[t!]
\begin{tikzpicture}
\pgfplotstableread{data/parametric/ex3/thetaX0.30_thetaP0.80_extrarv=1.dat}{\first}
\begin{loglogaxis}
[
title={case~(i)},											
xlabel={degree of freedom, $N_\ell$}, 						
ylabel={error estimate},									
ylabel style={font=\tiny,yshift=-1.0ex}, 					
ymajorgrids=true, xmajorgrids=true, grid style=dashed,		
xmin = 10^(1.5), 	xmax = 10^(6.1),						
ymin = 1*10^(-4),	ymax = 1*10^(0),						
width = 7.8cm, height=7cm,
legend style={legend pos=south west, legend cell align=left, fill=none, draw=none, font={\fontsize{9pt}{12pt}\selectfont}}
]
\addplot[blue,mark=diamond,mark size=3.0pt]	table[x=dofs, y=error_primal]{\first};
\addplot[red,mark=square,mark size=2.5pt]	table[x=dofs, y=error_dual]{\first};
\addplot[darkGreen,mark=o,mark size=2.5pt]	table[x=dofs, y=error_product]{\first};
\addplot[teal,mark=triangle,mark size=3.5pt]table[x=dofs, y=truegerr]{\first};
\addplot[black,solid,domain=10^(1.5):10^(7.8)] { 4.3*x^(-0.35) };
\addplot[black,solid,domain=10^(1.5):10^(7.8)] { 1.25*x^(-2/3) };
\node at (axis cs:8e4,8e-2) [anchor=south west] {$\mathcal{O}(N_\ell^{-0.35})$};
\node at (axis cs:3e4,1e-4) [anchor=south west] {$\mathcal{O}(N_\ell^{-2/3})$};
\legend{
{$\mu_\ell$ (primal)},
{$\zeta_\ell$ (dual)},
{$\mu_\ell\,\zeta_\ell$},
{$|G(u_{\rm ref}) - G(u_\ell)|$}
}
\end{loglogaxis}
\end{tikzpicture}
\begin{tikzpicture}
\pgfplotstableread{data/parametric/ex3/thetaX0.15_thetaP0.80_extrarv=1.dat}{\first}
\begin{loglogaxis}
[
title={case~(ii)},											
xlabel={degree of freedom, $N_\ell$}, 						
ylabel={error estimate},									
ylabel style={font=\tiny,yshift=-1.0ex}, 					
ymajorgrids=true, xmajorgrids=true, grid style=dashed,		
xmin = 10^(1.5), 	xmax = 10^(6.2),						
ymin = 1*10^(-4),	ymax = 1*10^(0),						
width = 7.8cm, height=7.0cm,
legend style={legend pos=south west, legend cell align=left, fill=none, draw=none, font={\fontsize{9pt}{12pt}\selectfont}}
]
\addplot[blue,mark=diamond,mark size=3.0pt]		table[x=dofs, y=error_primal]{\first};
\addplot[red,mark=square,mark size=2.5pt]		table[x=dofs, y=error_dual]{\first};
\addplot[darkGreen,mark=o,mark size=2.5pt]		table[x=dofs, y=error_product]{\first};
\addplot[teal,mark=triangle,mark size=3.5pt]	table[x=dofs, y=truegerr]{\first};
\addplot[black,solid,domain=10^(1.5):10^(7.8)] { 4.3*x^(-0.35) };
\addplot[darkGreen,solid,domain=10^(5):10^(6)] { 260*x^(-0.9) };
\addplot[black,solid,domain=10^(1.5):10^(7.8)] { 1.3*x^(-2/3) };
\node at (axis cs:8e4,8e-2) [anchor=south west] {$\mathcal{O}(N_\ell^{-0.35})$};
\node[darkGreen] at (axis cs:1e5,5e-3) [anchor=south west] {$\mathcal{O}(N_\ell^{-0.9})$};
\node[black] at (axis cs:3e4,1e-4) [anchor=south west] {$\mathcal{O}(N_\ell^{-2/3})$};
\legend{
{$\mu_\ell$ (primal)},
{$\zeta_\ell$ (dual)},
{$\mu_\ell\,\zeta_\ell$},
{$|G(u_{\rm ref}) - G(u_\ell)|$}
}
\end{loglogaxis}
\end{tikzpicture}
\caption{
Error estimates $\mu_\ell$, $\zeta_\ell$, $\mu_\ell\,\zeta_\ell$ and the reference error $|G(u_{\rm ref}) - G(u_\ell)|$ 
at each iteration of Algorithm~\ref{alg:parametric} with $\thetaX=0.3$, $\thetaP=0.8$ (case~(i), left) and 
$\thetaX=0.15$, $\thetaP=0.8$ (case~(ii), right) in Experiment~3 (here, $G(\uref) = 0.144497$e+$01$).
}
\label{Ex3:decay:est}
\end{figure}
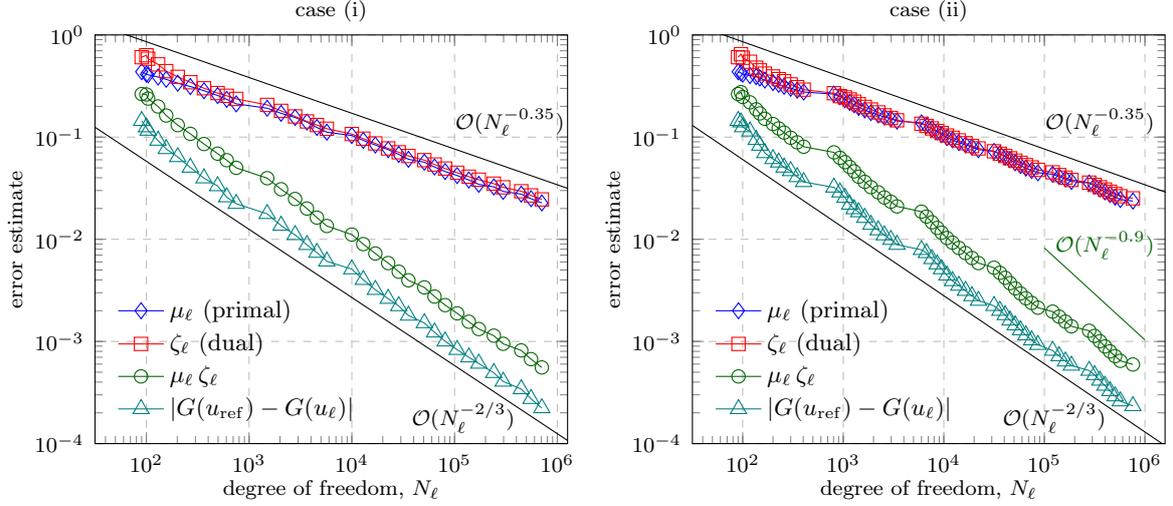

\begin{figure}[t!]
\begin{tikzpicture}
\pgfplotstableread{data/parametric/ex3/thetaX0.30_thetaP0.80_extrarv=1.dat}{\first}
\pgfplotstableread{data/parametric/ex3/thetaX0.15_thetaP0.80_extrarv=1.dat}{\second}
\begin{semilogxaxis}
[
width = 8.0cm, height=6.0cm,						
xlabel={degree of freedom, $N_\ell$}, 				
ylabel={effectivity index, $\Theta_\ell$},			
ylabel style={font=\tiny,yshift=-2.5ex}, 			
ymajorgrids=true, xmajorgrids=true, grid style=dashed,		
xmin = (6)*10^1,	xmax = 10^6,					
ymin = 1.75,		ymax = 2.6,						
ytick={1.8,2,2.3,2.6},								
legend style={legend pos=south east, legend cell align=left, fill=none, draw=none, font={\fontsize{9pt}{12pt}\selectfont}}
]
\addplot[blue,mark=o,mark size=3.0pt]		table[x=dofs, y=effindices]{\first};
\addplot[red,mark=triangle,mark size=3.5pt]	table[x=dofs, y=effindices]{\second};
\legend{
{$\thetaX = 0.3$, $\thetaP = 0.8$},
{$\thetaX = 0.15$, $\thetaP = 0.8$}
}
\end{semilogxaxis}
\end{tikzpicture}
\caption{
The effectivity indices for the goal-oriented error estimates in Experiment~3 
at each iteration of Algorithm~\ref{alg:parametric}.
}
\label{Ex3:goal:effindices}
\end{figure}

By looking now at Figure~\ref{Ex3:decay:est} we observe that the energy error estimates 
$\mu_\ell$ and $\zeta_\ell$ decay with the same rate of about $\mathcal{O}(N^{-0.35})$ for both sets of marking parameters; 
this yields an overall rate of about $\mathcal{O}(N^{-2/3})$ for $\mu_\ell\zeta_\ell$
in both cases.
However, we can see that in case (ii), the estimates $\mu_\ell\zeta_\ell$ decay with a
nearly optimal rate of $\mathcal{O}(N^{-0.9})$ during mesh refinement steps.
This is due to a smaller value of the marking parameter $\thetaX$ in this case
and consistent with what we observed in Experiment~2.

Finally, we compute the effectivity indices $\Theta_\ell$ at each iteration of the algorithm.
Here, we employ a reference Galerkin solution computed
using the triangulation $\TT_{\rm ref}$
(obtained by a uniform refinement of $\TT_L$ produced in case~(i))
and the reference index set $\gotP_{\rm ref} := \gotP_L \cup \gotM_L$,
where $\gotP_L$ and $\gotM_L$ are the index sets generated in case~(ii).
The effectivity indices are plotted in Figure~\ref{Ex3:goal:effindices}. 
This plot shows that $\mu_\ell\zeta_\ell$ provide sufficiently accurate
estimates of the error in approximating~$G(u)$,
as the effectivity indices tend to vary in a range between $2.0$ to $2.6$ for both 
sets of marking parameters.

The results of this experiment show that Algorithm~\ref{alg:parametric}
with appropriate choice of marking parameters generates
effective approximations to the mean of the quantity of interest
associated with point values of the
spatially singular solution to the considered parametric model problem.
In agreement with results of Experiment~2, we conclude that
smaller values of the spatial marking parameter $\theta_X$  (such as $\thetaX=0.15$ as in case~(ii))
are, in general, preferable, as they yield nearly optimal convergence rates
(for the error in the goal functional) during spatial refinement steps.

\section{Concluding remarks} \label{sec:remarks}
\noindent
The design and analysis of effective algorithms for the numerical solution of parametric PDEs is
of fundamental importance when dealing with mathematical models with inherent uncertainties.
In this context, adaptivity is a crucial ingredient to mitigate the so-called
\emph{curse of dimensionality}---a deterioration of convergence rates
and an exponential growth of the computational cost as the dimension of the parameter space increases.

In this paper, we developed a goal-oriented adaptive algorithm for the accurate approximation of
a quantity of interest, which is a linear functional of the solution to a parametric elliptic PDE.
The algorithm is based on an sGFEM discretization of the PDE and is driven by a novel \textsl{a~posteriori}
estimate of the energy errors in Galerkin approximations of the primal and dual solutions.
The proposed error estimate,
which is proved to be efficient and reliable (Theorem~\ref{theorem:twolevel}),
consists of two components:
a two-level estimate accounting for the error in the spatial discretization
and a hierarchical estimate accounting for the error in the parametric discretization.

We highlight two important features of our approach.
On the one hand, using a two-level error estimate,
the algorithm does not require the solution of any linear system for the estimation of the errors
arising from spatial discretizations.
When compared to the hierarchical error estimation in~\cite{bps14,bs16,br18},
this approach leads to an undeniable benefit in terms of the overall computational cost.
On the other hand, the components of the error estimates for primal and dual solutions
are used not only to guide the adaptive enhancement of the discrete space,
but also to assess the error reduction in the product of these estimates
(see Step~{\rm(v-a)} of Algorithm~\ref{alg:parametric}),
which is a reliable estimate for the approximation error in the quantity of interest.
This information about the error reduction is then employed
to choose between spatial refinement and parametric enrichment at each iteration of the algorithm (see Step~{\rm(v-b)}).

While we focused the presentation on the two-dimensional case, the results hold for arbitrary spatial
dimension, i.e., for $D \subset \R^d$ with $d \ge 1$.
Possible extensions of the work include the use of other compatible types of mesh-refinement
(e.g., newest vertex bisection or red refinement, instead of longest edge bisection)
and the treatment of other elliptic operators, different boundary conditions
as well as parameter-dependent right-hand sides $f$ in~\eqref{eq:strongform} and
parameter-dependent functions $g$ in the definition~\eqref{parametric:goal} of the goal functional.
Moreover, the focus of a future publication~\cite{conv2level} will be on the 
mathematical justification of the proposed adaptive algorithm via a rigorous convergence analysis.

We conclude by emphasizing that the software implementing
the proposed goal-oriented adaptive algorithm is available online (see~\cite{BespalovR_stoch_tifiss})
and can be used to reproduce the presented numerical results.

\bibliographystyle{alpha}
\bibliography{ref}
\end{document}